\documentclass{amsart}
\usepackage{amssymb}
\usepackage{amsmath}
\usepackage{amsfonts}
\usepackage{diagrams}

\newtheorem{theorem}{Theorem}[section]

\newtheorem{definition}[theorem]{Definition}
\newtheorem{example}[theorem]{Example}

\newtheorem{lemma}[theorem]{Lemma}

\newtheorem{proposition}[theorem]{Proposition}
\newtheorem{remark}[theorem]{Remark}

\input{tcilatex}

\begin{document}
\title[\emph{UCF} for shape homology]{On the universal coefficients formula
for shape homology}
\author{Andrei V. Prasolov}
\address{Institute of Mathematics and Statistics\\
University of Troms\o , N-9037 Troms\o , Norway}
\email{andrei.prasolov@uit.no}
\urladdr{http://serre.mat-stat.uit.no/ansatte/andrei/Welcome.html}
\date{}

\begin{abstract}
In this paper it is investigated whether various shape homology theories
satisfy the Universal Coefficients Formula (\emph{UCF}). It is proved that
pro-homology and strong homology satisfy \emph{UCF} in the class $\mathbf{FAB%
}$ of finitely generated abelian groups, while they do not satisfy UCF in
the class $\mathbf{AB}$ of all abelian groups. Two new shape homology
theories (called \emph{UCF}-balanced) are constructed. It is proved that
balanced pro-homology satisfies UCF\ in the class $\mathbf{AB}$, while
balanced strong homology satisfies \emph{UCF} only in the class $\mathbf{FAB}
$.
\end{abstract}

\maketitle
\tableofcontents

\setcounter{section}{-1}

\section{Introduction}

The purpose of this paper is to investigate whether various shape homology
theories satisfy the Universal Coefficients Formula (\emph{UCF}). It is
proved (Theorem \ref{Th-UCF-FAB}) that pro-homology and strong homology
satisfy \emph{UCF} in the class $\mathbf{FAB}$ of finitely generated abelian
groups, while they do not satisfy \emph{UCF} in the class $\mathbf{AB}$ of
all abelian groups (Theorems \ref{Th-pro} and \ref{Th-strong}). Two new
shape homology theories (called \emph{UCF-balanced}, or simply \emph{balanced%
}) are constructed. It is proved that \emph{balanced pro-homology} satisfies 
\emph{UCF}\ in the class $\mathbf{AB}$ (Theorem \ref{Th-balanced-pro}). It
happens that our \emph{balanced strong homology} is not \textquotedblleft
balanced enough\textquotedblright . It satisfies \emph{UCF} in the class $%
\mathbf{FAB}$ (Theorem \ref{Th-UCF-FAB}), but not in the class $\mathbf{AB}$
(Theorem \ref{Th-balanced-strong}). The latter theorem is the most difficult
part of the paper. Two counter-examples are constructed. The first one is
simpler, but depends on the Continuum Hypothesis (in fact, on a weaker
assumption \textquotedblleft $d=\aleph _{1}$\textquotedblright ). Hence,
that example cannot be considered as a \textquotedblleft
final\textquotedblright\ counter-example. Another counter-example is much
more complicated, but does not depend on any extra assumption, therefore can
be considered as \textquotedblleft final\textquotedblright .

To deal with \emph{UCF}, one needs to develop the torsion functor $\mathbf{%
Tor}_{\ast }^{k}$ on the category of pro-modules, which is done in Section %
\ref{Sec-pro-modules}. The problem is that the category $\mathbf{Pro}\left(
k\right) $ of pro-modules \textbf{does not} have enough projectives (Remark %
\ref{Rem-not-enough-projectives}). It does have, however, enough \emph{%
quasi-projectives} (Proposition \ref{Prop-enough-quasi-projectives}). It
follows also that quasi-projectives are flat (Proposition \ref%
{Prop-quasi-projectives-are-flat}) provided $k$ is quasi-noetherian. The
above two facts allow defining the torsion functors using quasi-projective
resolutions (Proposition \ref{Prop-quasi-projective-resolution}).

To calculate the strong homology groups (balanced or non-balanced), one
needs spectral sequences from Section \ref{Sec-spectral}. Those spectral
sequences are of \emph{inverse limit type} (in contrast to the \emph{direct
limit type} sequences), and are concentrated in the I and IV quadrants. It
is not very easy to treat the convergence of such sequences. There are
several papers that deal with the convergence of inverse limit type spectral
sequences. See, e.g., the \emph{most general} treatment in \cite%
{Boardman-MR1718076}, Part II, spectral sequences for \emph{towers of
fibrations} and \emph{homotopy limits} in \cite{Bousfield-Kan-MR0365573}, \S %
IX.4 and \S XI.7, spectral sequences for homotopy limits of \emph{spectra}
in \cite{Thomason-MR826102}, 1.16, 5.44-5.48, spectral sequences for $\Gamma 
$\emph{-spaces} in \cite{Prasolov-Extraordinatory-MR1821856}, Theorem 2, and
a spectral sequence for strong homology in \cite{Prasolov-Spectral-MR1027513}%
. None of the approaches above suits 100\% our purposes. That is why in
Section \ref{Sec-spectral}, spectral sequences for homotopy limits of \emph{%
diagrams of chain complexes} are developed (Theorem \ref{Spectral-holimit}).
The most important result in that Section is Theorem \ref%
{Th-Spectral-holimit-pro-complex} where homotopy limits of \emph{%
pro-complexes} are treated.

Let $G\in \mathbf{Mod}\left( \mathbb{Z}\right) $ be an abelian group (see
the notations for $\mathbf{Mod}\left( k\right) $ and $\mathbf{Pro}\left(
k\right) $ from Example \ref{Ex-pro-categories} (\ref{Ex-Pro(k)})).
Throughout this paper, $h_{n}\left( \_,G\right) $ will be one of the
following homology theories:

\begin{enumerate}
\item Pro-homology $\mathbf{H}_{n}\left( X,G\right) \in \mathbf{Pro}\left( 
\mathbb{Z}\right) $\ (see Definition \ref{Def-pro-homology}).

\item Strong homology $\overline{H}_{n}\left( X,G\right) \in \mathbf{Mod}%
\left( \mathbb{Z}\right) $ (see Definition \ref{Def-strong-homology}).

\item Balanced pro-homology $\mathbf{H}_{n}^{b}\left( X,G\right) \in \mathbf{%
Pro}\left( \mathbb{Z}\right) $ (see Definition \ref%
{Def-balanced-pro-homology}).

\item Balanced strong homology $\overline{H}_{n}^{b}\left( X,G\right) \in 
\mathbf{Mod}\left( \mathbb{Z}\right) $ (see Definition \ref%
{Def-balanced-strong-homology}).
\end{enumerate}

Let $\otimes _{\mathbb{Z}}$ denote either the usual tensor product 
\begin{equation*}
\otimes _{\mathbb{Z}}:\mathbf{Mod}\left( \mathbb{Z}\right) \mathbf{\times Mod%
}\left( \mathbb{Z}\right) \rightarrow \mathbf{Mod}\left( \mathbb{Z}\right)
\end{equation*}
or the tensor product 
\begin{equation*}
\otimes _{\mathbb{Z}}:\mathbf{Pro}\left( \mathbb{Z}\right) \times \mathbf{Mod%
}\left( \mathbb{Z}\right) \longrightarrow \mathbf{Pro}\left( \mathbb{Z}%
\right)
\end{equation*}%
from Theorem \ref{Th-tensor-product}.

\begin{proposition}
\label{Prop-pairing}For each of the four theories there exists a natural (on 
$X$ and $G$) pairing%
\begin{equation*}
h_{n}\left( X,\mathbb{Z}\right) \otimes _{\mathbb{Z}}G\longrightarrow
h_{n}\left( X,G\right) .
\end{equation*}
\end{proposition}

\begin{remark}
Since all the four theories are defined on the strong shape category $%
\mathbf{SSh}$, \textquotedblleft natural\textquotedblright\ here means that
for each $n\in \mathbb{Z}$ the pairing above is a morphism of functors%
\begin{equation*}
h_{n}\left( ?,\mathbb{Z}\right) \otimes _{\mathbb{Z}}?\longrightarrow
h_{n}\left( ?,?\right) :\mathbf{SSh}\times \mathbf{Mod}\left( \mathbb{Z}%
\right) \longrightarrow \mathbf{C}
\end{equation*}%
where $\mathbf{C}$ is either $\mathbf{Mod}\left( \mathbb{Z}\right) $ or $%
\mathbf{Pro}\left( \mathbb{Z}\right) $.
\end{remark}

\begin{proof}
See Section \ref{Sec-proof-pairing}.
\end{proof}

Let $\mathfrak{C}$ be either the class $\mathbf{AB}$ or the subclass $%
\mathbf{FAB}\subseteq \mathbf{AB}$.

\begin{definition}
We say that the homology theory $h_{\ast }$ \textbf{satisfies UCF in the
class} $\mathfrak{C}$ iff for any $n\in \mathbb{Z},~X\in \mathbf{TOP}$ and $%
G\in \mathfrak{C}$:

\begin{itemize}
\item The pairing $h_{n}\left( X,\mathbb{Z}\right) \otimes _{\mathbb{Z}%
}G\rightarrow h_{n}\left( X,G\right) $ is a monomorphism.

\item The cokernel%
\begin{equation*}
\mathbf{coker}\left( h_{n}\left( X,\mathbb{Z}\right) \otimes _{\mathbb{Z}%
}G\rightarrow h_{n}\left( X,G\right) \right)
\end{equation*}%
is naturally (on $X$ and $G$) isomorphic to $\mathfrak{Tor}_{1}^{\mathbb{Z}%
}\left( h_{n-1}\left( X,\mathbb{Z}\right) ,G\right) $ where $\mathfrak{Tor}$
is either the usual torsion functor%
\begin{equation*}
Tor_{\mathbb{\ast }}^{\mathbb{Z}}:\mathbf{Mod}\left( \mathbb{Z}\right) 
\mathbf{\times Mod}\left( \mathbb{Z}\right) \rightarrow \mathbf{Mod}\left( 
\mathbb{Z}\right)
\end{equation*}%
or the functor%
\begin{equation*}
\mathbf{Tor}_{\mathbb{\ast }}^{\mathbb{Z}}:\mathbf{Pro}\left( \mathbb{Z}%
\right) \mathbf{\times Mod}\left( \mathbb{Z}\right) \rightarrow \mathbf{Pro}%
\left( \mathbb{Z}\right)
\end{equation*}%
from Definition \ref{Def-Tor-Pro(k)}.
\end{itemize}
\end{definition}

\begin{remark}
Roughly speaking, the theory $h_{\ast }$ satisfies UCF in the class $%
\mathfrak{C}$ iff there are natural (on $X\in \mathbf{SSh}$ and $G\in 
\mathfrak{C}$) exact sequences ($n\in \mathbb{Z}$)%
\begin{equation*}
0\longrightarrow h_{n}\left( X,\mathbb{Z}\right) \otimes _{\mathbb{Z}%
}G\longrightarrow h_{n}\left( X,G\right) \longrightarrow \mathfrak{Tor}_{1}^{%
\mathbb{Z}}\left( h_{n-1}\left( X,\mathbb{Z}\right) ,G\right)
\longrightarrow 0.
\end{equation*}
\end{remark}

\section{Main results}

\begin{theorem}
\label{Th-UCF-FAB}All the four theories $h_{n}$ \textbf{satisfy} UCF in the
class $\mathbf{FAB}$.
\end{theorem}

\begin{proof}
See Section \ref{Sec-proof-UCF-FAB}.
\end{proof}

\begin{theorem}
\label{Th-balanced-pro}The balanced pro-homology $\mathbf{H}_{\ast }^{b}$ 
\textbf{satisfies} UCF in the class $\mathbf{AB}$.
\end{theorem}

\begin{proof}
See Section \ref{Sec-proof-balanced-pro}.
\end{proof}

\begin{theorem}
\label{Th-pro}The pro-homology $\mathbf{H}_{\ast }$ \textbf{does not}
satisfy UCF in the class $\mathbf{AB}$.
\end{theorem}

\begin{proof}
See Section \ref{Sec-proof-pro}.
\end{proof}

\begin{theorem}
\label{Th-strong}The strong homology $\overline{H}_{n}$ \textbf{does not}
satisfy UCF in the class $\mathbf{AB}$.
\end{theorem}

\begin{proof}
See Section \ref{Sec-proof-strong}.
\end{proof}

\begin{theorem}
\label{Th-balanced-strong}The balanced strong homology $\overline{H}_{n}^{b}$
\textbf{does not} satisfy UCF in the class $\mathbf{AB}$.
\end{theorem}

\begin{proof}
See Section \ref{Sec-proof-balanced-strong}.
\end{proof}

\section{\label{Sec-pro-modules}Pro-modules}

\subsection{Pro-category}

We will often refer to Chapter 6 of \cite{Kashiwara-Categories-MR2182076}
where ind-objects are considered. Since the category of pro-objects $\mathbf{%
Pro}\left( k\right) $ is isomorphic to the category $\mathbf{Ind}\left( 
\mathbf{Mod}\left( k\right) ^{op}\right) ^{op}$ where $\mathbf{Ind}\left( 
\mathbf{C}\right) $ is the category of ind-objects, all properties of $%
\mathbf{Pro}\left( k\right) $ are dual to the properties of $\mathbf{Ind}%
\left( \mathbf{Mod}\left( k\right) ^{op}\right) $. We can therefore easily
use the statements that are dual to the corresponding statements in \cite%
{Kashiwara-Categories-MR2182076}.

Pro-objects are represented by \emph{inverse systems}. In \cite%
{Mardesic-Segal-MR676973} and \cite{Mardesic-MR1740831}, inverse systems are
indexed by \emph{filtrant} (called \emph{directed} in \cite%
{Mardesic-Segal-MR676973}, \S I.1.1) ordered \textbf{sets} $I$. We use here
more general inverse systems indexed by cofiltrant index categories, i.e. an
inverse system in $\mathbf{C}$ (Definition \ref{Def-Inv(C)} below) will be a
functor $\mathbf{X}:\mathbf{I}\rightarrow \mathbf{C}$ where $\mathbf{I}$ is
a small \emph{cofiltrant} category (see \cite{Kashiwara-Categories-MR2182076}%
, Definition 3.1.1, for the definition of (co)filtrant categories). Without
loss of generality, let us denote such a functor by $\left( X_{i}\right)
_{i\in \mathbf{I}}$. The so-called 
{Marde{\u{s}}i{\'{c}}\ }%
trick (\cite{Mardesic-Segal-MR676973}, Theorem I.1.4) shows that our
construction (indexing by a category) is equivalent to the 
{Marde{\u{s}}i{\'{c}}}%
-Segal construction (indexing by an ordered set). The difference in
terminology (\textbf{co}filtrant categories instead of filtrant sets) is due
to the fact that our construction uses covariant functors, while the 
{Marde{\u{s}}i{\'{c}}}%
-Segal construction uses contravariant ones.

\begin{definition}
\label{Def-Inv(C)}(compare with \cite{Prasolov-Extraordinatory-MR1821856},
Definition 2.1.3) For a category $\mathbf{C}$, let $\mathbf{Inv}\left( 
\mathbf{C}\right) $ be the following category (of \textbf{inverse systems}).
The objects are functors $\mathbf{X}:\mathbf{I}\rightarrow \mathbf{C}$ where 
$\mathbf{I}$ is a small cofiltrant category. A morphism $\mathbf{X}%
\rightarrow \mathbf{Y}$ where $\mathbf{X}:\mathbf{I}\rightarrow \mathbf{C}$
and $\mathbf{Y}:\mathbf{J}\rightarrow \mathbf{C}$, is a pair $\left( \varphi
,\psi \right) $ where $\varphi :\mathbf{J}\rightarrow \mathbf{I}$ is a
functor and $\psi $ is a morphism of functors $\psi :\mathbf{X}\circ \varphi
\rightarrow \mathbf{Y}$. The morphisms are composed as follows:%
\begin{equation*}
\left( \varphi _{2},\psi _{2}\right) \circ \left( \varphi _{1},\psi
_{1}\right) =\left( \varphi _{1}\circ \varphi _{2},\psi _{2}\circ \psi
_{1}\left( \varphi _{1}\right) \right)
\end{equation*}%
where%
\begin{eqnarray*}
\varphi _{1} &:&\mathbf{J}\longrightarrow \mathbf{I},\varphi _{2}:\mathbf{K}%
\longrightarrow \mathbf{J}, \\
\psi _{1} &:&\mathbf{X}\circ \varphi _{1}\longrightarrow \mathbf{Y,}\psi
_{2}:\mathbf{Y}\circ \varphi _{2}\longrightarrow \mathbf{Z,}
\end{eqnarray*}%
and $\psi _{1}\left( \varphi _{2}\right) :\mathbf{X}\circ \varphi _{1}\circ
\varphi _{2}\rightarrow \mathbf{Y}\circ \varphi _{2}$ is the morphism given
by%
\begin{equation*}
\psi _{1}\left( \varphi _{2}\right) _{k}=\left( \psi _{1}\right) _{\varphi
_{2}\left( k\right) }:X_{\varphi _{1}\left( \varphi _{2}\left( k\right)
\right) }\longrightarrow Y_{\varphi _{2}\left( k\right) }.
\end{equation*}%
The identity morphisms $1_{\mathbf{X}}$ are given by families $1_{\mathbf{X}%
}=\left( 1:\mathbf{I}\rightarrow \mathbf{I},\left( 1_{i}:X_{i}\rightarrow
X_{i}\right) _{i\in \mathbf{I}}\right) $.
\end{definition}

It is easily checked that $\mathbf{Inv}\left( \mathbf{C}\right) $ is indeed
a category.

\begin{definition}
\label{Def-level-morphism}If $\mathbf{I}=\mathbf{J}$ and $\varphi $ is the
identity functor, the morphism $\left( \varphi ,\psi \right) $ is called a 
\textbf{level} morphism.
\end{definition}

\begin{definition}
\label{Pro-category}Let $\mathbf{C}$ be a category. The pro-category $%
\mathbf{Pro}\left( \mathbf{C}\right) $ (see \cite%
{Kashiwara-Categories-MR2182076}, Definition 6.1.1, \cite%
{Mardesic-Segal-MR676973}, Remark I.1.4, or \cite{Artin-Mazur-MR883959},
Appendix) is the category with the same class of objects as $\mathbf{Inv}%
\left( \mathbf{C}\right) $, and the following sets of morphisms between two
objects $\mathbf{X}=\left( X_{i}\right) _{i\in \mathbf{I}}$ and $\mathbf{Y}%
=\left( Y_{j}\right) _{j\in \mathbf{J}}$:%
\begin{equation*}
Hom_{\mathbf{Pro}\left( \mathbf{C}\right) }\left( \mathbf{X},\mathbf{Y}%
\right) =\underleftarrow{\lim }_{j\in \mathbf{J}}~\underrightarrow{\lim }%
_{i\in \mathbf{I}}Hom_{\mathbf{C}}\left( X_{i},Y_{j}\right) .
\end{equation*}
\end{definition}

\begin{remark}
It follows from the definition that a morphism in $\mathbf{Pro}\left( 
\mathbf{C}\right) $ between $\mathbf{X}=\left( X_{i}\right) _{i\in \mathbf{I}%
}$ and $\mathbf{Y}=\left( Y_{j}\right) _{j\in \mathbf{J}}$ can be
represented by a pair $\left( \varphi ,\psi \right) $ where $\varphi
:Ob\left( \mathbf{J}\right) \rightarrow Ob\left( \mathbf{I}\right) $ is a
mapping between the sets of objects, and $\psi =\left( \psi _{j}\right)
:X_{\varphi \left( j\right) }\rightarrow Y_{j}$ is a family of morphisms
satisfying the standard inverse limit relations. A morphism $\left( \varphi
,\psi \right) $ in $\mathbf{Inv}\left( \mathbf{C}\right) $ can be therefore
interpreted as a morphism in $\mathbf{Pro}\left( \mathbf{C}\right) $, and
one has an evident functor $\mathbf{Inv}\left( \mathbf{C}\right) \rightarrow 
\mathbf{Pro}\left( \mathbf{C}\right) $.
\end{remark}

\begin{definition}
A morphism $\left( \varphi ,\psi \right) \in Hom_{\mathbf{Inv}\left( \mathbf{%
C}\right) }\left( \mathbf{X},\mathbf{Y}\right) $ is called \textbf{cofinal},
if $\varphi :\mathbf{J}\rightarrow \mathbf{I}$ is a cofinal functor (\textbf{%
left cofinal} in \cite{Bousfield-Kan-MR0365573}, XI.9.1, see also \cite%
{Artin-Mazur-MR883959}, Appendix (1.5)), $\mathbf{X}\circ \varphi =\mathbf{Y}
$, and%
\begin{equation*}
\psi =1_{\mathbf{Y}}:\mathbf{X}\circ \varphi \longrightarrow \mathbf{X}\circ
\varphi =\mathbf{Y}.
\end{equation*}
\end{definition}

\begin{theorem}
\label{Th-description-Pro(C)}The category $\mathbf{Pro}\left( \mathbf{C}%
\right) $ is equivalent to the category of fractions $\mathbf{Inv}\left( 
\mathbf{C}\right) \left[ \Sigma ^{-1}\right] $ where $\Sigma $ is the class
of cofinal morphisms.
\end{theorem}

\begin{proof}
See \cite{Prasolov-Extraordinatory-MR1821856}, Theorem 2.1.10.
\end{proof}

\begin{example}
\label{Ex-pro-categories}~

\begin{enumerate}
\item \label{Ex-Pro(k)}Let $k$ be an associative commutative ring with unit.
The category of $k$-modules will be denoted $\mathbf{Mod}\left( k\right) $.
The pro-category $\mathbf{Pro}\left( \mathbf{Mod}\left( k\right) \right) $
of pro-modules will be shortly denoted $\mathbf{Pro}\left( k\right) $. Since
the category $\mathbf{AB}$ of abelian groups can be naturally identified
with $\mathbf{Mod}\left( \mathbb{Z}\right) $, the category $\mathbf{Pro}%
\left( \mathbf{AB}\right) $ of abelian pro-groups can be naturally
identified with $\mathbf{Pro}\left( \mathbb{Z}\right) $.

\item \label{Ex-Pro(TOP)}We will consider both the categories $\mathbf{Pro}%
\left( \mathbf{TOP}\right) $ and $\mathbf{Pro}\left( \mathbf{HTOP}\right) $
where $\mathbf{TOP}$ is the category of topological spaces and $\mathbf{HTOP}
$ is the category of homotopy types having topological spaces as objects and
homotopy classes of mappings as morphisms. There are full subcategories $%
\mathbf{Pro}\left( \mathbf{POL}\right) \subseteq \mathbf{Pro}\left( \mathbf{%
TOP}\right) $ and $\mathbf{Pro}\left( \mathbf{HPOL}\right) \subseteq \mathbf{%
Pro}\left( \mathbf{HTOP}\right) $ where $\mathbf{POL}$ is the category of
topological spaces having the homotopy type of a polyhedron.

\item \label{Ex-Pro(CHAIN)}Let $\mathbf{CHAIN}\left( k\right) $ be the
category of chain complexes of $k$-modules. We will consider the category $%
\mathbf{Pro}\left( \mathbf{CHAIN}\left( k\right) \right) $ of chain
pro-complexes.
\end{enumerate}
\end{example}

It is well-known that the category $\mathbf{Pro}\left( \mathbf{C}\right) $
admits coproducts if $\mathbf{C}$ does (see, e.g., \cite%
{Prasolov-Extraordinatory-MR1821856}, Proposition 2.4.1). Below is the
explicit description of a coproduct.

\begin{example}
\label{Ex-coproduct}Let $\left( \mathbf{X}^{\alpha }=\left( X_{i}^{\alpha
}\right) _{i\in \mathbf{I}^{\alpha }}\right) _{\alpha \in A}$ be a family of
pro-objects. Define a pro-object $\mathbf{Y}=\left( Y_{j}\right) _{j\in 
\mathbf{J}}$ as follows:%
\begin{equation*}
\mathbf{J}=\dprod\limits_{\alpha \in A}\mathbf{I}^{\alpha
},Y_{j}=\dcoprod\limits_{\alpha \in A}X_{j\left( \alpha \right) }^{\alpha },
\end{equation*}%
where $j=\left( j\left( \alpha \right) \in \mathbf{I}^{\alpha }\right)
_{\alpha \in A}$. One has for another pro-object $\mathbf{Z}=\left(
Z_{k}\right) _{k\in \mathbf{K}}$%
\begin{eqnarray*}
&&Hom_{\mathbf{Pro}\left( \mathbf{C}\right) }\left( \mathbf{Y},\mathbf{Z}%
\right) 
\simeq%
\underleftarrow{\lim }_{k}\underrightarrow{\lim }_{j}Hom_{\mathbf{C}}\left(
Y_{j},Z_{k}\right) 
\simeq%
\underleftarrow{\lim }_{k}\underrightarrow{\lim }_{j}Hom_{\mathbf{C}}\left(
\dcoprod\limits_{\alpha \in A}X_{j\left( \alpha \right) }^{\alpha
},Z_{k}\right) 
\simeq
\\
&&%
\simeq%
\underleftarrow{\lim }_{k}\underrightarrow{\lim }_{j}\dprod\limits_{\alpha
\in A}Hom_{\mathbf{C}}\left( X_{j\left( \alpha \right) }^{\alpha
},Z_{k}\right) 
\simeq%
\underleftarrow{\lim }_{k}\dprod\limits_{\alpha \in A}\underrightarrow{\lim }%
_{i\in \mathbf{I}^{\alpha }}Hom_{\mathbf{C}}\left( X_{i}^{\alpha
},Z_{k}\right) 
\simeq
\\
&&%
\simeq%
\underleftarrow{\lim }_{k}\dprod\limits_{\alpha \in A}Hom_{\mathbf{Pro}%
\left( \mathbf{C}\right) }\left( \mathbf{X}^{\alpha },Z_{k}\right) 
\simeq%
\dprod\limits_{\alpha \in A}\underleftarrow{\lim }_{k}Hom_{\mathbf{Pro}%
\left( \mathbf{C}\right) }\left( \mathbf{X}^{\alpha },Z_{k}\right) 
\simeq%
\dprod\limits_{\alpha \in A}Hom_{\mathbf{Pro}\left( \mathbf{C}\right)
}\left( \mathbf{X}^{\alpha },\mathbf{Z}\right) ,
\end{eqnarray*}%
and $\mathbf{Y}$ is indeed a coproduct of $\left( \mathbf{X}^{\alpha
}\right) _{\alpha \in A}$.
\end{example}

\begin{definition}
\label{Def-pro-complex-pro}Given a pro-complex 
\begin{equation*}
\mathbf{C}_{\ast }=\left( C_{\ast i}\right) _{i\in \mathbf{I}}\in \mathbf{Pro%
}\left( \mathbf{CHAIN}\left( k\right) \right)
\end{equation*}%
let $\mathbf{C}_{\ast }^{\#}\in \mathbf{CHAIN}\left( \mathbf{Pro}\left(
k\right) \right) $ be the following complex of pro-modules: $\mathbf{C}%
_{n}^{\#}=\left( C_{n,i}\right) _{i\in \mathbf{I}}$, $d_{n}=\left(
d_{n,i}\right) _{i\in \mathbf{I}}$. There are two ways of defining the
pro-homologies of the pro-complex:%
\begin{equation*}
\mathbf{H}_{n}\left( \mathbf{C}_{\ast }\right) =\left( \frac{\ker \left(
d_{n,i}:C_{n,i}\longrightarrow C_{n-1,i}\right) }{Im\left(
d_{n+1,i}:C_{n+1,i}\longrightarrow C_{n,i}\right) }\right) _{i\in \mathbf{I}}
\end{equation*}%
and%
\begin{equation*}
\mathbf{H}_{n}^{\#}\left( \mathbf{C}_{\ast }\right) =\frac{\ker \left( d_{n}:%
\mathbf{C}_{n}\longrightarrow \mathbf{C}_{n-1}\right) }{Im\left( d_{n+1}:%
\mathbf{C}_{n+1}\longrightarrow \mathbf{C}_{n}\right) }.
\end{equation*}
\end{definition}

\begin{proposition}
\label{Prop-pro-complex-pro}$\mathbf{H}_{n}\left( \mathbf{C}_{\ast }\right) 
\simeq%
\mathbf{H}_{n}^{\#}\left( \mathbf{C}_{\ast }\right) $.
\end{proposition}

\begin{proof}
Straightforward. Compare with \cite{Mardesic-Segal-MR676973}, \S II.2. We
need to describe kernels, cokernels, images and coimages.

\begin{enumerate}
\item \textbf{Kernels}. Let%
\begin{equation*}
\left( f:\mathbf{A}\longrightarrow \mathbf{B}\right) =\left(
f_{i}:A_{i}\longrightarrow B_{i}\right) _{i\in \mathbf{I}}
\end{equation*}%
be a level morphism, and let%
\begin{equation*}
\mathbf{C=}\left( \ker \left( f_{i}\right) \right) _{i\in \mathbf{I}}.
\end{equation*}%
We have to prove that $\mathbf{C}$ is the kernel of $f$. Let $\mathbf{D}%
=\left( D_{j}\right) _{j\in \mathbf{J}}\in \mathbf{Pro}\left( k\right) $.
Then%
\begin{equation*}
Hom_{\mathbf{Pro}\left( k\right) }\left( \mathbf{D},\mathbf{C}\right) =%
\underleftarrow{\lim }_{i\in \mathbf{I}}\underrightarrow{\lim }_{j\in 
\mathbf{J}}Hom\left( D_{j},C_{i}\right) =
\end{equation*}%
\begin{eqnarray*}
\underleftarrow{\lim }_{i\in \mathbf{I}}\underrightarrow{\lim }_{j\in 
\mathbf{J}}\ker \left( Hom\left( D_{j},A_{i}\right) \longrightarrow
Hom\left( D_{j},B_{i}\right) \right) &=& \\
\underleftarrow{\lim }_{i\in \mathbf{I}}\ker \left( \underrightarrow{\lim }%
_{j\in \mathbf{J}}Hom\left( D_{j},A_{i}\right) \longrightarrow 
\underrightarrow{\lim }_{j\in \mathbf{J}}Hom\left( D_{j},B_{i}\right)
\right) &=& \\
\ker \left( \underleftarrow{\lim }_{i\in \mathbf{I}}\underrightarrow{\lim }%
_{j\in \mathbf{J}}Hom\left( D_{j},A_{i}\right) \longrightarrow 
\underleftarrow{\lim }_{i\in \mathbf{I}}\underrightarrow{\lim }_{j\in 
\mathbf{J}}Hom\left( D_{j},B_{i}\right) \right) &=&
\end{eqnarray*}%
\begin{equation*}
\ker \left( Hom_{\mathbf{Pro}\left( k\right) }\left( \mathbf{D},\mathbf{A}%
\right) \longrightarrow Hom_{\mathbf{Pro}\left( k\right) }\left( \mathbf{D},%
\mathbf{B}\right) \right) .
\end{equation*}

\item \textbf{Cokernels}. Let now%
\begin{equation*}
\mathbf{E=}\left( \mathbf{coker}\left( f_{i}\right) \right) _{i\in \mathbf{I}%
}.
\end{equation*}%
We have to prove that $\mathbf{E}$ is the cokernel of $f$. Let $\mathbf{D}%
=\left( D_{j}\right) _{j\in \mathbf{J}}\in \mathbf{Pro}\left( k\right) $.
Then%
\begin{equation*}
Hom_{\mathbf{Pro}\left( k\right) }\left( \mathbf{E},\mathbf{D}\right) =%
\underleftarrow{\lim }_{j\in \mathbf{J}}\underrightarrow{\lim }_{i\in 
\mathbf{I}}Hom\left( E_{i},D_{j}\right) =
\end{equation*}%
\begin{eqnarray*}
\underleftarrow{\lim }_{j\in \mathbf{J}}\underrightarrow{\lim }_{i\in 
\mathbf{I}}\ker \left( Hom\left( B_{i},D_{j}\right) \longrightarrow
Hom\left( A_{i},D_{j}\right) \right) &=& \\
\underleftarrow{\lim }_{j\in \mathbf{J}}\ker \left( \underrightarrow{\lim }%
_{i\in \mathbf{I}}Hom\left( B_{i},D_{j}\right) \longrightarrow 
\underrightarrow{\lim }_{i\in \mathbf{I}}Hom\left( A_{i},D_{j}\right)
\right) &=& \\
\ker \left( \underleftarrow{\lim }_{j\in \mathbf{J}}\underrightarrow{\lim }%
_{i\in \mathbf{I}}Hom\left( B_{i},D_{j}\right) \longrightarrow 
\underleftarrow{\lim }_{j\in \mathbf{J}}\underrightarrow{\lim }_{i\in 
\mathbf{I}}Hom\left( A_{i},D_{j}\right) \right) &=&
\end{eqnarray*}%
\begin{equation*}
\ker \left( Hom_{\mathbf{Pro}\left( k\right) }\left( \mathbf{B},\mathbf{D}%
\right) \longrightarrow Hom_{\mathbf{Pro}\left( k\right) }\left( \mathbf{A},%
\mathbf{D}\right) \right) .
\end{equation*}

\item \textbf{Images}. 
\begin{equation*}
Im\left( f\right) =\ker \mathbf{coker}\left( f\right) =\ker \left(
B_{i}\longrightarrow \mathbf{coker}\left( f_{i}\right) \right) _{i\in 
\mathbf{I}}=\left( Im\left( f_{i}\right) \right) _{i\in \mathbf{I}}.
\end{equation*}

\item \textbf{Coimages}. 
\begin{equation*}
Coim\left( f\right) =\mathbf{coker}\ker \left( f\right) =\mathbf{coker}%
\left( \ker \left( f_{i}\right) \longrightarrow A_{i}\right) _{i\in \mathbf{I%
}}=\left( A_{i}/\ker \left( f_{i}\right) \right) _{i\in \mathbf{I}}%
\simeq%
\left( Im\left( f_{i}\right) \right) _{i\in \mathbf{I}}.
\end{equation*}
\end{enumerate}

Actually, $Im\left( f\right) 
\simeq%
Coim\left( f\right) $ because $\mathbf{Pro}\left( k\right) $ is an abelian
category (see \cite{Kashiwara-Categories-MR2182076}, dual to Theorem 8.6.5).
Finally, 
\begin{equation*}
\mathbf{H}_{n}^{\#}\left( \mathbf{C}\right) =\mathbf{coker}\left( \mathbf{C}%
_{n+1}\longrightarrow \ker \left( \mathbf{C}_{n}\longrightarrow \mathbf{C}%
_{n-1}\right) \right) 
\simeq%
\left( \frac{\ker d_{n.i}}{Im\left( d_{n+1.i}\right) }\right) _{i\in \mathbf{%
I}}=\mathbf{H}_{n}\left( \mathbf{C}\right) .
\end{equation*}
\end{proof}

\begin{proposition}
\label{Prop-coproduct-complexes}Let $\left( \mathbf{C}_{\alpha }\in \mathbf{%
Pro}\left( \mathbf{CHAIN}\left( k\right) \right) \right) _{\alpha \in A}$ be
a family of chain pro-complexes of $k$-modules. Then, for any $n\in \mathbb{Z%
}$,%
\begin{equation*}
\mathbf{H}_{n}\left( \dcoprod\limits_{\alpha \in A}\mathbf{C}_{\alpha
}\right) 
\simeq%
\dcoprod\limits_{\alpha \in A}\mathbf{H}_{n}\left( \mathbf{C}_{\alpha
}\right) .
\end{equation*}
\end{proposition}

\begin{proof}
Straightforward, due to the explicit description in Example \ref%
{Ex-coproduct}.
\end{proof}

\begin{remark}
Since both pro-complexes and pro-modules form additive (even abelian)
categories, one can use the symbol $\oplus $ (direct sum) instead of $\sqcup 
$:%
\begin{equation*}
\mathbf{H}_{n}\left( \dbigoplus\limits_{\alpha \in A}\mathbf{C}_{\alpha
}\right) 
\simeq%
\dbigoplus\limits_{\alpha \in A}\mathbf{H}_{n}\left( \mathbf{C}_{\alpha
}\right) .
\end{equation*}
\end{remark}

\subsection{Strong shape}

Theorem \ref{Th-description-SSh} below gives a definition of the strong
shape category which is equivalent to the Lisitsa-%
{Marde{\u{s}}i{\'{c}}}
one (\cite{Mardesic-MR1740831}, \S 8.2).

\begin{definition}
A level morphism $f:\mathbf{X}\rightarrow \mathbf{Y}$ in $\mathbf{Pro}\left( 
\mathbf{TOP}\right) $ is called a \textbf{level equivalence} iff $%
f_{i}:X_{i}\rightarrow Y_{i}$ are homotopy equivalences for all $i\in 
\mathbf{I}$.
\end{definition}

\begin{definition}
(\cite{Prasolov-Extraordinatory-MR1821856}, Definition 2.1.9). A morphism 
\begin{equation*}
\left( \varphi ,\psi \right) :\mathbf{X}\longrightarrow \mathbf{Y}
\end{equation*}%
in $\mathbf{Inv}\left( \mathbf{TOP}\right) $ is called \textbf{special} iff $%
\varphi $ is a cofinal functor and $\psi _{j}:X_{\varphi \left( j\right)
}\rightarrow Y_{j}$ is a homotopy equivalence for all $j\in \mathbf{J}$.
\end{definition}

\begin{proposition}
\label{Prop-special-morphisms}Any special morphism $f:\mathbf{X}\rightarrow 
\mathbf{Y}$ is a composition $f=h\circ g$ of a cofinal morphism $g$ and a
level equivalence $h$.
\end{proposition}

\begin{proof}
Given a special morphism $\left( \varphi ,\psi \right) :\mathbf{X}%
\rightarrow \mathbf{Y}$, let 
\begin{equation*}
\mathbf{Z}=\left( Z_{j}=X_{\varphi \left( j\right) }\right) _{j\in \mathbf{J}%
}\in \mathbf{Inv}\left( \mathbf{TOP}\right)
\end{equation*}%
and%
\begin{eqnarray*}
g &=&\left( \varphi ,1_{\mathbf{X}\circ \varphi }\right) :\mathbf{X}%
\longrightarrow \mathbf{Z}, \\
h &=&\left( 1_{J},\psi \right) :\mathbf{Z}\longrightarrow \mathbf{Y}.
\end{eqnarray*}%
Clearly $g$ is cofinal, and $h$ is a level equivalence, while $f=h\circ g$.
\end{proof}

\begin{theorem}
\label{Th-description-SSh}The strong shape category $\mathbf{SSh}$ is a full
subcategory of the category of fractions $\mathbf{Pro}\left( \mathbf{POL}%
\right) \left[ \Sigma ^{-1}\right] $ where $\Sigma $ is the class of special
morphisms.
\end{theorem}

\begin{proof}
Due to \cite{Prasolov-Extraordinatory-MR1821856}, Theorem 1, the category $%
\mathbf{SSh}$, being a full subcategory of $\mathbf{SSh}\left( \mathbf{Pro}%
\left( \mathbf{TOP}\right) \right) $, is therefore a full subcategory of $%
\mathbf{Pro}\left( \mathbf{ANR}\right) \left[ \Sigma ^{-1}\right] $. It is
easy (and standard in strong shape theory) to substitute the category $%
\mathbf{Pro}\left( \mathbf{ANR}\right) \left[ \Sigma ^{-1}\right] $ by the
alternative category $\mathbf{Pro}\left( \mathbf{POL}\right) \left[ \Sigma
^{-1}\right] $ because any $ANR$ is homotopy equivalent to a polyhedron.
\end{proof}

\subsection{Tensor product}

Let $\mathbf{P}\in \mathbf{Pro}\left( k\right) $, $M\in \mathbf{Mod}\left(
k\right) $. Consider the functor%
\begin{eqnarray*}
F_{\mathbf{P},M} &:&\mathbf{Pro}\left( k\right) \longrightarrow \mathbf{Mod}%
\left( k\right) \mathbf{,} \\
F_{\mathbf{P},M}\left( \mathbf{N}\right) &=&Hom_{\mathbf{Mod}\left( k\right)
}\left( M,Hom_{\mathbf{Pro}\left( k\right) }\left( \mathbf{P},\mathbf{N}%
\right) \right) .
\end{eqnarray*}

\begin{theorem}
\label{Th-tensor-product}~

\begin{enumerate}
\item \label{Th-tensor-product-representable}The functor $F_{\mathbf{P},M}$
is representable. It means that there exists a pro-module (unique up to an
isomorphism) $\mathbf{P}\otimes _{k}M$ such that%
\begin{equation*}
Hom_{\mathbf{Pro}\left( k\right) }\left( \mathbf{P}\otimes _{k}M,\mathbf{N}%
\right) 
\simeq%
Hom_{\mathbf{Mod}\left( k\right) }\left( M,Hom_{\mathbf{Pro}\left( k\right)
}\left( \mathbf{P},\mathbf{N}\right) \right)
\end{equation*}%
naturally on $\mathbf{N}\in \mathbf{Pro}\left( k\right) $.

\item \label{Th-tensor-product-functor}Moreover, since the mapping $\left( 
\mathbf{P},M\right) \longmapsto F_{\mathbf{P},M}$ is functorial, then the
mapping $\left( \mathbf{P},M\right) \longmapsto \mathbf{P}\otimes _{k}M$ is
in fact an bi-additive (even $k$-bilinear) functor%
\begin{equation*}
\otimes _{k}:\mathbf{Pro}\left( k\right) \times \mathbf{Mod}\left( k\right)
\longrightarrow \mathbf{Pro}\left( k\right) .
\end{equation*}

\item \label{Th-tensor-product- right-exact}The functor $\otimes _{k}$ is
right exact with respect to both variables.

\item \label{Th-tensor-product-exact}If $M$ is projective, then the functor $%
?\otimes _{k}M:\mathbf{Pro}\left( k\right) \rightarrow \mathbf{Pro}\left(
k\right) $ is exact.
\end{enumerate}
\end{theorem}

\begin{proof}
\textbf{Representability}. Compare Theorem 2.1.8 in \cite{Sugiki-2001-33} or
Proposition 2.1 in \cite{Schneiders-MR885939}.

Assume $M=k^{X}$ is a free $k$-module generated by the set $X$, and let 
\begin{equation*}
\mathbf{P}\otimes _{k}M:=\dbigoplus\limits_{X}\mathbf{P}
\end{equation*}%
be the direct sum in $\mathbf{Pro}\left( k\right) $. Then 
\begin{eqnarray*}
&&Hom_{\mathbf{Pro}\left( k\right) }\left( \mathbf{P}\otimes _{k}M,\mathbf{N}%
\right) 
\simeq%
Hom_{\mathbf{Pro}\left( k\right) }\left( \dbigoplus\limits_{X}\mathbf{P},%
\mathbf{N}\right) 
\simeq
\\
&&\dprod\limits_{X}Hom_{\mathbf{Pro}\left( k\right) }\left( \mathbf{P},%
\mathbf{N}\right) 
\simeq%
Hom_{\mathbf{Mod}\left( k\right) }\left( M,Hom_{\mathbf{Pro}\left( k\right)
}\left( \mathbf{P},\mathbf{N}\right) \right)
\end{eqnarray*}%
as desired. Consider now a general $k$-module $M$. It can be represented as%
\begin{equation*}
M=\mathbf{coker}\left( k^{Y}\overset{f}{\longrightarrow }k^{X}\right)
\end{equation*}%
where $f$ is given by a (infinite in general) $X\times Y$ matrix with
coefficients in $k$. Define%
\begin{equation*}
\mathbf{P}\otimes _{k}M=\mathbf{coker}\left( \dbigoplus\limits_{Y}\mathbf{P}%
\overset{F}{\longrightarrow }\dbigoplus\limits_{X}\mathbf{P}\right)
\end{equation*}%
where $F$ is given by the same $X\times Y$ matrix. Then%
\begin{eqnarray*}
&&Hom_{\mathbf{Pro}\left( k\right) }\left( \mathbf{P}\otimes _{k}M,\mathbf{N}%
\right) 
\simeq
\\
&&\ker \left( Hom_{\mathbf{Pro}\left( k\right) }\left( \dbigoplus\limits_{X}%
\mathbf{P},\mathbf{N}\right) \longrightarrow Hom_{\mathbf{Pro}\left(
k\right) }\left( \dbigoplus\limits_{Y}\mathbf{P},\mathbf{N}\right) \right) 
\simeq
\\
&&\ker \left( Hom_{\mathbf{Mod}\left( k\right) }\left( k^{X},Hom_{\mathbf{Pro%
}\left( k\right) }\left( \mathbf{P},\mathbf{N}\right) \right)
\longrightarrow Hom_{\mathbf{Mod}\left( k\right) }\left( k^{Y},Hom_{\mathbf{%
Pro}\left( k\right) }\left( \mathbf{P},\mathbf{N}\right) \right) \right) 
\simeq
\\
&&Hom_{\mathbf{Mod}\left( k\right) }\left( \mathbf{coker}\left(
k^{Y}\rightarrow k^{X}\right) ,Hom_{\mathbf{Pro}\left( k\right) }\left( 
\mathbf{P},\mathbf{N}\right) \right) 
\simeq
\\
&&%
\simeq%
Hom_{\mathbf{Mod}\left( k\right) }\left( M,Hom_{\mathbf{Pro}\left( k\right)
}\left( \mathbf{P},\mathbf{N}\right) \right)
\end{eqnarray*}%
as desired.

\textbf{Exactness}. We will prove even more: $\otimes _{k}$ commutes with
arbitrary direct limits (with respect to both variables). Indeed, let $%
\mathbf{P}\in \mathbf{Pro}\left( k\right) $. The functor 
\begin{equation*}
\mathbf{P}\otimes _{k}?:\mathbf{Mod}\left( k\right) \longrightarrow \mathbf{%
Pro}\left( k\right)
\end{equation*}%
is left adjoint to the functor 
\begin{equation*}
Hom_{\mathbf{Pro}\left( k\right) }\left( \mathbf{P},?\right) :\mathbf{Pro}%
\left( k\right) \longrightarrow \mathbf{Mod}\left( k\right) .
\end{equation*}%
Hence, it commutes with direct limits, and therefore is right exact.

Let now $M\in \mathbf{Mod}\left( k\right) $, and let $\mathbf{P}=%
\underrightarrow{\lim }_{i\in \mathbf{I}}~\mathbf{P}_{i}$ be the direct
limit of a diagram (not necessarily filtrant!) of pro-modules. For any
pro-module $\mathbf{N}$,%
\begin{eqnarray*}
&&Hom_{\mathbf{Pro}\left( k\right) }\left( \left( \underrightarrow{\lim }%
_{i\in \mathbf{I}}~\mathbf{P}_{i}\right) \otimes _{k}M,\mathbf{N}\right) 
\simeq%
Hom_{\mathbf{Mod}\left( k\right) }\left( M,Hom_{\mathbf{Pro}\left( k\right)
}\left( \left( \underrightarrow{\lim }_{i\in \mathbf{I}}~\mathbf{P}%
_{i}\right) ,\mathbf{N}\right) \right) 
\simeq
\\
&&Hom_{\mathbf{Mod}\left( k\right) }\left( M,\underleftarrow{\lim }_{i\in 
\mathbf{I}}~Hom_{\mathbf{Pro}\left( k\right) }\left( \mathbf{P}_{i},\mathbf{N%
}\right) \right) 
\simeq
\\
&&\underleftarrow{\lim }_{i\in \mathbf{I}}~Hom_{\mathbf{Mod}\left( k\right)
}\left( M,Hom_{\mathbf{Pro}\left( k\right) }\left( \mathbf{P}_{i},\mathbf{N}%
\right) \right) 
\simeq
\\
&&\underleftarrow{\lim }_{i\in \mathbf{I}}~Hom_{\mathbf{Pro}\left( k\right)
}\left( \mathbf{P}_{i}\otimes _{k}M,\mathbf{N}\right) .
\end{eqnarray*}%
Therefore, $\left( \underrightarrow{\lim }_{i\in \mathbf{I}}~\mathbf{P}%
_{i}\right) \otimes _{k}M%
\simeq%
\underrightarrow{\lim }_{i\in \mathbf{I}}~\left( \mathbf{P}_{i}\otimes
_{k}M\right) $, and $?\otimes _{k}M:\mathbf{Pro}\left( k\right) \rightarrow 
\mathbf{Pro}\left( k\right) $ is right exact as well.

Finally, let $M$ be a free $k$-module, $M%
\simeq%
k^{X}$ for some set $X$. Since $\otimes _{k}$ commutes with direct limits,
the functor $?\otimes _{k}M$ is the direct sum of the identity functors $%
\otimes _{k}k$. Hence, $?\otimes _{k}M$ is exact as a direct sum of exact
functors. If $M$ is projective, then it is a retract of a free module $F$.
Therefore, $?\otimes _{k}M$, being a retract of the exact functor $?\otimes
_{k}F$, is exact as well.
\end{proof}

\begin{remark}
\label{Rem-tensor-product-complexes}Let $\mathbf{P}_{\ast }\in \mathbf{Pro}%
\left( \mathbf{CHAIN}\left( k\right) \right) $ be a pro-complex of $k$%
-modules, and let $M$ be a $k$-module. Then the tensor product $\mathbf{P}%
_{\ast }\otimes _{k}M$ can be represented by a complex of inverse systems
and level differentials indexed by the \textbf{same} index category $J$:%
\begin{eqnarray*}
\mathbf{P}_{n}\otimes _{k}M &=&\left( \left( Q_{\ast }\right) _{j}\right)
_{j\in \mathbf{J}}, \\
\left( d_{n}:\mathbf{P}_{n}\longrightarrow \mathbf{P}_{n-1}\right) &=&\left(
d_{n}:\left( Q_{n}\right) _{j}\longrightarrow \left( Q_{n-1}\right)
_{j}\right) _{j\in \mathbf{J}}.
\end{eqnarray*}%
Indeed, the complex of pro-modules 
\begin{equation*}
\mathbf{Q}_{n}=\mathbf{P}_{n}\otimes _{k}M=\mathbf{coker}\left(
\dbigoplus\limits_{Y}\mathbf{P}_{n}\overset{F}{\longrightarrow }%
\dbigoplus\limits_{X}\mathbf{P}_{n}\right) ,
\end{equation*}%
as in the proof above, can be represented by the following inverse system.
Let $\mathbf{J}$ be the product of the $X\times Y$ copies of $\mathbf{I}$: $%
\mathbf{J}=\mathbf{I}^{X\times Y}$. Given $\mathbf{J}\ni j=\left(
i_{x,y}\right) _{x\in X,y\in Y}$, let%
\begin{eqnarray*}
\left( Q_{n}\right) _{j} &=&\mathbf{coker}\left( \dbigoplus\limits_{y\in
Y}\left( P_{n}\right) _{j\left( x,y\right) }\overset{F_{j}}{\longrightarrow }%
\dbigoplus\limits_{x\in X}\left( P_{n}\right) _{j\left( x,y\right) }\right) ,
\\
\left( d_{n}:Q_{n}\longrightarrow Q_{n-1}\right) _{j} &=&\mathbf{coker}%
\left( \dbigoplus\limits_{y\in Y}\left( d_{n}\right) _{j\left( x,y\right) }%
\overset{F_{j}}{\longrightarrow }\dbigoplus\limits_{x\in X}\left(
d_{n}\right) _{j\left( x,y\right) }\right) ,
\end{eqnarray*}%
where $F_{j}$ is the following constant (not depending on $j$) $X\times Y$
matrix: $\left( F_{j}\right) _{x,y}=f_{x,y}$.
\end{remark}

\begin{definition}
There is another (weak) tensor product%
\begin{equation*}
\widetilde{\otimes }_{k}:\left( \mathbf{Pro}\left( k\right) \right) \times 
\mathbf{Mod}\left( k\right) \longrightarrow \mathbf{Pro}\left( k\right)
\end{equation*}%
which is defined as follows:%
\begin{equation*}
\mathbf{P}\widetilde{\otimes }_{k}M:=\left( P_{i}\otimes _{k}M\right) _{i\in 
\mathbf{I}}
\end{equation*}%
where the pro-module $\mathbf{P}$ is defined by an inverse system $\left(
P_{i}\right) _{i\in \mathbf{I}}$.
\end{definition}

\begin{theorem}
\label{Th-map-to-weak-products}There is a natural homomorphism $\otimes
_{k}\rightarrow \widetilde{\otimes }_{k}$. If $M$ is \textbf{finitely
generated}, the homomorphism above becomes an \textbf{epimorphism} $\mathbf{P%
}\otimes _{k}M\rightarrow \mathbf{P}\widetilde{\otimes }_{k}M$. If $M$ is 
\textbf{finitely presented}, the homomorphism becomes an \textbf{isomorphism}%
.
\end{theorem}

\begin{proof}
It is easy to see that, if $M=k$, then the homomorphism becomes an
isomorphism:%
\begin{equation*}
\mathbf{P}\otimes _{k}M=\mathbf{P}\otimes _{k}k=\mathbf{P=\mathbf{P}}%
\widetilde{\otimes }_{k}k\mathbf{=P}\widetilde{\otimes }_{k}M.
\end{equation*}

Since both functors are additive (even $k$-linear), the same is true when $%
M=k^{n}$ is a free finitely generated module:%
\begin{equation*}
\mathbf{P}\otimes _{k}M=\mathbf{P}\otimes _{k}k^{n}=\dbigoplus\limits^{n}%
\mathbf{P=\mathbf{P}\widetilde{\otimes }}_{k}k^{n}\mathbf{=P}\widetilde{%
\otimes }_{k}M.
\end{equation*}

If $M$ is finitely generated, then%
\begin{eqnarray*}
M &=&\mathbf{coker}\left( k^{X}\longrightarrow k^{n}\right) , \\
\mathbf{P}\otimes _{k}M &=&\mathbf{coker}\left( \dbigoplus\limits_{X}\mathbf{%
P}\longrightarrow \mathbf{P}^{n}\right) , \\
\mathbf{P}\widetilde{\otimes }_{k}M &=&\mathbf{coker}\left( \mathbf{P}%
\widetilde{\otimes }_{k}k^{X}\longrightarrow \mathbf{P}^{n}\right) ,
\end{eqnarray*}%
and both pro-modules are factormodules of the same pro-module $\mathbf{P}%
^{n} $, hence $\mathbf{P}\otimes _{k}M\rightarrow \mathbf{P}\widetilde{%
\otimes }_{k}M$ is an epimorphism. Finally, if $M$ is finitely presented,
then%
\begin{eqnarray*}
M &=&\mathbf{coker}\left( k^{m}\longrightarrow k^{n}\right) , \\
\mathbf{P}\otimes _{k}M &=&\mathbf{coker}\left( \mathbf{P}%
^{m}\longrightarrow \mathbf{P}^{n}\right) =\mathbf{coker}\left( \mathbf{P}%
\widetilde{\otimes }_{k}k^{m}\longrightarrow \mathbf{P}\widetilde{\otimes }%
_{k}k^{n}\right) =\mathbf{P}\widetilde{\otimes }_{k}M.
\end{eqnarray*}
\end{proof}

\subsection{Quasi-projectives}

\begin{definition}
\label{Def-quasi-projective}(dual to \cite{Kashiwara-Categories-MR2182076},
Definition 15.2.1) A pro-module $\mathbf{P}$ is called \textbf{%
quasi-projective} if the functor%
\begin{equation*}
Hom_{\mathbf{Pro}\left( k\right) }\left( \mathbf{P},?\right) :\mathbf{Mod}%
\left( k\right) \longrightarrow \mathbf{Mod}\left( k\right)
\end{equation*}%
is exact.
\end{definition}

\begin{proposition}
\label{Prop-quasi-projective}A pro-module $\mathbf{P}$ is quasi-projective
iff it is isomorphic to a pro-module $\left( Q_{i}\right) _{i\in \mathbf{I}}$
where all modules $Q_{i}\in \mathbf{Mod}\left( k\right) $ are projective.
\end{proposition}

\begin{proof}
The statement is dual to \cite{Kashiwara-Categories-MR2182076}, Proposition
15.2.3.
\end{proof}

\begin{remark}
\label{Rem-not-enough-projectives}The category $\mathbf{Pro}\left( k\right) $
does not have enough projectives (compare with \cite%
{Kashiwara-Categories-MR2182076}, Corollary 15.1.3). However, it has enough
quasi-projectives (see Proposition \ref{Prop-enough-quasi-projectives}
below).
\end{remark}

\begin{proposition}
\label{Prop-enough-quasi-projectives}For any pro-module $\mathbf{M}$ there
exists an epimorphism $\mathbf{P}\rightarrow \mathbf{M}$ where $\mathbf{P}$
is quasi-projective.
\end{proposition}

\begin{proof}
The statement is dual to the rather complicated Theorem 15.2.5 from \cite%
{Kashiwara-Categories-MR2182076}. However, the proof is much simpler in our
case. Given $\mathbf{M}=\left( M_{i}\right) _{i\in \mathbf{I}}$, let $%
\mathbf{P}=\left( P_{i}=F\left( M_{i}\right) \right) _{i\in \mathbf{I}}$,
where $F\left( M_{i}\right) $ is the free $k$-module generated by the set of
symbols $\left( \left[ m\right] \right) _{m\in M}$. A family of epimorphisms 
\begin{equation*}
f_{i}:P_{i}\longrightarrow M_{i}~\left( f_{i}\left( \dsum\limits_{j}k_{j} 
\left[ m_{j}\right] \right) =\dsum\limits_{j}k_{j}m_{j},k_{j}\in k,m_{j}\in
M_{i}\right) ,
\end{equation*}%
defines the desired level epimorphism $f:\mathbf{P}\rightarrow \mathbf{M}$.
\end{proof}

\subsection{Quasi-noetherian rings}

\begin{definition}
\label{Def-quasi-noetherian}A commutative ring $k$ is called \textbf{%
quasi-noetherian} if for any quasi-projective $\mathbf{P}\in \mathbf{Pro}%
\left( k\right) $, and any injective $J\in \mathbf{Mod}\left( k\right) $,
the $k$-module $Hom_{\mathbf{Pro}\left( k\right) }\left( \mathbf{P},J\right) 
$ is injective.
\end{definition}

\begin{remark}
In \cite{Sugiki-2001-33}, Definition 2.1.10, such rings are called
\textquotedblleft satisfying condition A\textquotedblright . However,
Proposition \ref{Prop-noetherian-is-quasi-noetherian} below justifies our
name.
\end{remark}

\begin{lemma}
\label{Direct-limit-injectives}Let $k$ be a noetherian ring. Then any
filtrant direct limit of injective $k$-modules is injective.
\end{lemma}

\begin{proof}
See \cite{Lam-MR1653294}, Theorem 3.46.
\end{proof}

\begin{proposition}
\label{Prop-noetherian-is-quasi-noetherian}A noetherian ring is
quasi-noetherian.
\end{proposition}

\begin{remark}
This proposition answers positively Conjecture 2.1.11 from \cite%
{Sugiki-2001-33}.
\end{remark}

\begin{proof}
A quasi-projective $\mathbf{P}$ can be represented by an inverse system $%
\left( P_{i}\right) _{i\in \mathbf{I}}$ where all $P_{i}$ are projective $k$%
-modules. Therefore, the $k$-module%
\begin{equation*}
Hom_{\mathbf{Pro}\left( k\right) }\left( \mathbf{P},J\right) =%
\underrightarrow{\lim }_{i\in \mathbf{I}}~Hom_{\mathbf{Mod}\left( k\right)
}\left( P_{i},J\right) ,
\end{equation*}%
being a filtrant direct limit of injectives, is injective as well, by Lemma %
\ref{Direct-limit-injectives}.
\end{proof}

The proposition \ref{Prop-quasi-projectives-are-flat} below shows that
quasi-projective pro-modules over a quasi-noetherian ring are flat.

\begin{proposition}
\label{Prop-quasi-projectives-are-flat}If $k$ is quasi-noetherian and $%
\mathbf{P}\in \mathbf{Pro}\left( k\right) $ is quasi-projective, then the
functor $\mathbf{P}\otimes _{k}?:\mathbf{Mod}\left( k\right) \rightarrow 
\mathbf{Mod}\left( k\right) $ is exact.
\end{proposition}

\begin{proof}
See \cite{Sugiki-2001-33}, Theorem 2.1.12.
\end{proof}

\subsection{$\mathbf{Tor}$ for pro-modules}

We define the torsion functors $\mathbf{Tor}_{\ast }^{k}:\mathbf{Pro}\left(
k\right) \rightarrow \mathbf{Pro}\left( k\right) $ as the left derived
functors of the functor 
\begin{equation*}
\otimes _{k}:\mathbf{Pro}\left( k\right) \times \mathbf{Mod}\left( k\right)
\longrightarrow \mathbf{Pro}\left( k\right)
\end{equation*}%
with respect to the \textbf{second} variable. Later, in Proposition \ref%
{Prop-quasi-projective-resolution}, we will show that these functors can be
equally defined by using the \textbf{first} variable (provided $k$ is
quasi-noetherian).

\begin{definition}
\label{Def-Tor-Pro(k)}Let $\mathbf{M}\in \mathbf{Pro}\left( k\right) $, $%
G\in \mathbf{Mod}\left( k\right) $, and let 
\begin{equation*}
0\longleftarrow G\longleftarrow P_{0}\longleftarrow P_{1}\longleftarrow
P_{2}\longleftarrow \mathbf{...}
\end{equation*}%
be a projective resolution of $G$. Define 
\begin{equation*}
\mathbf{Tor}_{n}^{k}\left( \mathbf{M},G\right) :=\mathbf{H}_{n}\left( 
\mathbf{M}\otimes _{k}P_{\ast }\right) ,n\geq 0,
\end{equation*}%
to be the pro-homology of $\left( \mathbf{M}\otimes _{k}P_{\ast }\right) \in 
\mathbf{CHAIN}\left( \mathbf{Pro}\left( k\right) \right) $.
\end{definition}

\begin{remark}
Using the standard homological techniques, one can easily show that $\mathbf{%
Tor}_{n}^{k}$ are well defined bi-additive (even $k$-bilinear) functors from 
$\mathbf{Pro}\left( k\right) \times \mathbf{Mod}\left( k\right) $ to $%
\mathbf{Pro}\left( k\right) $. Moreover, since $\otimes _{k}$ is right
exact, $\mathbf{Tor}_{0}^{k}$ is naturally isomorphic to $\otimes _{k}$.
\end{remark}

\begin{proposition}
\label{Prop-first-long-exact-sequence}(The first long exact sequence) Let $%
G\in \mathbf{Mod}\left( k\right) $, and let%
\begin{equation*}
0\mathbf{\longrightarrow M\longrightarrow N\longrightarrow K}\longrightarrow
0
\end{equation*}%
be a short exact sequence of pro-modules. Then there exists a natural long
exact sequence of pro-modules%
\begin{eqnarray*}
0 &\longleftarrow &\mathbf{K}\otimes _{k}G\longleftarrow \mathbf{N}\otimes
_{k}G\longleftarrow \mathbf{M}\otimes _{k}G\longleftarrow \mathbf{Tor}%
_{1}^{k}\left( \mathbf{K},G\right) \longleftarrow ... \\
... &\longleftarrow &\mathbf{Tor}_{n}^{k}\left( \mathbf{K},G\right)
\longleftarrow \mathbf{Tor}_{n}^{k}\left( \mathbf{N},G\right) \longleftarrow 
\mathbf{Tor}_{n}^{k}\left( \mathbf{M},G\right) \longleftarrow \mathbf{Tor}%
_{n+1}^{k}\left( \mathbf{K},G\right) \longleftarrow ...
\end{eqnarray*}
\end{proposition}

\begin{proof}
Let%
\begin{equation*}
0\longleftarrow G\longleftarrow P_{0}\longleftarrow P_{1}\longleftarrow
P_{2}\longleftarrow \mathbf{...}
\end{equation*}%
be a projective resolution of $G$. Due to Theorem \ref{Th-tensor-product} (%
\ref{Th-tensor-product-exact}), there is a short exact sequence of complexes
of pro-modules%
\begin{equation*}
0\mathbf{\longrightarrow M\otimes }_{k}P_{\ast }\mathbf{\longrightarrow N}%
\otimes _{k}P_{\ast }\mathbf{\longrightarrow K\otimes }_{k}P_{\ast
}\longrightarrow 0.
\end{equation*}%
The corresponding long exact sequence of pro-homologies is as desired.
\end{proof}

\begin{proposition}
\label{Prop-pairing-pro-complexes}Let $\mathbf{C}_{\ast }\in \mathbf{Pro}%
\left( \mathbf{CHAIN}\left( k\right) \right) $ be a pro-complex of $k$%
-modules, and let $G\in \mathbf{Mod}\left( k\right) $. There is a natural on 
$\mathbf{C}_{\ast }$ and $G$ pairing%
\begin{equation*}
\mathbf{H}_{\ast }\left( \mathbf{C}_{\ast }\right) \otimes
_{k}G\longrightarrow \mathbf{H}_{\ast }\left( \mathbf{C}_{\ast }\otimes
_{k}G\right) .
\end{equation*}
\end{proposition}

\begin{proof}
For each $a\in G$ we will define morphisms%
\begin{equation*}
\overline{a}_{\ast }:\mathbf{H}_{\ast }\left( \mathbf{C}_{\ast }\right)
\longrightarrow \mathbf{H}_{\ast }\left( \mathbf{C}_{\ast }\otimes
_{k}G\right)
\end{equation*}%
which satisfy the condition $\left( \overline{sa+tb}\right) _{\ast }=%
\overline{a}_{\ast }+\overline{b}_{\ast }$, $a$, $b\in G$, $s$, $t\in k$.
Since, due to Theorem \ref{Th-tensor-product},%
\begin{equation*}
Hom_{\mathbf{Mod}\left( k\right) }\left( G,Hom_{\mathbf{Pro}\left( k\right)
}\left( \mathbf{C}_{\ast },\mathbf{C}_{\ast }\otimes _{k}G\right) \right) 
\simeq%
Hom_{\mathbf{Pro}\left( k\right) }\left( \mathbf{C}_{\ast }\otimes _{k}G,%
\mathbf{C}_{\ast }\otimes _{k}G\right) ,
\end{equation*}%
let 
\begin{equation*}
\varphi _{n}:G\longrightarrow Hom_{\mathbf{Pro}\left( k\right) }\left( 
\mathbf{C}_{n},\mathbf{C}_{n}\otimes _{k}G\right)
\end{equation*}%
be a family of morphisms corresponding to $1_{n}\in Hom_{\mathbf{Pro}\left(
k\right) }\left( \mathbf{C}_{n}\otimes _{k}G,\mathbf{C}_{n}\otimes
_{k}G\right) $. The morphisms 
\begin{equation*}
\varphi _{n}\left( a\right) :\mathbf{C}_{n}\longrightarrow \mathbf{C}%
_{n}\otimes _{k}G
\end{equation*}%
define a chain mapping inducing the desired morphisms $\overline{a}_{n}:%
\mathbf{H}_{n}\left( \mathbf{C}_{\ast }\right) \rightarrow \mathbf{H}%
_{n}\left( \mathbf{C}_{\ast }\otimes _{k}G\right) $.
\end{proof}

\textbf{From now on, we assume that }$k$\textbf{\ is quasi-noetherian}.
Proposition \ref{Prop-quasi-projective-resolution} below shows that
quasi-projective resolutions can be used to define the torsion functors.

\begin{proposition}
\label{Prop-quasi-projective-resolution}Let $\mathbf{M}\in \mathbf{Pro}%
\left( k\right) $, $G\in \mathbf{Mod}\left( k\right) $, and let 
\begin{equation*}
0\longleftarrow \mathbf{M}\longleftarrow \mathbf{P}_{0}\longleftarrow 
\mathbf{P}_{1}\longleftarrow \mathbf{P}_{2}\longleftarrow \mathbf{...}
\end{equation*}%
be an exact sequence of pro-modules where all $\mathbf{P}_{i}$ are
quasi-projective. Then the homology groups of the complex $\mathbf{P}_{\ast
}\otimes _{k}G$ are naturally isomorphic to the torsion pro-modules: 
\begin{equation*}
\mathbf{H}_{n}\left( \mathbf{P}_{\ast }\otimes _{k}G\right) 
\simeq%
\mathbf{Tor}_{n}^{k}\left( \mathbf{M},G\right) ,n\geq 0.
\end{equation*}
\end{proposition}

\begin{proof}
Induction on $n$.

\textbf{Step} $n=0$: since $\otimes _{k}$ commutes with direct limits,%
\begin{eqnarray*}
&&\mathbf{H}_{0}\left( \mathbf{P}_{\ast }\otimes _{k}G\right) 
\simeq%
\mathbf{coker}\left( \mathbf{P}_{1}\otimes _{k}G\longrightarrow \mathbf{P}%
_{0}\otimes _{k}G\right) 
\simeq
\\
&&\mathbf{coker}\left( \mathbf{P}_{1}\longrightarrow \mathbf{P}_{0}\right)
\otimes _{k}G%
\simeq%
\mathbf{M}\otimes _{k}G%
\simeq%
\mathbf{Tor}_{0}^{k}\left( \mathbf{M},G\right) .
\end{eqnarray*}

\textbf{Step} $n\Rightarrow n+1$: let $\mathbf{N}=\ker \left( \mathbf{M}%
\longleftarrow \mathbf{P}_{0}\right) $. The short exact sequence%
\begin{equation*}
0\longrightarrow \mathbf{N}\longrightarrow \mathbf{P}_{0}\longrightarrow 
\mathbf{M}\longrightarrow 0
\end{equation*}%
induces, due to Proposition \ref{Prop-first-long-exact-sequence}, a long
exact sequence%
\begin{equation*}
...\longleftarrow 0=\mathbf{Tor}_{n}^{k}\left( \mathbf{P}_{0},G\right)
\longleftarrow \mathbf{Tor}_{n}^{k}\left( \mathbf{N},G\right) \longleftarrow 
\mathbf{Tor}_{n+1}^{k}\left( \mathbf{M},G\right) \longleftarrow 0=\mathbf{Tor%
}_{n}^{k}\left( \mathbf{P}_{0},G\right) \longleftarrow ...,
\end{equation*}%
which, in turn, implies an isomorphism $\mathbf{Tor}_{n}^{k}\left( \mathbf{N}%
,G\right) 
\simeq%
\mathbf{Tor}_{n+1}^{k}\left( \mathbf{M},G\right) $. There is an evident
resolution for $\mathbf{N}$:%
\begin{equation*}
0\longleftarrow \mathbf{N}\longleftarrow \mathbf{P}_{1}\longleftarrow 
\mathbf{P}_{2}\longleftarrow \mathbf{P}_{3}\longleftarrow \mathbf{...}
\end{equation*}%
By the induction hypothesis,%
\begin{equation*}
\mathbf{Tor}_{n}^{k}\left( \mathbf{N},G\right) 
\simeq%
\mathbf{H}_{n}\left( \mathbf{P}_{\ast +1}\otimes _{k}G\right) 
\simeq%
\mathbf{H}_{n+1}\left( \mathbf{P}_{\ast }\otimes _{k}G\right) .
\end{equation*}%
Finally, $\mathbf{Tor}_{n+1}^{k}\left( \mathbf{M},G\right) 
\simeq%
\mathbf{Tor}_{n}^{k}\left( \mathbf{N},G\right) 
\simeq%
\mathbf{H}_{n+1}\left( \mathbf{P}_{\ast }\otimes _{k}G\right) $, as desired.
\end{proof}

\begin{proposition}
\label{Prop-second-long-exact-sequence}(The second long exact sequence) Let $%
\mathbf{M}\in \mathbf{Pro}\left( k\right) $, and let%
\begin{equation*}
0\longrightarrow A\longrightarrow B\longrightarrow C\longrightarrow 0
\end{equation*}%
be a short exact sequence of $k$-modules. Then there exists a natural long
exact sequence of pro-modules%
\begin{eqnarray*}
0 &\longleftarrow &\mathbf{M}\otimes _{k}C\longleftarrow \mathbf{M}\otimes
_{k}B\longleftarrow \mathbf{M}\otimes _{k}A\longleftarrow \mathbf{Tor}%
_{1}^{k}\left( \mathbf{M},C\right) \longleftarrow ... \\
... &\longleftarrow &\mathbf{Tor}_{n}^{k}\left( \mathbf{M},C\right)
\longleftarrow \mathbf{Tor}_{n}^{k}\left( \mathbf{B},B\right) \longleftarrow 
\mathbf{Tor}_{n}^{k}\left( \mathbf{M},A\right) \longleftarrow \mathbf{Tor}%
_{n+1}^{k}\left( \mathbf{M},C\right) \longleftarrow ...
\end{eqnarray*}
\end{proposition}

\begin{proof}
Let%
\begin{equation*}
0\longleftarrow \mathbf{M}\longleftarrow \mathbf{P}_{0}\longleftarrow 
\mathbf{P}_{1}\longleftarrow \mathbf{P}_{2}\longleftarrow \mathbf{...}
\end{equation*}%
be a quasi-projective resolution of $\mathbf{M}$. Due to Proposition \ref%
{Prop-quasi-projectives-are-flat}, there is a short exact sequence of
complexes of pro-modules%
\begin{equation*}
0\mathbf{\longrightarrow P}_{\ast }\mathbf{\otimes }_{k}A\mathbf{%
\longrightarrow P}_{\ast }\otimes _{k}B\mathbf{\longrightarrow P}_{\ast }%
\mathbf{\otimes }_{k}C\longrightarrow 0.
\end{equation*}%
Because of Proposition \ref{Prop-quasi-projective-resolution}, the
corresponding long exact sequence of pro-homologies is isomorphic to the
desired sequence.
\end{proof}

\begin{theorem}
\label{Th-UCF-pro-complexes}Let $\mathbf{P}_{\ast }\in \mathbf{Pro}\left( 
\mathbf{CHAIN}\left( \mathbb{Z}\right) \right) $ be a pro-complex of abelian
groups, such that $\mathbf{P}_{n}$ is quasi-projective for all $n\in \mathbb{%
Z}$, and let $G\in \mathbf{Mod}\left( \mathbb{Z}\right) $. Then for any $%
n\in \mathbb{Z}$ the pairing 
\begin{equation*}
\mathbf{H}_{n}\left( \mathbf{P}_{\ast }\right) \otimes _{\mathbb{Z}%
}G\longrightarrow \mathbf{H}_{n}\left( \mathbf{P}_{\ast }\otimes _{\mathbb{Z}%
}G\right)
\end{equation*}%
is a monomorphism, with the cokernel naturally (on $\mathbf{P}_{\ast }$ and $%
G)$ isomorphic to $\mathbf{Tor}_{1}^{\mathbb{Z}}\left( \mathbf{H}%
_{n-1}\left( \mathbf{P}_{\ast }\right) ,G\right) $. In other words, there
exist natural on $\mathbf{P}_{\ast }$ and $G$ short exact sequences ($n\in 
\mathbb{Z}$)%
\begin{equation*}
0\longrightarrow \mathbf{H}_{n}\left( \mathbf{P}_{\ast }\right) \otimes _{%
\mathbb{Z}}G\longrightarrow \mathbf{H}_{n}\left( \mathbf{P}_{\ast }\otimes _{%
\mathbb{Z}}G\right) \longrightarrow \mathbf{Tor}_{1}^{\mathbb{Z}}\left( 
\mathbf{H}_{n-1}\left( \mathbf{P}_{\ast }\right) ,G\right) \longrightarrow 0.
\end{equation*}
\end{theorem}

\begin{proof}
Let%
\begin{equation*}
0\longleftarrow G\longleftarrow Q_{0}\longleftarrow Q_{1}\longleftarrow 0
\end{equation*}%
be a free resolution for $G$. Due to Proposition \ref%
{Prop-quasi-projectives-are-flat}, there is a short exact sequence in $%
\mathbf{CHAIN}\left( \mathbf{Pro}\left( \mathbb{Z}\right) \right) $:%
\begin{equation*}
0\longrightarrow \mathbf{P}_{\ast }\otimes _{\mathbb{Z}}Q_{1}\longrightarrow 
\mathbf{P}_{\ast }\otimes _{\mathbb{Z}}Q_{0}\longrightarrow \mathbf{P}_{\ast
}\otimes _{\mathbb{Z}}G\longrightarrow 0.
\end{equation*}%
Since the functors $?\otimes _{\mathbb{Z}}Q_{0}$ and $?\otimes _{\mathbb{Z}%
}Q_{1}$ are exact, the induced long exact sequence of pro-homologies will
have the following form:%
\begin{eqnarray*}
... &\longrightarrow &\mathbf{H}_{n}\left( \mathbf{P}_{\ast }\right) \otimes
_{\mathbb{Z}}Q_{1}\longrightarrow \mathbf{H}_{n}\left( \mathbf{P}_{\ast
}\right) \otimes _{\mathbb{Z}}Q_{0}\longrightarrow \mathbf{H}_{n}\left( 
\mathbf{P}_{\ast }\otimes _{\mathbb{Z}}G\right) \longrightarrow \\
&\longrightarrow &\mathbf{H}_{n-1}\left( \mathbf{P}_{\ast }\right) \otimes _{%
\mathbb{Z}}Q_{1}\longrightarrow \mathbf{H}_{n-1}\left( \mathbf{P}_{\ast
}\right) \otimes _{\mathbb{Z}}Q_{0}\longrightarrow ....
\end{eqnarray*}%
The latter sequence splits into the desired pieces:%
\begin{eqnarray*}
0 &\longrightarrow &\mathbf{coker}\left( \mathbf{H}_{n}\left( \mathbf{P}%
_{\ast }\right) \otimes _{\mathbb{Z}}Q_{1}\longrightarrow \mathbf{H}%
_{n}\left( \mathbf{P}_{\ast }\right) \otimes _{\mathbb{Z}}Q_{0}\right) 
\simeq%
\mathbf{H}_{n}\left( \mathbf{P}_{\ast }\right) \otimes _{\mathbb{Z}%
}G\longrightarrow \mathbf{H}_{n}\left( \mathbf{P}_{\ast }\otimes _{\mathbb{Z}%
}G\right) \longrightarrow \\
&\longrightarrow &\ker \left( \mathbf{H}_{n-1}\left( \mathbf{P}_{\ast
}\right) \otimes _{\mathbb{Z}}Q_{1}\longrightarrow \mathbf{H}_{n-1}\left( 
\mathbf{P}_{\ast }\right) \otimes _{\mathbb{Z}}Q_{0}\right) 
\simeq%
\mathbf{Tor}_{1}^{\mathbb{Z}}\left( \mathbf{H}_{n-1}\left( \mathbf{P}_{\ast
}\right) ,G\right) \longrightarrow 0.
\end{eqnarray*}
\end{proof}

\section{\label{Sec-spectral}Spectral sequences}

\subsection{Towers}

Let%
\begin{equation*}
\mathbf{A}=\left( ...\longrightarrow A_{n}\overset{i=i_{n}}{\longrightarrow }%
A_{n-1}\longrightarrow ...\longrightarrow A_{1}\longrightarrow A_{0}\right)
\end{equation*}%
be a tower of abelian groups, and let%
\begin{equation*}
\mathbf{A}^{\left( r\right) }=\left( ...\longrightarrow A_{n}^{\left(
r\right) }\overset{p}{\longrightarrow }A_{n-1}^{\left( r\right)
}\longrightarrow ...\longrightarrow A_{1}^{\left( r\right) }\longrightarrow
A_{0}^{\left( r\right) }\right)
\end{equation*}%
where%
\begin{equation*}
A_{n}^{\left( r\right) }=i^{r}A_{n}.
\end{equation*}

\begin{remark}
We assume that $i_{n}:A_{n}\longrightarrow A_{n-1}$ is zero when $n\geq 0$,
hence $A_{n}^{\left( r\right) }=0$ when $r>n$.
\end{remark}

\begin{definition}
\label{Def-lim-1}Let $\mathbf{A=}\left( A_{n},i\right) $ be a tower. The
first derived inverse limit is defined as follows:%
\begin{equation*}
\underleftarrow{\lim }_{n}^{1}A_{n}=\mathbf{coker}\left(
\dprod\limits_{n\geq 0}A_{n}\overset{1-i}{\longrightarrow }%
\dprod\limits_{n\geq 0}A_{n}\right) .
\end{equation*}

\begin{remark}
It is clear that%
\begin{equation*}
\underleftarrow{\lim }_{n}A_{n}=\ker \left( \dprod\limits_{n\geq 0}A_{n}%
\overset{1-i}{\longrightarrow }\dprod\limits_{n\geq 0}A_{n}\right) .
\end{equation*}
\end{remark}
\end{definition}

\begin{theorem}
\label{Mittag-Leffler}(Mittag-Leffler)

\begin{enumerate}
\item \label{Mittag-Leffler-exact-sequence}For a short exact sequence of
towers%
\begin{equation*}
0\longrightarrow \left( A_{n}\right) \longrightarrow \left( B_{n}\right)
\longrightarrow \left( C_{n}\right) \longrightarrow 0
\end{equation*}%
there exist an 8 term exact sequence%
\begin{equation*}
0\longrightarrow \underleftarrow{\lim }_{n}~A_{n}\longrightarrow 
\underleftarrow{\lim }_{n}~B_{n}\longrightarrow \underleftarrow{\lim }%
_{n}~C_{n}\longrightarrow \underleftarrow{\lim }_{n}^{1}A_{n}\longrightarrow 
\underleftarrow{\lim }_{n}^{1}B_{n}\longrightarrow \underleftarrow{\lim }%
_{n}^{1}C_{n}\longrightarrow 0.
\end{equation*}

\item \label{Mittag-Leffler-epimorphisms}If $i_{n}$ are epimorphisms for all 
$n$, then $\underleftarrow{\lim }_{n}^{1}A_{n}=0$. The condition can be
substituted by a weaker \textquotedblleft Mittag-Leffler
condition\textquotedblright , but we do not need this generalization.

\item \label{Mittag-Leffler-numeration}For any $r$,%
\begin{equation*}
\underleftarrow{\lim }_{n}A_{n+r}%
\simeq%
\underleftarrow{\lim }_{n}A_{n},~\underleftarrow{\lim }_{n}^{1}A_{n+r}%
\simeq%
\underleftarrow{\lim }_{n}^{1}A_{n}.
\end{equation*}

\item \label{Mittag-Leffler-images}For any $r$,%
\begin{equation*}
\underleftarrow{\lim }_{n}A_{n}^{\left( r\right) }%
\simeq%
\underleftarrow{\lim }_{n}A_{n},~\underleftarrow{\lim }_{n}^{1}A_{n}^{\left(
r\right) }%
\simeq%
\underleftarrow{\lim }_{n}^{1}A_{n}.
\end{equation*}

\item \label{Mittag-Leffler-lim-lim}%
\begin{equation*}
\underleftarrow{\lim }_{n}~\underleftarrow{\lim }_{r}A_{n}^{\left( r\right) }%
\simeq%
\underleftarrow{\lim }_{n}A_{n}.
\end{equation*}

\item \label{Mittag-Leffler-lim-lim1}There is a short exact sequence%
\begin{equation*}
0\longrightarrow \underleftarrow{\lim }_{n}^{1}~\underleftarrow{\lim }%
_{r}A_{n}^{\left( r\right) }\longrightarrow \underleftarrow{\lim }%
_{n}^{1}A_{n}\longrightarrow \underleftarrow{\lim }_{n}~\underleftarrow{\lim 
}_{r}^{1}A_{n}^{\left( r\right) }\longrightarrow 0.
\end{equation*}

\item \label{Mittag-Leffler-lim1-lim1}%
\begin{equation*}
\underleftarrow{\lim }_{n}^{1}~\underleftarrow{\lim }_{r}^{1}A_{n}^{\left(
r\right) }=0
\end{equation*}

\item \label{Mittag-Leffler-lim1-lim1-epi}%
\begin{equation*}
\underleftarrow{\lim }_{r}^{1}~A_{n}^{\left( r\right) }\longrightarrow 
\underleftarrow{\lim }_{r}^{1}~A_{n-1}^{\left( r\right) }
\end{equation*}%
are epimorphisms for all $n\in \mathbb{Z}$.
\end{enumerate}
\end{theorem}

\begin{proof}
To prove (\ref{Mittag-Leffler-exact-sequence}), consider the following short
exact sequence of cochain complexes:%
\begin{equation*}
\begin{diagram}[size=3.0em,textflow]
0 & \rTo & \dprod\limits_{n\geq 0}A_{n} & \rTo & \dprod\limits_{n\geq 0}B_{n} & \rTo & \dprod\limits_{n\geq 0}C_{n} & \rTo & 0 \\
  &      & \dTo_{1-i}                   &      & \dTo_{1-i}                   &      & \dTo_{1-i}           \\
0 & \rTo & \dprod\limits_{n\geq 0}A_{n} & \rTo & \dprod\limits_{n\geq 0}B_{n} & \rTo & \dprod\limits_{n\geq 0}C_{n} & \rTo & 0 \\
\end{diagram}%
\end{equation*}%
The long exact sequence of cohomologies reduces to the desired $8$ term
exact sequence, because%
\begin{equation*}
H^{n}\left( \dprod\limits_{n\geq 0}A_{n}\overset{1-i}{\longrightarrow }%
\dprod\limits_{n\geq 0}A_{n}\right) =\left\{ 
\begin{array}{ccccc}
\underleftarrow{\lim }_{n}~A_{n} & \text{if} & n=0 &  &  \\ 
\underleftarrow{\lim }_{n}^{1}~A_{n} & \text{if} & n=1 &  &  \\ 
0 & \text{if} & n<0 & \text{or} & n>1%
\end{array}%
\right. ,
\end{equation*}%
and similarly for $\left( B_{n}\right) $ and $\left( C_{n}\right) $.

To prove (\ref{Mittag-Leffler-epimorphisms}), notice that%
\begin{equation*}
\dprod\limits_{n\geq 0}A_{n}\overset{1-i}{\longrightarrow }%
\dprod\limits_{n\geq 0}A_{n}
\end{equation*}%
is an epimorphism whenever all $i_{n}$ are epimorphisms.

To prove (\ref{Mittag-Leffler-numeration}), let%
\begin{equation*}
\begin{diagram}[size=3.0em,textflow]
0 & \rTo & C^{0}=\dprod\limits_{n\geq 0}A_{n+1} & \rTo^{d=1-i} & C^{1}=\dprod\limits_{n\geq 0}A_{n+1} & \rTo & 0 \\
  &      &  \dTo^{i} \uTo_{j}                   &              & \dTo^{i} \uTo_{j} \\
0 & \rTo & D^{0}=\dprod\limits_{n\geq 0}A_{n} & \rTo^{d=1-i} & D^{1}=\dprod\limits_{n\geq 0}A_{n} & \rTo & 0 \\
\end{diagram}%
\end{equation*}%
be a pair of morphisms%
\begin{equation*}
C^{\ast }\overset{i}{\longrightarrow }D^{\ast }\overset{j}{\longrightarrow }%
C^{\ast }
\end{equation*}%
of cochain complexes where $j$ is the natural projection.

Define cochain homotopies%
\begin{equation*}
\begin{diagram}[size=3.0em,textflow]
0 & \rTo & C^{0}=\dprod\limits_{n\geq 0}A_{n+1} & \rTo^{1-i} & C^{1}=\dprod\limits_{n\geq 0}A_{n+1} & \rTo & 0 \\
  &      &  \dTo^{ji} \dTo_{1}                  & \ldTo^{S}    & \dTo^{ji} \dTo_{1} \\
0 & \rTo & C^{0}=\dprod\limits_{n\geq 0}A_{n+1}   & \rTo_{1-i} & C^{1}=\dprod\limits_{n\geq 0}A_{n+1} & \rTo & 0 \\
\end{diagram}%
\end{equation*}%
by%
\begin{eqnarray*}
S\left( a_{1},a_{2},...,a_{n},...\right) &=&\left(
-a_{1},-a_{2},...,-a_{n},...\right) \\
\left( ji-1\right) \left( a_{1},a_{2},...,a_{n},...\right) &=&\left(
ia_{2}-a_{1},ia_{3}-a_{2},...,ia_{n+1}-a_{n},...\right) , \\
ji-1 &=&Sd+dS,
\end{eqnarray*}%
and%
\begin{equation*}
\begin{diagram}[size=3.0em,textflow]
0 & \rTo & D^{0}=\dprod\limits_{n\geq 0}A_{n} & \rTo^{1-i} & D^{1}=\dprod\limits_{n\geq 0}A_{n} & \rTo & 0 \\
  &      &  \dTo^{ij} \dTo_{1}                  & \ldTo^{T}    & \dTo^{ji} \dTo_{1} \\
0 & \rTo & D^{0}=\dprod\limits_{n\geq 0}A_{n}   & \rTo_{1-i} & D^{1}=\dprod\limits_{n\geq 0}A_{n} & \rTo & 0 \\
\end{diagram}%
\end{equation*}

by%
\begin{eqnarray*}
T\left( a_{0},a_{1},...,a_{n},...\right) &=&\left(
-a_{0},-a_{1},...,-a_{n},...\right) \\
\left( ij-1\right) \left( a_{0},a_{1},...,a_{n},...\right) &=&\left(
ia_{1}-a_{0},ia_{1}-a_{0},...,ia_{n+1}-a_{n},...\right) , \\
ij-1 &=&Td+dT.
\end{eqnarray*}

Therefore, $i$ and $j$ are cochain homotopy equivalences which induce the
isomorphisms%
\begin{equation*}
\underleftarrow{\lim }_{n}~A_{n+1}%
\simeq%
\underleftarrow{\lim }_{n}~A_{n},~\underleftarrow{\lim }_{n}^{1}~A_{n+1}%
\simeq%
\underleftarrow{\lim }_{n}^{1}~A_{n}.
\end{equation*}%
By induction on $r$, one can easily deduce the desired isomorphisms%
\begin{equation*}
\underleftarrow{\lim }_{n}^{{}}~A_{n+r}%
\simeq%
\underleftarrow{\lim }_{n}^{{}}~A_{n},~\underset{n}{\underleftarrow{\lim }%
^{1}}A_{n+r}%
\simeq%
\underleftarrow{\lim }_{n}^{1}~A_{n}.
\end{equation*}%
To prove (\ref{Mittag-Leffler-images}), let%
\begin{equation*}
\begin{diagram}[size=3.0em,textflow]
0 & \rTo & C^{0}=\dprod\limits_{n\geq 0}A_{n} & \rTo^{d=1-i} & C^{1}=\dprod\limits_{n\geq 0}A_{n} & \rTo & 0 \\
  &      &  \dTo^{i} \uTo_{j}                   &              & \dTo^{i} \uTo_{j} \\
0 & \rTo & D^{0}=\dprod\limits_{n\geq 0}A_{n}^{(1)} & \rTo^{d=1-i} & D^{1}=\dprod\limits_{n\geq 0}A_{n}^{(1)} & \rTo & 0 \\
\end{diagram}%
\end{equation*}%
be a pair of morphisms%
\begin{equation*}
C^{\ast }\overset{i}{\longrightarrow }D^{\ast }\overset{j}{\longrightarrow }%
C^{\ast }
\end{equation*}%
of cochain complexes where $j$ is the natural inclusion.

Define cochain homotopies%
\begin{equation*}
\begin{diagram}[size=3.0em,textflow]
0 & \rTo & C^{0}=\dprod\limits_{n\geq 0}A_{n} & \rTo^{1-i} & C^{1}=\dprod\limits_{n\geq 0}A_{n} & \rTo & 0 \\
  &      &  \dTo^{ji} \dTo_{1}                  & \ldTo^{S}    & \dTo^{ji} \dTo_{1} \\
0 & \rTo & C^{0}=\dprod\limits_{n\geq 0}A_{n}   & \rTo_{1-i} & C^{1}=\dprod\limits_{n\geq 0}A_{n} & \rTo & 0 \\
\end{diagram}%
\end{equation*}%
by%
\begin{eqnarray*}
S\left( a_{0},a_{1},...,a_{n},...\right) &=&\left(
-a_{0},-a_{1},...,-a_{n},...\right) \\
\left( ji-1\right) \left( a_{0},a_{1},...,a_{n},...\right) &=&\left(
ia_{1}-a_{0},ia_{1}-a_{0},...,ia_{n+1}-a_{n},...\right) , \\
ji-1 &=&Sd+dS.
\end{eqnarray*}%
and%
\begin{equation*}
\begin{diagram}[size=3.0em,textflow]
0 & \rTo & D^{0}=\dprod\limits_{n\geq 0}A_{n}^{(1)} & \rTo^{1-i} & D^{1}=\dprod\limits_{n\geq 0}A_{n}^{(1)} & \rTo & 0 \\
  &      &  \dTo^{ij} \dTo_{1}                  & \ldTo^{T}    & \dTo^{ji} \dTo_{1} \\
0 & \rTo & D^{0}=\dprod\limits_{n\geq 0}A_{n}^{(1)}   & \rTo_{1-i} & D^{1}=\dprod\limits_{n\geq 0}A_{n}^{(1)} & \rTo & 0 \\
\end{diagram}%
\end{equation*}

by%
\begin{eqnarray*}
T\left( a_{0},a_{1},...,a_{n},...\right) &=&\left(
-a_{0},-a_{1},...,-a_{n},...\right) \\
\left( ij-1\right) \left( a_{0},a_{1},...,a_{n},...\right) &=&\left(
ia_{1}-a_{0},ia_{1}-a_{0},...,ia_{n+1}-a_{n},...\right) , \\
ij-1 &=&Td+dT.
\end{eqnarray*}

Hence, $i$ and $j$ are cochain homotopy equivalences which induce the
isomorphisms%
\begin{equation*}
\underleftarrow{\lim }_{n}~A_{n}^{\left( 1\right) }%
\simeq%
\underleftarrow{\lim }_{n}~A_{n},~\underleftarrow{\lim }_{n}^{1}~A_{n}^{%
\left( 1\right) }%
\simeq%
\underleftarrow{\lim }_{n}^{1}~A_{n}.
\end{equation*}%
By induction on $r$, one can easily deduce the desired isomorphisms%
\begin{equation*}
\underleftarrow{\lim }_{n}~A_{n}^{\left( r\right) }%
\simeq%
\underleftarrow{\lim }_{n}~A_{n},~\underleftarrow{\lim }_{n}^{1}~A_{n}^{%
\left( r\right) }%
\simeq%
\underleftarrow{\lim }_{n}^{1}~A_{n}.
\end{equation*}

To prove (\ref{Mittag-Leffler-lim-lim})-(\ref{Mittag-Leffler-lim1-lim1}),
consider the cochain bicomplex $F^{\ast \ast }$:%
\begin{equation*}
\begin{diagram}[size=3.0em,textflow]
0 & \rTo & F^{01}=\dprod\limits_{n\geq 0}\dprod\limits_{r\geq 0}A_{n}^{\left( r\right) } & \rTo^{\left( 1-i,1\right)} & F^{11}=\dprod\limits_{n\geq 0}\dprod\limits_{r\geq 0}A_{n}^{\left( r\right) } & \rTo & 0 \\
  &      &  \uTo^{\left( 1,1-i\right)}                   &              & \uTo^{\left( 1,1-i\right)} \\
0 & \rTo & F^{00}=\dprod\limits_{n\geq 0}\dprod\limits_{r\geq 0}A_{n}^{\left( r\right) } & \rTo^{\left( 1-i,1\right)} & F^{10}=\dprod\limits_{n\geq 0}\dprod\limits_{r\geq 0}A_{n}^{\left( r\right) } & \rTo & 0 \\
\end{diagram}%
\end{equation*}

($F^{st}=0$ if $s\neq 0$, $1$, or if $t\neq 0$, $1$). It produces two
\textquotedblleft classical\textquotedblright\ spectral sequences converging
to the cohomology of the total complex $T^{\ast }$:

\begin{eqnarray*}
E_{r}^{st} &\implies &H^{s+t}\left( T^{\ast }\right) , \\
\mathcal{E}_{r}^{st} &\implies &H^{s+t}\left( T^{\ast }\right) ,
\end{eqnarray*}%
where%
\begin{eqnarray*}
E_{2}^{00} &=&\underleftarrow{\lim }_{n}~\underleftarrow{\lim }%
_{r}A_{n}^{\left( r\right) },~E_{2}^{10}=\underleftarrow{\lim }_{n}^{1}~%
\underleftarrow{\lim }_{r}A_{n}^{\left( r\right) }, \\
E_{2}^{01} &=&\underleftarrow{\lim }_{n}~\underleftarrow{\lim }%
_{r}^{1}~A_{n}^{\left( r\right) },~E_{2}^{11}=\underleftarrow{\lim }_{n}^{1}~%
\underleftarrow{\lim }_{r}^{1}~A_{n}^{\left( r\right) },
\end{eqnarray*}

and, due to (\ref{Mittag-Leffler-images}),%
\begin{eqnarray*}
\mathcal{E}_{2}^{00} &=&\underleftarrow{\lim }_{r}~\underleftarrow{\lim }%
_{n}~A_{n}^{\left( r\right) }=\underleftarrow{\lim }_{r}~\underleftarrow{%
\lim }_{n}~A_{n}=\underleftarrow{\lim }_{n}~A_{n}, \\
\mathcal{E}_{2}^{10} &=&\underleftarrow{\lim }_{r}~\underleftarrow{\lim }%
_{n}^{1}~A_{n}^{\left( r\right) }=\underleftarrow{\lim }_{r}~\underleftarrow{%
\lim }_{n}^{1}~A_{n}=\underleftarrow{\lim }_{n}^{1}~A_{n}, \\
\mathcal{E}_{2}^{01} &=&\underleftarrow{\lim }_{r}^{1}~\underleftarrow{\lim }%
_{n}~A_{n}^{\left( r\right) }=\underleftarrow{\lim }_{r}^{1}~\underleftarrow{%
\lim }_{n}~A_{n}=0, \\
\mathcal{E}_{2}^{11} &=&\underleftarrow{\lim }_{r}^{1}~\underleftarrow{\lim }%
_{n}^{1}~A_{n}^{\left( r\right) }=\underleftarrow{\lim }_{r}^{1}~%
\underleftarrow{\lim }_{n}^{1}~A_{n}=0.
\end{eqnarray*}

It follows that%
\begin{eqnarray*}
\underleftarrow{\lim }_{n}^{{}}~\underleftarrow{\lim }_{r}^{{}}~A_{n}^{%
\left( r\right) } &=&E_{2}^{00}=H^{0}\left( T^{\ast }\right) =\mathcal{E}%
_{2}^{00}=\underleftarrow{\lim }_{n}~A_{n}, \\
\underleftarrow{\lim }_{n}^{1}~\underleftarrow{\lim }_{r}^{1}~A_{n}^{\left(
r\right) } &=&E_{2}^{11}=H^{2}\left( T^{\ast }\right) =\mathcal{E}%
_{2}^{11}=0.
\end{eqnarray*}

Moreover, from the second spectral sequence, 
\begin{equation*}
H^{1}\left( T^{\ast }\right) =\mathcal{E}_{2}^{10}=\underleftarrow{\lim }%
_{n}^{1}~A_{n},
\end{equation*}%
while from the first one, $H^{1}\left( T^{\ast }\right) $ has a two-fold
filtration with subquotients $E_{2}^{01}$ and $E_{2}^{10}$, giving the
desired exact sequence%
\begin{equation*}
0\longrightarrow E_{2}^{10}=\underleftarrow{\lim }_{n}^{1}~\underleftarrow{%
\lim }_{r}^{{}}~A_{n}^{\left( r\right) }\longrightarrow \underleftarrow{\lim 
}_{n}^{1}~A_{n}\longrightarrow \underleftarrow{\lim }_{n}^{{}}~%
\underleftarrow{\lim }_{r}^{1}~A_{n}^{\left( r\right)
}=E_{2}^{01}\longrightarrow 0.
\end{equation*}

It remains only to prove (\ref{Mittag-Leffler-lim1-lim1}). Fix $n$, and
consider an exact sequence of towers (with respect to $r$):%
\begin{equation*}
0\longrightarrow \left( B^{\left( r\right) }\right) \longrightarrow \left(
A_{n}^{\left( r\right) }\right) \overset{i}{\longrightarrow }\left(
A_{n-1}^{\left( r+1\right) }\right) \longrightarrow 0
\end{equation*}%
where 
\begin{equation*}
B^{\left( r\right) }=\ker \left( A_{n}^{\left( r\right) }\overset{i}{%
\longrightarrow }A_{n-1}^{\left( r+1\right) }\right) .
\end{equation*}%
This exact sequence induces the $8$ term exact sequence from (1)%
\begin{eqnarray*}
0 &\longrightarrow &\underleftarrow{\lim }_{r}^{{}}~B^{\left( r\right)
}\longrightarrow \underleftarrow{\lim }_{r}^{{}}~A_{n}^{\left( r\right)
}\longrightarrow \underleftarrow{\lim }_{r}^{{}}~A_{n-1}^{\left( r+1\right)
}\longrightarrow \\
&\longrightarrow &\underleftarrow{\lim }_{r}^{1}~B^{\left( r\right)
}\longrightarrow \underleftarrow{\lim }_{r}^{1}~A_{n}^{\left( r\right)
}\longrightarrow \underleftarrow{\lim }_{r}^{1}~A_{n-1}^{\left( r+1\right)
}\longrightarrow 0.
\end{eqnarray*}%
It follows that, due to (\ref{Mittag-Leffler-images}), 
\begin{equation*}
\underleftarrow{\lim }_{r}^{1}~A_{n-1}^{\left( r+1\right) }=\underleftarrow{%
\lim }_{r}^{1}~A_{n-1}^{\left( r\right) },
\end{equation*}%
and 
\begin{equation*}
\underleftarrow{\lim }_{r}^{1}~A_{n}^{\left( r\right) }\longrightarrow 
\underleftarrow{\lim }_{r}^{1}~A_{n-1}^{\left( r+1\right) }=\underleftarrow{%
\lim }_{r}^{1}~A_{n-1}^{\left( r\right) }
\end{equation*}%
is an epimorphism.
\end{proof}

\subsection{Bicomplexes}

\begin{definition}
\label{Def-total-complex}Let $\left( C^{\ast \ast },d,\delta \right) $ be a
cochain bicomplex, with the horizontal differential $d$ and the vertical
differential $\delta $, concentrated in the I and IV quadrant, i.e. $%
C^{st}=0 $ if $s<0$. Let $Tot\left( C\right) $ be the following cochain
complex:%
\begin{equation*}
Tot\left( C\right) ^{n}=\dprod\limits_{s+t=n}C^{st}
\end{equation*}%
with the differential $\left( \partial c\right) ^{st}=dc^{s-1,t}+\left(
-1\right) ^{s}\delta c^{s,t-1}$ where%
\begin{eqnarray*}
\left( c^{st}\right) &\in &\dprod\limits_{s+t=n-1}C^{st}=Tot\left( C\right)
^{n-1}, \\
\left( \left( \partial c\right) ^{st}\right) &\in
&\dprod\limits_{s+t=n-1}C^{st}=Tot\left( C\right) ^{n}.
\end{eqnarray*}
\end{definition}

Clearly, $\partial \partial =0$:%
\begin{eqnarray*}
\left( \partial \partial c\right) ^{st} &=&d\left( \partial c\right)
^{s-1,t}+\left( -1\right) ^{s}\delta \left( \partial c\right) ^{s,t-1}= \\
&=&d\left( dc^{s-2,t}+\left( -1\right) ^{s-1}\delta c^{s-1,t-1}\right)
+\left( -1\right) ^{s}\delta \left( dc^{s-1,t-1}+\left( -1\right) ^{s}\delta
c^{s,t-2}\right) = \\
&=&\left( -1\right) ^{s-1}d\delta c^{s-1,t-1}+\left( -1\right) ^{s}\delta
dc^{s-1,t-1}=\left( -1\right) ^{s}\left( -d\delta +\delta d\right) =0,
\end{eqnarray*}%
and $\left( Tot\left( C\right) ,\partial \right) $ is a cochain complex.

\begin{remark}
Notice that we use \textbf{products} instead of \textbf{direct sums}.
\end{remark}

Let now 
\begin{equation*}
Tot^{\left( p\right) }\left( C\right) ^{n}=\dprod\limits_{s=0}^{p}C^{s,n-s}.
\end{equation*}%
Denote by letter $\partial ^{\left( p\right) }$ the differential%
\begin{eqnarray*}
\partial ^{\left( p\right) } &:&Tot^{\left( p\right) }\left( C\right)
^{n-1}\longrightarrow Tot^{\left( p\right) }\left( C\right) ^{n-1}, \\
\left( \partial ^{\left( p\right) }c\right) ^{s,n-s} &=&dc^{s-1,n-s}+\left(
-1\right) ^{s}\delta c^{s,n-s-1}.
\end{eqnarray*}%
$\left( Tot^{\left( p\right) }\left( C\right) ,\partial ^{\left( p\right)
}\right) $ become cochain complexes, and there are natural projections%
\begin{equation*}
Tot\left( C\right) \longrightarrow Tot^{\left( p\right) }\left( C\right)
\end{equation*}%
that are epimorphisms of cochain complexes. Clearly,%
\begin{equation*}
Tot\left( C\right) =\underleftarrow{\lim }_{p}~Tot^{\left( p\right) }\left(
C\right) .
\end{equation*}

Denote by 
\begin{eqnarray*}
D^{st} &=&D_{1}^{st}=H^{s+t}\left( Tot^{\left( s\right) }\left( C\right)
\right) , \\
E^{st} &=&E_{1}^{st}=H^{s+t}\Gamma ^{\left( s\right) }\left( C\right) ,
\end{eqnarray*}%
where 
\begin{equation*}
\Gamma ^{\left( s\right) }\left( C\right) =\ker \left( Tot^{\left( s\right)
}\left( C\right) \longrightarrow Tot^{\left( s-1\right) }\left( C\right)
\right)
\end{equation*}%
The long exact sequence corresponding to the short exact sequence of
complexes%
\begin{equation*}
0\longrightarrow \Gamma ^{\left( s\right) }\left( C\right) \longrightarrow
Tot^{\left( s\right) }\left( C\right) \longrightarrow Tot^{\left( s-1\right)
}\left( C\right) \longrightarrow 0
\end{equation*}%
gives an exact couple%
\begin{equation*}
\begin{diagram}[size=3.0em,textflow]
D=D_{1} &    & \rTo^{i}_{\left( -1,1\right)}  &       & D=D_{1} \\
  & \luTo^{k}_{\left( 0,0\right)}  &     & \ldTo^{j}_{\left( 1,0\right)}  \\
  &        & E=E_{1} \\
\end{diagram}%
\end{equation*}%
of bi-graded abelian groups and homogeneous morphisms with the bi-degrees
written at the corresponding arrows.

One can construct the derived exact couples%
\begin{equation*}
\begin{diagram}[size=3.0em,textflow]
D_{r} &    & \rTo^{i_{r}}_{\left( -1,1\right)}  &       & D_{r} \\
  & \luTo^{k_{r}}_{\left( 0,0\right)}  &     & \ldTo^{j_{r}}_{\left( r,1-r\right)}  \\
  &        & E_{r} \\
\end{diagram}%
\end{equation*}%
where $D_{r}=i^{r-1}D$, $E_{r}=H\left( E_{r-1},d=j_{r-1}k_{r-1}\right) $ is
the cohomology of the complex build on the previous couple, $i_{r}$ is
induced by $i_{r-1}$, $k_{r}$ is induced by $k_{r-1}$, while $%
j_{r}=j_{r-1}\left( i_{r-1}\right) ^{-1}$.

\begin{theorem}
\label{Spectral-bicomplex}(spectral sequence of a bicomplex)

\begin{enumerate}
\item \label{Spectral-bicomplex-E1}$E_{1}^{st}=\left( H_{ver}\left( C\right)
\right) ^{st}$ where $H_{ver}$ is the cohomology in the vertical direction.

\item \label{Spectral-bicomplex-E2}$E_{2}^{st}=\left( H_{hor}H_{ver}\left(
C\right) \right) ^{st}$ where $H_{hor}$ is the cohomology in the horizontal
direction.

\item \label{Spectral-bicomplex-zero-step}$D_{2}^{0,n}%
\simeq%
E_{2}^{0,n}.$

\item \label{Spectral-bicomplex-long-exact-sequences}The groups $D_{2}^{st}$
are included in long exact sequences%
\begin{eqnarray*}
0 &\longrightarrow &E_{2}^{1,n-1}\longrightarrow
D_{2}^{1,n-1}\longrightarrow D_{2}^{0,n}\longrightarrow \\
&\longrightarrow &E_{2}^{2,n-1}\longrightarrow D_{2}^{2,n-1}\longrightarrow
D_{2}^{1,n}\longrightarrow E_{2}^{3,n-1}\longrightarrow ... \\
... &\longrightarrow &E_{2}^{s,n-1}\longrightarrow
D_{2}^{s,n-1}\longrightarrow D_{2}^{s-1,n}\longrightarrow
E_{2}^{s+1,n-1}\longrightarrow ...
\end{eqnarray*}

\item \label{Spectral-bicomplex-lim1-Tot-lim-Tot}There are short exact
sequences ($n\in \mathbb{Z}$)%
\begin{equation*}
0\longrightarrow \underleftarrow{\lim }_{s}^{1}H^{n-1}Tot^{\left( s\right)
}\left( C\right) \longrightarrow H^{n}Tot\left( C\right) \longrightarrow 
\underleftarrow{\lim }_{s}H^{n}Tot^{\left( s\right) }\left( C\right)
\longrightarrow 0.
\end{equation*}

\item \label{Spectral-bicomplex-lim1-Dr-lim-Dr}There are short exact
sequences ($r\geq 1,n\in \mathbb{Z}$)%
\begin{equation*}
0\longrightarrow \underleftarrow{\lim }_{s}^{1}D_{r}^{s,n-s-1}%
\longrightarrow H^{n}Tot\left( C\right) \longrightarrow \underleftarrow{\lim 
}_{s}D_{r}^{s,n-s}\longrightarrow 0.
\end{equation*}

\item \label{Spectral-bicomplex-isomorphism}If $C\longrightarrow C^{\prime }$
is a morphism of bicomplexes inducing isomorphisms $E_{2}^{st}\left(
C\right) 
\simeq%
E_{2}^{st}\left( C^{\prime }\right) $ for all $s$, $t$, then%
\begin{equation*}
H^{n}Tot\left( C\right) 
\simeq%
H^{n}Tot\left( C^{\prime }\right) .
\end{equation*}

\item \label{Spectral-bicomplex-convergence}If 
\begin{equation*}
\underleftarrow{\lim }_{r}^{1}~E_{r}^{st}=0
\end{equation*}%
then $E_{r}^{st}$ \textbf{completely} (in the sense of \cite%
{Bousfield-Kan-MR0365573}, IX.5.3) or \textbf{strongly} (in the sense of 
\cite{Boardman-MR1718076}, Definition 5.2) converges to $H^{s+t}\left(
Tot\left( C\right) \right) $. The latter means that $H^{n}\left( Tot\left(
C\right) \right) $ is an inverse limit: $H^{n}\left( Tot\left( C\right)
\right) 
\simeq%
$ 
\begin{equation*}
\underleftarrow{\lim }_{s}\left( ...\longrightarrow Q^{s,n-s}\longrightarrow
Q^{s-1,n-s+1}\longrightarrow ...\longrightarrow Q^{0,n}\longrightarrow
Q^{-1,n+1}=0\right)
\end{equation*}%
of epimorphisms with the kernels%
\begin{equation*}
\ker \left( Q^{s,n-s}\longrightarrow Q^{s-1,n-s+1}\right) 
\simeq%
E_{\infty }^{st}:=\underset{r}{\underleftarrow{\lim }}E_{r}^{st}.
\end{equation*}
\end{enumerate}
\end{theorem}

\begin{proof}
We could not find in the literature the proof of the statements above
directly in the form we need. The ideas of the desired proof can be found,
however, in \cite{Boardman-MR1718076}, \cite{Bousfield-Kan-MR0365573}, \cite%
{Thomason-MR826102}, \cite{Prasolov-Extraordinatory-MR1821856}, \cite%
{Prasolov-Spectral-MR1027513} and \cite{Mardesic-MR1740831}.

\begin{enumerate}
\item The $n$-th component of the kernel $\Gamma ^{\left( s\right) }\left(
C\right) $ consists of the elements%
\begin{equation*}
c=\left( 0,0,...,a\right) \in \dprod\limits_{i=0}^{s}C^{i,n-i}.
\end{equation*}%
Let us compute $\partial c$:%
\begin{equation*}
\left( \partial c\right) ^{j,n+1-j}=dc^{j-1,n+1-j}+\left( -1\right)
^{s}\delta c^{j,n-j}=0
\end{equation*}%
when $j<s$, and $=\left( -1\right) ^{s}\delta a$ when $j=s$, therefore%
\begin{equation*}
\partial c=\left( 0,0,...,\left( -1\right) ^{s}\delta a\right) .
\end{equation*}%
It means that (up to sign of the differentials) the complex $\Gamma ^{\left(
\ast \right) }\left( C\right) $ is isomorphic to the vertical $s$-th line of
the bicomplex $C^{\ast \ast }$, hence its cohomology is the vertical
cohomology of $C^{\ast \ast }$.

\item It is enough to prove that the differential%
\begin{equation*}
jk:E_{1}^{st}=H_{ver}\left( C\right) ^{st}\longrightarrow H_{ver}\left(
C\right) ^{s+1,t}
\end{equation*}%
is induced by the horizontal differential $d$. To do this, we need to
calculate $j$ and $k$. Clearly, a class%
\begin{equation*}
e^{\prime }=\left[ e\right] \in E_{1}^{st}=\frac{\ker \delta }{Im\left(
\delta \right) }
\end{equation*}%
is mapped (under the mapping $k$) to the class $ke^{\prime }=\left[
\varepsilon \right] $ of 
\begin{equation*}
\varepsilon =\left( 0,0,...,e\right) \in \ker \partial \subseteq Tot^{\left(
s\right) }\left( C\right) ^{s+t}.
\end{equation*}%
Since $Tot^{\left( s+1\right) }\left( C\right) ^{s+t}\rightarrow Tot^{\left(
s\right) }\left( C\right) ^{s+t}$ is an epimorphism, it is possible to find
an element $\theta \in Tot^{\left( s+1\right) }\left( C\right) ^{s+t}$
\textquotedblleft lying over\textquotedblright\ $\varepsilon $. We have the
right to choose%
\begin{equation*}
\theta =\left( 0,0,...,e,0\right) \in Tot^{\left( s+1\right) }\left(
C\right) ^{s+t}.
\end{equation*}%
Apply $\partial $:%
\begin{eqnarray*}
\partial \theta &=&\left( 0,0,...,\left( -1\right) ^{s}\delta e,de\right) =
\\
&=&\left( 0,0,...,0,de\right) \in \Gamma ^{\left( s+1\right) }\left(
C\right) \subseteq Tot^{\left( s+1\right) }\left( C\right) ^{s+t+1}.
\end{eqnarray*}%
Finally, $jke^{\prime }=de$, and the differential $d^{\prime }$ in the
cochain complex $E_{1}^{st}$ is induced by the horizontal differential $d$.
Therefore,%
\begin{equation*}
E_{2}^{st}=\frac{\ker \left( d^{\prime }:E_{1}^{st}\longrightarrow
E_{1}^{s+1,t}\right) }{Im\left( d^{\prime }:E_{1}^{st}\longrightarrow
E_{1}^{s+1,t}\right) }=H_{hor}H_{ver}\left( C\right) ^{st}.
\end{equation*}

\item Consider the following parts of the long exact sequences of the
derived couple%
\begin{equation*}
\begin{diagram}[size=3.0em,textflow]
D_{2} &    & \rTo^{i}_{\left( -1,1\right)}  &       & D_{2} \\
  & \luTo^{k}_{\left( 0,0\right)}  &     & \ldTo^{j}_{\left( 2,-1\right)}  \\
  &        & E_{2} \\
\end{diagram}%
\end{equation*}%
\begin{equation*}
0=D_{2}^{-2,n+1}\longrightarrow E_{2}^{0,n}\longrightarrow
D_{2}^{0,n}\longrightarrow D_{2}^{-1,n+1}=0.
\end{equation*}%
It follows that $E_{2}^{0,n}%
\simeq%
D_{2}^{0,n}$.

\item Consider again long exact sequences of the derived couple%
\begin{eqnarray*}
0 &=&D_{2}^{-1,n}\longrightarrow E_{2}^{1,n-1}\longrightarrow
D_{2}^{1,n-1}\longrightarrow D_{2}^{0,n}\longrightarrow \\
&\longrightarrow &E_{2}^{2,n-1}\longrightarrow D_{2}^{2,n-1}\longrightarrow
D_{2}^{1,n}\longrightarrow E_{2}^{3,n-1}\longrightarrow ... \\
... &\longrightarrow &E_{2}^{s,n-1}\longrightarrow
D_{2}^{s,n-1}\longrightarrow D_{2}^{s-1,n}\longrightarrow
E_{2}^{s+1,n-1}\longrightarrow ...
\end{eqnarray*}

\item Consider a short exact sequence of cochain complexes:%
\begin{equation*}
0\longrightarrow Tot\left( C\right) \longrightarrow \dprod\limits_{s\geq
0}Tot^{\left( s\right) }\left( C\right) \overset{1-p}{\longrightarrow }%
\dprod\limits_{s\geq 0}Tot^{\left( s\right) }\left( C\right) \longrightarrow
0
\end{equation*}%
where%
\begin{equation*}
p:Tot^{\left( s\right) }\left( C\right) \longrightarrow Tot^{\left(
s-1\right) }\left( C\right)
\end{equation*}%
is the natural projection. The long exact sequence of cohomologies%
\begin{eqnarray*}
... &\longrightarrow &\dprod\limits_{s\geq 0}H^{n-1}Tot^{\left( s\right)
}\left( C\right) \overset{1-p}{\longrightarrow }\dprod\limits_{s\geq
0}H^{n-1}Tot^{\left( s\right) }\left( C\right) \longrightarrow
H^{n}Tot\left( C\right) \longrightarrow \\
&\longrightarrow &\dprod\limits_{s\geq 0}H^{n}Tot^{\left( s\right) }\left(
C\right) \overset{1-p}{\longrightarrow }\dprod\limits_{s\geq
0}H^{n}Tot^{\left( s\right) }\left( C\right) \longrightarrow ...
\end{eqnarray*}%
splits into the desired pieces%
\begin{eqnarray*}
0 &\longrightarrow &\mathbf{coker}\left( \dprod\limits_{s\geq
0}H^{n-1}Tot^{\left( s\right) }\left( C\right) \overset{1-p}{\longrightarrow 
}\dprod\limits_{s\geq 0}H^{n-1}Tot^{\left( s\right) }\left( C\right) \right)
=\underleftarrow{\lim }_{s}^{1}~H^{n-1}Tot^{\left( s\right) }\left( C\right)
\longrightarrow \\
&\longrightarrow &H^{n}Tot\left( C\right) \longrightarrow \ker \left(
\dprod\limits_{s\geq 0}H^{n}Tot^{\left( s\right) }\left( C\right) \overset{%
1-p}{\longrightarrow }\dprod\limits_{s\geq 0}H^{n}Tot^{\left( s\right)
}\left( C\right) \right) =\underleftarrow{\lim }_{s}~H^{n}Tot^{\left(
s\right) }\left( C\right) \longrightarrow 0
\end{eqnarray*}

\item It follows from Theorem \ref{Mittag-Leffler} (\ref%
{Mittag-Leffler-images}) that 
\begin{eqnarray*}
&&\underleftarrow{\lim }_{s}~D_{r}^{s,n-s}%
\simeq%
\underleftarrow{\lim }_{s}~D_{1}^{s,n-s}=\underleftarrow{\lim }%
_{s}~H^{n}Tot^{\left( s\right) }\left( C\right) , \\
&&\underleftarrow{\lim }_{s}^{1}~D_{r}^{s,n-s-1}%
\simeq%
\underleftarrow{\lim }_{s}^{1}~D_{1}^{s,n-s-1}=\underleftarrow{\lim }%
_{s}^{1}~H^{n-1}Tot^{\left( s\right) }\left( C\right) .
\end{eqnarray*}

\item It follows from (\ref{Spectral-bicomplex-zero-step}) that 
\begin{equation*}
D_{2}^{0,n}\left( C\right) 
\simeq%
E_{2}^{0,n}\left( C\right) 
\simeq%
E_{2}^{0,n}\left( C^{\prime }\right) 
\simeq%
D_{2}^{0,n}\left( C^{\prime }\right) .
\end{equation*}%
The morphism $C\rightarrow C^{\prime }$ induces morphisms of exact sequences
from (\ref{Spectral-bicomplex-long-exact-sequences}) for $C$ and $C^{\prime
} $. Applying the $5$-lemma several times, one gets 
\begin{equation*}
D_{2}^{s,t}\left( C\right) 
\simeq%
D_{2}^{s,t}\left( C^{\prime }\right) .
\end{equation*}%
There is a morphism (which is an isomorphism in the second and the forth
term) of short exact sequences from (\ref{Spectral-bicomplex-lim1-Dr-lim-Dr}%
) for $r=2$. Using again the $5$-lemma, one gets the desired isomorphism 
\begin{equation*}
H^{n}Tot\left( C\right) 
\simeq%
H^{n}Tot\left( C^{\prime }\right) .
\end{equation*}

\item Consider the $r$-th derived couple%
\begin{equation*}
\begin{diagram}[size=3.0em,textflow]
D_{r} &    & \rTo^{i_{r}}_{\left( -1,1\right)}  &       & D_{r} \\
  & \luTo^{k_{r}}_{\left( 0,0\right)}  &     & \ldTo^{j_{r}}_{\left( r,1-r\right)}  \\
  &        & E_{r} \\
\end{diagram}%
\end{equation*}%
Fix $s$ and $t$. In the long exact sequence%
\begin{equation*}
...\longrightarrow D_{r}^{s-r,t+r-1}\longrightarrow
E_{r}^{st}\longrightarrow D_{r}^{st}\longrightarrow
D_{r}^{s-1,t+1}\longrightarrow ...
\end{equation*}%
the term $D_{r}^{s-r,t+r-1}$ equals zero for $r$ large enough ($r>s$), and
one gets a short exact sequence%
\begin{equation*}
0\longrightarrow E_{r}^{st}\longrightarrow D_{r}^{st}\longrightarrow
D_{r+1}^{s-1,t+1}\longrightarrow 0.
\end{equation*}%
Since $\underleftarrow{\lim }_{r}^{1}E_{r}^{st}=0$, one gets, due to Theorem %
\ref{Mittag-Leffler} (\ref{Mittag-Leffler-exact-sequence}, \ref%
{Mittag-Leffler-images}), a short exact sequence%
\begin{equation*}
0\longrightarrow \underleftarrow{\lim }_{r}~E_{r}^{st}\longrightarrow 
\underleftarrow{\lim }_{r}~D_{r}^{st}\longrightarrow \underleftarrow{\lim }%
_{r}~D_{r}^{s-1,t+1}\longrightarrow 0,
\end{equation*}%
and an isomorphism%
\begin{equation*}
\underleftarrow{\lim }_{r}^{1}~D_{r}^{st}%
\simeq%
\underleftarrow{\lim }_{r}^{1}~D_{r}^{s-1,t+1}.
\end{equation*}%
Varying $s$, one gets%
\begin{equation*}
\underleftarrow{\lim }_{r}^{1}~D_{r}^{st}%
\simeq%
\underleftarrow{\lim }_{r}^{1}~D_{r}^{s-1,t+1}%
\simeq%
...%
\simeq%
\underleftarrow{\lim }_{r}^{1}~D_{r}^{0,s+t}%
\simeq%
\underleftarrow{\lim }_{r}^{1}~D_{r}^{-1,s+t+1}=0.
\end{equation*}%
Let us denote 
\begin{equation*}
Q^{s,n-s}:=\underleftarrow{\lim }_{r}D_{r}^{s.n-s}=D_{\infty }^{s.n-s}.
\end{equation*}%
It follows from Theorem \ref{Mittag-Leffler} (\ref{Mittag-Leffler-lim-lim})
that%
\begin{equation*}
\underleftarrow{\lim }_{s}H^{n}\left( Tot^{\left( s\right) }\left( C\right)
\right) 
\simeq%
\underleftarrow{\lim }_{s}~\underleftarrow{\lim }_{r}~D_{r}^{s,n-s}%
\simeq%
\underleftarrow{\lim }_{s}~Q^{s,n-s}.
\end{equation*}%
Since 
\begin{equation*}
\left( ...\longrightarrow Q^{s,n-s}\longrightarrow
Q^{s-1,n-s+1}\longrightarrow ...\right)
\end{equation*}%
is a tower of epimorphisms, $\underleftarrow{\lim }_{s}^{1}~Q^{s,n-s}=0$,
due to Theorem \ref{Mittag-Leffler} (\ref{Mittag-Leffler-epimorphisms}). It
follows from Theorem \ref{Mittag-Leffler} (\ref{Mittag-Leffler-lim-lim1}),
since $\underleftarrow{\lim }_{s}^{1}~D_{r}^{s,n-s-1}=0$, that%
\begin{equation*}
\underleftarrow{\lim }_{s}^{1}~H^{n-1}\left( Tot^{\left( s\right) }\left(
C\right) \right) =0.
\end{equation*}%
The short exact sequence (\ref{Spectral-bicomplex-lim1-Tot-lim-Tot})\ gives
the desired isomorphism%
\begin{equation*}
H^{n}\left( Tot\left( C\right) \right) 
\simeq%
\underleftarrow{\lim }_{s}~H^{n}\left( Tot^{\left( s\right) }\left( C\right)
\right) 
\simeq%
\underleftarrow{\lim }_{s}~Q^{s,n-s}.
\end{equation*}%
We have proved that $H^{n}\left( Tot\left( C\right) \right) $ is isomorphic
to the inverse limit of the tower $\left( Q^{s,n-s}\right) _{s=0}^{\infty }$
of epimorphisms with kernels $E_{\infty }^{s,n-s}$, hence $E_{r}^{st}$
converges completely (in the sense of \cite{Bousfield-Kan-MR0365573},
IX.5.3) or strongly (in the sense of \cite{Boardman-MR1718076}, Definition
5.2) to $H^{s+t}\left( Tot\left( C\right) \right) $.
\end{enumerate}
\end{proof}

\subsection{Homotopy inverse limits}

On homotopy inverse limits of topological spaces, see \cite%
{Bousfield-Kan-MR0365573}, Chapter XI. Here we define and investigate the
homotopy inverse limits for the diagrams of chain complexes.

\begin{definition}
\label{Def-homotopy-limit}Let $\mathbf{I}$ be a small category, and let $C:%
\mathbf{I}\longrightarrow \mathbf{CHAIN}\left( \mathbb{Z}\right) $ be a
functor to the category of chain complexes of abelian groups. Negating the
indices, one gets a functor to the category of \textbf{co}chain complexes.
Take the \textbf{cosimplicial replacement} (\cite{Bousfield-Kan-MR0365573},
XI.5.1), which is a cosimplicial cochain complex, and, finally, a cochain
bicomplex%
\begin{equation*}
\left( R\mathbf{C}\right) ^{st}=\dprod\limits_{\left( i_{0}\rightarrow
i_{1}\rightarrow ...\rightarrow i_{s}\right) \in \mathbf{I}}C_{-t}\left(
i_{s}\right)
\end{equation*}%
with the horizontal differentials%
\begin{equation*}
d:\left( R\mathbf{C}\right) ^{s-1,t}\longrightarrow \left( R\mathbf{C}%
\right) ^{s,t},
\end{equation*}%
\begin{equation*}
dc\left( i_{0}\rightarrow i_{1}\rightarrow ...\rightarrow i_{s}\right) =
\end{equation*}%
\begin{equation*}
\dsum\limits_{k=0}^{s-1}\left( -1\right) ^{k}c\left( i_{0}\rightarrow
...\longrightarrow \widehat{i_{k}}\longrightarrow ...\rightarrow
i_{s}\right) +\left( -1\right) ^{s}C\left( i_{s-1}\longrightarrow
i_{s}\right) c\left( i_{0}\rightarrow ...\rightarrow i_{s-1}\right) ,
\end{equation*}%
and the vertical differentials%
\begin{eqnarray*}
\delta &:&\left( R\mathbf{C}\right) ^{s,t-1}\longrightarrow \left( R\mathbf{C%
}\right) ^{s,t}, \\
\delta c\left( i_{0}\rightarrow i_{1}\rightarrow ...\rightarrow i_{s}\right)
&=&d_{-t}\left( i_{s}\right) c\left( i_{0}\rightarrow i_{1}\rightarrow
...\rightarrow i_{s}\right)
\end{eqnarray*}%
where%
\begin{equation*}
d_{-t}\left( i_{s}\right) :C_{-t+1}\left( i_{s}\right) \longrightarrow
C_{-t}\left( i_{s}\right)
\end{equation*}%
is the differential in the chain complex $C_{\ast }\left( i_{s}\right) $.
Take the total complex $Tot\left( R\mathbf{C}\right) $ (using products, as
in Definition \ref{Def-total-complex}), and negate the indices again. The
resulting \textbf{chain} complex is called the homotopy inverse limit $%
\underleftarrow{\mathbf{holim}}_{i}~\mathbf{C}$ of the functor $\mathbf{C}$:%
\begin{equation*}
\left( \underleftarrow{\mathbf{holim}}_{i}~\mathbf{C}\right)
_{n}=\dprod\limits_{s}\dprod\limits_{\left( i_{0}\rightarrow
i_{1}\rightarrow ...\rightarrow i_{s}\right) \in \mathbf{I}}C_{n+s}\left(
i_{s}\right)
\end{equation*}%
with the differential%
\begin{equation*}
\partial :\left( \underleftarrow{\mathbf{holim}}_{i\in \mathbf{I}}~\mathbf{C}%
\right) _{n+1}\longrightarrow \left( \underleftarrow{\mathbf{holim}}_{i\in 
\mathbf{I}}~\mathbf{C}\right) _{n}
\end{equation*}%
given by%
\begin{equation*}
\partial c\left( i_{0}\rightarrow i_{1}\rightarrow ...\rightarrow
i_{s}\right) =
\end{equation*}%
\begin{eqnarray*}
&&\dsum\limits_{k=0}^{s-1}\left( -1\right) ^{k}c\left( i_{0}\rightarrow
...\longrightarrow \widehat{i_{k}}\longrightarrow ...\rightarrow
i_{s}\right) + \\
&&+\left( -1\right) ^{s}C\left( i_{s-1}\longrightarrow i_{s}\right) c\left(
i_{0}\rightarrow ...\rightarrow i_{s-1}\right) +\left( -1\right)
^{s}d_{n+s}\left( i_{s}\right) c\left( i_{0}\rightarrow i_{1}\rightarrow
...\rightarrow i_{s}\right) .
\end{eqnarray*}
\end{definition}

\begin{definition}
Let $\varphi :\mathbf{J}\rightarrow \mathbf{I}$ be a \textbf{cofinal}
functor (\textbf{left cofinal} in \cite{Bousfield-Kan-MR0365573}, XI.9.1),
and let $C:\mathbf{I}\longrightarrow CHAIN\left( \mathbb{Z}\right) $ be a
functor. The natural homomorphisms%
\begin{eqnarray*}
\dprod\limits_{\left( i_{0}\rightarrow i_{1}\rightarrow ...\rightarrow
i_{s}\right) \in \mathbf{I}}C\left( i_{s}\right) &\longrightarrow
&\dprod\limits_{\left( j_{0}\rightarrow j_{1}\rightarrow ...\rightarrow
j_{s}\right) \in \mathbf{J}}C\left( \varphi \left( j_{s}\right) \right) , \\
c &\longmapsto &c^{\prime }, \\
c^{\prime }\left( j_{0}\rightarrow j_{1}\rightarrow ...\rightarrow
j_{s}\right) &=&c\left( \varphi \left( j_{0}\right) \rightarrow \varphi
\left( j_{1}\right) \rightarrow ...\rightarrow \varphi \left( j_{s}\right)
\right)
\end{eqnarray*}%
induce the morphism of chain complexes%
\begin{equation*}
\varphi ^{\ast }:\underleftarrow{\mathbf{holim}}_{i\in \mathbf{I}}\mathbf{C}%
\longrightarrow \underleftarrow{\mathbf{holim}}_{j\in \mathbf{J}}\left( 
\mathbf{C}\circ \varphi \right) .
\end{equation*}
\end{definition}

\begin{remark}
Notice that we \textbf{do not} require that the index categories $\mathbf{I}$
and $\mathbf{J}$ are cofiltrant.
\end{remark}

\begin{theorem}
\label{Spectral-holimit}(spectral sequence of a homotopy inverse limit). Let 
$E_{r}^{st}$ be the spectral sequence from Theorem \ref{Spectral-bicomplex}
for the bicomplex $R\mathbf{C}$. Then:

\begin{enumerate}
\item \label{Spectral-holimit-E1}%
\begin{equation*}
E_{1}^{st}=\dprod\limits_{\left( i_{0}\rightarrow i_{1}\rightarrow
...\rightarrow i_{s}\right) \in \mathbf{I}}H_{-t}C\left( i_{s}\right) .
\end{equation*}

\item \label{Spectral-holimit-E2}%
\begin{equation*}
E_{2}^{st}=\underleftarrow{\lim }_{i\in \mathbf{I}}^{s}H_{-t}\left( C\left(
i\right) \right) .
\end{equation*}

\item \label{Spectral-holimit-zero-step}$D_{2}^{0,n}%
\simeq%
E_{2}^{0,n}%
\simeq%
\underleftarrow{\lim }_{i\in \mathbf{I}}H_{-n}\left( C\left( i\right)
\right) $.

\item \label{Spectral-holimit-exact-sequences}The groups $D_{2}^{st}$ are
included in long exact sequences%
\begin{eqnarray*}
0 &\longrightarrow &E_{2}^{1,n-1}\longrightarrow
D_{2}^{1,n-1}\longrightarrow D_{2}^{0,n}\longrightarrow \\
&\longrightarrow &E_{2}^{2,n-1}\longrightarrow D_{2}^{2,n-1}\longrightarrow
D_{2}^{1,n}\longrightarrow E_{2}^{3,n-1}\longrightarrow ... \\
... &\longrightarrow &E_{2}^{s,n-1}\longrightarrow
D_{2}^{s,n-1}\longrightarrow D_{2}^{s-1,n}\longrightarrow
E_{2}^{s+1,n-1}\longrightarrow ...
\end{eqnarray*}

\item \label{Spectral-holimit-lim1-Tot-lim-Tot}There are short exact
sequences ($n\in \mathbb{Z}$)%
\begin{equation*}
0\longrightarrow \underleftarrow{\lim }_{s}^{1}H_{n+1}Tot^{\left( s\right)
}\left( C\right) \longrightarrow H_{n}\left( \underleftarrow{\mathbf{holim}}%
_{i\in \mathbf{I}}~\mathbf{C}\right) \longrightarrow \underleftarrow{\lim }%
_{s}H_{n}Tot^{\left( s\right) }\left( C\right) \longrightarrow 0.
\end{equation*}

\item \label{Spectral-holimit-lim1-Dr-lim-Dr}There are short exact sequences
($r\geq 1,n\in \mathbb{Z}$)%
\begin{equation*}
0\longrightarrow \underleftarrow{\lim }_{s}^{1}D_{r}^{s,-n-s-1}%
\longrightarrow H_{n}\left( \underleftarrow{\mathbf{holim}}_{i\in \mathbf{I}%
}~\mathbf{C}\right) \longrightarrow \underleftarrow{\lim }%
_{s}D_{r}^{s,-n-s}\longrightarrow 0.
\end{equation*}

\item \label{Spectral-holimit-isomorphism}If $\mathbf{C}\longrightarrow 
\mathbf{C}^{\prime }$ is a morphism of functors inducing isomorphisms $%
E_{2}^{st}\left( \mathbf{C}\right) 
\simeq%
E_{2}^{st}\left( \mathbf{C}^{\prime }\right) $ for all $s$, $t$, then%
\begin{equation*}
H_{n}\left( \underleftarrow{\mathbf{holim}}_{i\in \mathbf{I}}~\mathbf{C}%
\right) 
\simeq%
H_{n}\left( \underleftarrow{\mathbf{holim}}_{i\in \mathbf{I}}~\mathbf{C}%
^{\prime }\right) .
\end{equation*}

\item \label{Spectral-holimit-cofinality}Let $\varphi :\mathbf{J}\rightarrow 
\mathbf{I}$ be a cofinal functor, and let 
\begin{equation*}
\varphi ^{\ast }:\underleftarrow{\mathbf{holim}}_{i\in \mathbf{I}}~\mathbf{C}%
\longrightarrow \underleftarrow{\mathbf{holim}}_{j\in \mathbf{J}}\left( 
\mathbf{C}\circ \varphi \right)
\end{equation*}%
be the corresponding morphism of complexes. Then $\varphi ^{\ast }$ is a
weak equivalence (induces an isomorphism of homologies).

\item \label{Spectral-holimit-convergence}If 
\begin{equation*}
\underleftarrow{\lim }_{r}^{1}~E_{r}^{st}=0
\end{equation*}%
then $E_{r}^{st}$ \textbf{completely} (in the sense of \cite%
{Bousfield-Kan-MR0365573}, IX.5.3) or \textbf{strongly} (in the sense of 
\cite{Boardman-MR1718076}, Definition 5.2) converges to $H_{-s-t}\left( 
\underleftarrow{\mathbf{holim}}_{i\in \mathbf{I}}~\mathbf{C}\right) $. The
latter means that $H_{n}\left( \underleftarrow{\mathbf{holim}}_{i\in \mathbf{%
I}}C\right) $ is isomorphic to an inverse limit%
\begin{equation*}
\underleftarrow{\lim }_{s}~\left( ...\longrightarrow
Q^{s,-n-s}\longrightarrow Q^{s-1,-n-s+1}\longrightarrow ...\longrightarrow
Q^{0,-n}\longrightarrow Q^{-1,-n+1}=0\right)
\end{equation*}%
of epimorphisms with the kernels%
\begin{equation*}
\ker \left( Q^{s,-n-s}\longrightarrow Q^{s-1,-n-s+1}\right) 
\simeq%
E_{\infty }^{s,-n-s}:=\underleftarrow{\lim }_{r}E_{r}^{s,-n-s}.
\end{equation*}
\end{enumerate}
\end{theorem}

\begin{proof}
~

\begin{enumerate}
\item Follows from Theorem \ref{Spectral-bicomplex} (\ref%
{Spectral-bicomplex-E1}).

\item Follows from Theorem \ref{Spectral-bicomplex} (\ref%
{Spectral-bicomplex-E2}) and Proposition XI.6.2 in \cite%
{Bousfield-Kan-MR0365573}.

\item Follows from Theorem \ref{Spectral-bicomplex} (\ref%
{Spectral-bicomplex-zero-step}).

\item Follows from Theorem \ref{Spectral-bicomplex} (\ref%
{Spectral-bicomplex-long-exact-sequences}).

\item Negate indices in Theorem \ref{Spectral-bicomplex} (\ref%
{Spectral-bicomplex-lim1-Tot-lim-Tot}).

\item Negate indices in Theorem \ref{Spectral-bicomplex} (\ref%
{Spectral-bicomplex-lim1-Dr-lim-Dr}).

\item Negate indices in Theorem \ref{Spectral-bicomplex} (\ref%
{Spectral-bicomplex-isomorphism}).

\item A cofinal functor induces an isomorphism of higher limits. This is a
rather well-known fact. It can be proven similarly to the Cofinality Theorem
(\cite{Bousfield-Kan-MR0365573}, XI.9.2). Moreover, the statement follows
from (\cite{Bousfield-Kan-MR0365573}, XI.9.2 and XI.7.2). Therefore%
\begin{equation*}
E_{2}^{st}\left( \underleftarrow{\mathbf{holim}}_{i\in \mathbf{I}}~\mathbf{C}%
\right) \longrightarrow E_{2}^{st}\left( \underleftarrow{\mathbf{holim}}%
_{j\in \mathbf{J}}\left( \mathbf{C}\circ \varphi \right) \right)
\end{equation*}%
is an isomorphism for all $s$, $t$. The desired isomorphism of homologies
follows from (\ref{Spectral-holimit-isomorphism}).

\item Negate indices in Theorem \ref{Spectral-bicomplex} (\ref%
{Spectral-bicomplex-convergence}).
\end{enumerate}
\end{proof}

\begin{definition}
\label{Def-holim-pro-CHAIN}Let $\mathbf{C}_{\ast }=\left( C_{\ast i}\right)
_{i\in \mathbf{I}}\in \mathbf{Pro}\left( \mathbf{CHAIN}\left( \mathbb{Z}%
\right) \right) $ be a pro-complex. Define its homotopy inverse limit $%
\underleftarrow{\mathbf{holim}}$ as follows:%
\begin{equation*}
\underleftarrow{\mathbf{holim}}~\mathbf{C}_{\ast }:=\underleftarrow{\mathbf{%
holim}}_{i\in \mathbf{I}}~C_{\ast i}.
\end{equation*}
\end{definition}

\begin{remark}
It follows from Proposition \ref{Prop-holim-is-well-defined} below that the
homotopy inverse limit of a pro-complex is well defined up to weak
equivalence.
\end{remark}

\begin{proposition}
\label{Prop-holim-is-well-defined}The homotopy inverse limit from Definition %
\ref{Def-holim-pro-CHAIN} is a well-defined functor%
\begin{equation*}
\underleftarrow{\mathbf{holim}}:\mathbf{Pro}\left( \mathbf{CHAIN}\left( 
\mathbb{Z}\right) \right) \longrightarrow \mathbf{Ho}\left( \mathbf{CHAIN}%
\left( \mathbb{Z}\right) \right)
\end{equation*}%
where $\mathbf{Ho}\left( \mathbf{CHAIN}\left( \mathbb{Z}\right) \right) $ is
the category of fractions of the category $\mathbf{CHAIN}\left( \mathbb{Z}%
\right) $ with respect to weak equivalences of complexes.
\end{proposition}

\begin{proof}
The functor is well-defined on the category $\mathbf{Inv}\left( \mathbf{CHAIN%
}\left( \mathbb{Z}\right) \right) $. In order to use Theorem \ref%
{Th-description-Pro(C)}, one needs only to check that cofinal morphisms are
mapped into weak equivalences. The latter fact, however, follows from
Theorem \ref{Spectral-holimit} (\ref{Spectral-holimit-cofinality}).
\end{proof}

\begin{theorem}
\label{Th-Spectral-holimit-pro-complex}(spectral sequence of a pro-complex).
Let $\mathbf{C}\in \mathbf{Pro}\left( \mathbf{CHAIN}\left( \mathbb{Z}\right)
\right) $ be a chain pro-complex. Then there exists a spectral sequence $%
E_{r}^{st}\left( \mathbf{C}\right) $, natural on $\mathbf{C}$ from $E_{2}$
on, such that:

\begin{enumerate}
\item \label{Th-Spectral-holimit-pro-complex-E2}%
\begin{equation*}
E_{2}^{st}=\underleftarrow{\lim }^{s}~\mathbf{H}_{-t}\left( \mathbf{C}%
\right) .
\end{equation*}

\item \label{Th-Spectral-holimit-pro-complex-D2-E2}$D_{2}^{0,n}%
\simeq%
E_{2}^{0,n}%
\simeq%
\underleftarrow{\lim }\mathbf{H}_{-n}\left( \mathbf{C}\right) $.

\item \label{Th-Spectral-holimit-pro-complex-long-D2-E2}The groups $%
D_{2}^{st}$ are included in long exact sequences%
\begin{eqnarray*}
0 &\longrightarrow &E_{2}^{1,n-1}\longrightarrow
D_{2}^{1,n-1}\longrightarrow D_{2}^{0,n}\longrightarrow \\
&\longrightarrow &E_{2}^{2,n-1}\longrightarrow D_{2}^{2,n-1}\longrightarrow
D_{2}^{1,n}\longrightarrow E_{2}^{3,n-1}\longrightarrow ... \\
... &\longrightarrow &E_{2}^{s,n-1}\longrightarrow
D_{2}^{s,n-1}\longrightarrow D_{2}^{s-1,n}\longrightarrow
E_{2}^{s+1,n-1}\longrightarrow ...
\end{eqnarray*}

\item \label{Th-Spectral-holimit-pro-complex-short-lim-D2-lim1-D2}There are
short exact sequences ($r\geq 2,n\in \mathbb{Z}$)%
\begin{equation*}
0\longrightarrow \underleftarrow{\lim }_{r}^{1}D_{r}^{s,-n-s-1}%
\longrightarrow H_{n}\left( \underleftarrow{\mathbf{holim}}~\mathbf{C}%
\right) \longrightarrow \underleftarrow{\lim }_{r}D_{r}^{s,-n-s}%
\longrightarrow 0.
\end{equation*}

\item \label{Th-Spectral-holimit-pro-complex-isomorphism}If $\mathbf{C}%
\longrightarrow \mathbf{C}^{\prime }$ is a morphism of pro-complexes
inducing isomorphisms%
\begin{equation*}
E_{2}^{st}\left( \mathbf{C}\right) 
\simeq%
E_{2}^{st}\left( \mathbf{C}^{\prime }\right)
\end{equation*}%
for all $s$, $t$, then%
\begin{equation*}
H_{n}\left( \underleftarrow{\mathbf{holim}}~\mathbf{C}\right) 
\simeq%
H_{n}\left( \underleftarrow{\mathbf{holim}}~\mathbf{C}^{\prime }\right) .
\end{equation*}

\item \label{Th-Spectral-holimit-pro-complex-convergence}If 
\begin{equation*}
\underleftarrow{\lim }_{r}^{1}E_{r}^{st}=0
\end{equation*}%
then $E_{r}^{st}$ \textbf{completely} (in the sense of \cite%
{Bousfield-Kan-MR0365573}, IX.5.3) or \textbf{strongly} (in the sense of 
\cite{Boardman-MR1718076}, Definition 5.2) converges to $H_{-s-t}\left( 
\underleftarrow{\mathbf{holim}}~\mathbf{C}\right) $. The latter means that $%
H_{n}\left( \underleftarrow{\mathbf{holim}}~\mathbf{C}\right) 
\simeq%
$ 
\begin{equation*}
\underleftarrow{\lim }_{s}\left( ...\longrightarrow
Q^{s,-n-s}\longrightarrow Q^{s-1,-n-s+1}\longrightarrow ...\longrightarrow
Q^{0,-n}\longrightarrow Q^{-1,-n+1}=0\right)
\end{equation*}%
of epimorphisms with kernels%
\begin{equation*}
\ker \left( Q^{s,-n-s}\longrightarrow Q^{s-1,-n-s+1}\right) 
\simeq%
E_{\infty }^{s,-n-s}:=\underleftarrow{\lim }_{r}E_{r}^{s,-n-s}.
\end{equation*}
\end{enumerate}
\end{theorem}

\begin{proof}
~

\begin{enumerate}
\item The formula for $E_{2}^{st}$ follows from Theorem \ref%
{Spectral-holimit} (\ref{Spectral-holimit-E2}). To check naturality, use
Theorem \ref{Th-description-Pro(C)} and the facts that both $\mathbf{H}_{-t}$
and $\lim^{s}$ map cofinal morphisms into isomorphisms.

\item Follows from Theorem \ref{Spectral-holimit} (\ref%
{Spectral-holimit-zero-step}).

\item Follows from Theorem \ref{Spectral-holimit} (\ref%
{Spectral-holimit-exact-sequences}).

\item Follows from Theorem \ref{Spectral-holimit} (\ref%
{Spectral-holimit-lim1-Dr-lim-Dr}).

\item It follows from (\ref{Spectral-holimit-zero-step}) that 
\begin{equation*}
D_{2}^{0,n}\left( \mathbf{C}\right) 
\simeq%
E_{2}^{0,n}\left( \mathbf{C}\right) 
\simeq%
E_{2}^{0,n}\left( \mathbf{C}^{\prime }\right) 
\simeq%
D_{2}^{0,n}\left( \mathbf{C}^{\prime }\right) .
\end{equation*}%
The morphism $\mathbf{C}\rightarrow \mathbf{C}^{\prime }$ induces morphisms
of exact sequences from (\ref{Spectral-holimit-exact-sequences}) for $%
\mathbf{C}$ and $\mathbf{C}^{\prime }$. Applying the $5$-lemma several
times, one gets 
\begin{equation*}
D_{2}^{s,t}\left( \mathbf{C}\right) 
\simeq%
D_{2}^{s,t}\left( \mathbf{C}^{\prime }\right) .
\end{equation*}%
There is a morphism (which is an isomorphism in the second and the forth
term) of short exact sequences from (\ref{Spectral-holimit-lim1-Dr-lim-Dr})
for $r=2$. Using again the $5$-lemma, one gets the desired isomorphism 
\begin{equation*}
H^{n}\left( \underleftarrow{\mathbf{holim}}~\mathbf{C}\right) 
\simeq%
H^{n}\left( \underleftarrow{\mathbf{holim}}~\mathbf{C}^{\prime }\right) .
\end{equation*}

\item Follows from Theorem \ref{Spectral-holimit} (\ref%
{Spectral-holimit-convergence}).
\end{enumerate}
\end{proof}

\section{Shape homology}

\subsection{Pro-homology}

Let $\mathbf{Pro}\left( \mathbf{TOP}\right) $, $\mathbf{Pro}\left( \mathbf{%
HTOP}\right) $, $\mathbf{Pro}\left( \mathbf{POL}\right) $, and $\mathbf{Pro}%
\left( \mathbf{HPOL}\right) $ be the pro-categories from Example \ref%
{Ex-pro-categories} (\ref{Ex-Pro(TOP)}).

\begin{definition}
\label{Def-pro-homology}Let $G$ be an abelian group. Given a pro-space $%
\mathbf{X=}\left( X_{i}\right) _{i\in \mathbf{I}}\in \mathbf{Pro}\left( 
\mathbf{TOP}\right) $, or a pro-homotopy type $\mathbf{X=}\left(
X_{i}\right) _{i\in \mathbf{I}}\in \mathbf{Pro}\left( \mathbf{HTOP}\right) $%
, define%
\begin{equation*}
\mathbf{C}_{\ast }\left( \mathbf{X},G\right) =\left( C_{\ast }\left(
X_{i},G\right) \right) _{i\in \mathbf{I}}
\end{equation*}%
to be the corresponding chain \textbf{pro-complex} (see Example \ref%
{Ex-pro-categories} (\ref{Ex-Pro(CHAIN)})) if $\mathbf{X}\in \mathbf{Pro}%
\left( \mathbf{TOP}\right) $ or a \textbf{family} of chain complexes (if $%
\mathbf{X}\in \mathbf{Pro}\left( \mathbf{HTOP}\right) $) where $C_{\ast
}\left( X_{i},G\right) $ is the singular chain complex for $X_{i}$ with
coefficients in $G$. Let $\mathbf{H}_{n}\left( \mathbf{X},G\right) $ be the
corresponding pro-homology group (or a family of abelian groups):%
\begin{equation*}
\mathbf{H}_{n}\left( \mathbf{X},G\right) =\left( H_{n}\left( C_{\ast }\left(
X_{i},G\right) \right) \right) _{i\in \mathbf{I}}.
\end{equation*}%
For a topological space $X$, let $X\rightarrow \mathbf{X=}\left(
X_{i}\right) _{i\in \mathbf{I}}$ be an $HPol$-expansion (\cite%
{Mardesic-Segal-MR676973}, \S I.2.1), or a strong polyhedral expansion (\cite%
{Mardesic-MR1740831}, \S 7.1.), i.e. $\mathbf{X}\in \mathbf{Pro}\left( 
\mathbf{HPOL}\right) $ or $\mathbf{X}\in \mathbf{Pro}\left( \mathbf{POL}%
\right) $. Let finally 
\begin{equation*}
\mathbf{H}_{n}\left( X,G\right) :=\mathbf{H}_{n}\left( \mathbf{X},G\right) .
\end{equation*}%
It follows from \cite{Mardesic-Segal-MR676973}, \S II.3.1, that $\mathbf{H}%
_{n}\left( \mathbf{X},G\right) $ are well defined abelian pro-groups that do
not depend on the choice of an expansion $\mathbf{X}$.
\end{definition}

\subsection{Strong homology}

We define strong homology as in \cite{Prasolov-Extraordinatory-MR1821856},
Definition 3.1.3. The definition is equivalent to that in \cite%
{Mardesic-MR1740831}, Chapter 17 and 18, see \cite%
{Prasolov-Extraordinatory-MR1821856}, Theorem 3(a).

\begin{definition}
\label{Def-strong-homology}Given a pro-space $\mathbf{X=}\left( X_{i}\right)
_{i\in \mathbf{I}}\in \mathbf{Pro}\left( \mathbf{TOP}\right) $, let 
\begin{equation*}
\mathbf{C}_{\ast }\left( \mathbf{X},G\right) =\left( C_{\ast }\left(
X_{i},G\right) \right) _{i\in \mathbf{I}}\in \mathbf{Pro}\left( \mathbf{CHAIN%
}\left( \mathbb{Z}\right) \right)
\end{equation*}%
be the corresponding singular chain pro-complex from Definition \ref%
{Def-pro-homology}. Define%
\begin{equation*}
\overline{H}_{n}\left( \mathbf{X},G\right) :=H_{n}\left( \underleftarrow{%
\mathbf{holim}}\left( \mathbf{C}_{\ast }\left( \mathbf{X},G\right) \right)
\right)
\end{equation*}%
where $\underleftarrow{\mathbf{holim}}$ is the homotopy inverse limit from
Definition \ref{Def-holim-pro-CHAIN}. Given a topological space $X$, let $%
X\rightarrow \mathbf{X=}\left( X_{i}\right) _{i\in \mathbf{I}}$ be a strong
polyhedral expansion. The homology of $\underleftarrow{\mathbf{holim}}\left( 
\mathbf{C}_{\ast }\left( \mathbf{X},G\right) \right) $ is called the strong
homology of $X$ with coefficients in $G$:%
\begin{equation*}
\overline{H}_{n}\left( X,G\right) :=H_{n}\left( \underleftarrow{\mathbf{holim%
}}\left( \mathbf{C}_{\ast }\left( \mathbf{X},G\right) \right) \right) .
\end{equation*}
\end{definition}

\begin{remark}
Strong homology is strong shape invariant (see \cite{Mardesic-MR1740831},
Theorem 18.12). Compare with Proposition \ref{Prop-invariance} (\ref%
{Prop-invariance-balanced-strong}).
\end{remark}

\begin{remark}
$\overline{H}_{n}$ is defined for all $n\in \mathbb{Z}$ (the negative values
of $n$ included).
\end{remark}

\begin{theorem}
\label{Th-Spectral-strong-homology}(Spectral sequence for strong homology).
Let $X$ be a topological space as an object of the strong shape category $%
\mathbf{SSh}$. Then there exists a spectral sequence $E_{r}^{st}\left(
X\right) $, natural on $X\in \mathbf{SSh}$ from $E_{2}$ on, such that:

\begin{enumerate}
\item \label{Th-Spectral-strong-homology-E2}%
\begin{equation*}
E_{2}^{st}=\underleftarrow{\lim }^{s}~\mathbf{H}_{-t}\left( X,G\right) .
\end{equation*}

\item \label{Th-Spectral-strong-homology-zero-step}$D_{2}^{0,n}%
\simeq%
E_{2}^{0,n}%
\simeq%
\check{H}_{-n}\left( X,G\right) $ where $\check{H}_{\ast }$ is 
\u{C}ech \ 
homology.

\item \label{Th-Spectral-strong-homology-exact-sequences}The groups $%
D_{2}^{st}$ are included in long exact sequences%
\begin{eqnarray*}
0 &\longrightarrow &E_{2}^{1,n-1}\longrightarrow
D_{2}^{1,n-1}\longrightarrow D_{2}^{0,n}\longrightarrow \\
&\longrightarrow &E_{2}^{2,n-1}\longrightarrow D_{2}^{2,n-1}\longrightarrow
D_{2}^{1,n}\longrightarrow E_{2}^{3,n-1}\longrightarrow ... \\
... &\longrightarrow &E_{2}^{s,n-1}\longrightarrow
D_{2}^{s,n-1}\longrightarrow D_{2}^{s-1,n}\longrightarrow
E_{2}^{s+1,n-1}\longrightarrow ...
\end{eqnarray*}

\item \label{Th-Spectral-strong-homology-lim1-Dr-lim-Dr}There are short
exact sequences ($r\geq 2,n\in \mathbb{Z}$)%
\begin{equation*}
0\longrightarrow \underleftarrow{\lim }_{r}^{1}D_{r}^{s,-n-s+1}%
\longrightarrow \overline{H}_{n}\left( X,G\right) \longrightarrow 
\underleftarrow{\lim }_{r}D_{r}^{s,-n-s}\longrightarrow 0.
\end{equation*}

\item \label{Th-Spectral-strong-homology-convergence}If 
\begin{equation*}
\underleftarrow{\lim }_{r}^{1}~E_{r}^{st}=0
\end{equation*}%
then $E_{r}^{st}$ \textbf{completely} (in the sense of \cite%
{Bousfield-Kan-MR0365573}, IX.5.3) or \textbf{strongly} (in the sense of 
\cite{Boardman-MR1718076}, Definition 5.2) converges to $\overline{H}%
_{-n}\left( X,G\right) $. The latter means that $\overline{H}_{-n}\left(
X,G\right) 
\simeq%
$ 
\begin{equation*}
\underleftarrow{\lim }_{s}\left( ...\longrightarrow
Q^{s,-n-s}\longrightarrow Q^{s-1,-n-s+1}\longrightarrow ...\longrightarrow
Q^{0,-n}\longrightarrow Q^{-1,-n+1}=0\right)
\end{equation*}%
of epimorphisms with kernels%
\begin{equation*}
\ker \left( Q^{s,-n-s}\longrightarrow Q^{s-1,-n-s+1}\right) 
\simeq%
E_{\infty }^{s,-n-s}:=\underleftarrow{\lim }_{r}E_{r}^{s,-n-s}.
\end{equation*}
\end{enumerate}
\end{theorem}

\begin{proof}
Most statements of this Theorem were proved in \cite%
{Prasolov-Spectral-MR1027513}. Let $X\rightarrow \mathbf{X=}\left(
X_{i}\right) _{i\in \mathbf{I}}$ be a strong polyhedral expansion. Apply
Theorem \ref{Th-Spectral-holimit-pro-complex} to the pro-complex $\left(
C_{\ast }\left( X_{i},G\right) \right) _{i\in \mathbf{I}}$ where $C_{\ast
}\left( X_{i},G\right) $ is the singular complex for $X_{i}$ with
coefficients in $G$.

To check that the spectral sequence is natural on $E_{2}$ on, apply Theorem %
\ref{Th-description-SSh}. It is enough to check that both cofinal morphisms
and level equivalences induce an isomorphism on $E_{2}$.

If $\mathbf{X}\rightarrow \mathbf{X}^{\prime }$ is cofinal, then the
corresponding morphism $\mathbf{C}_{\ast }\left( \mathbf{X},G\right)
\rightarrow \mathbf{C}_{\ast }\left( \mathbf{X}^{\prime },G\right) $ is
isomorphisms in $\mathbf{Pro}\left( \mathbf{CHAIN}\left( \mathbb{Z}\right)
\right) $, and induces therefore an isomorphism 
\begin{equation*}
E_{2}^{st}\left( \mathbf{X}\right) =\lim\nolimits^{s}\mathbf{H}_{-t}\left( 
\mathbf{X},G\right) \longrightarrow E_{2}^{st}\left( \mathbf{X}^{\prime
}\right) =\lim\nolimits^{s}\mathbf{H}_{-t}\left( \mathbf{X}^{\prime
},G\right) .
\end{equation*}

Finally, if $\mathbf{X=}\left( X_{i}\right) _{i\in \mathbf{I}}\rightarrow 
\mathbf{X}^{\prime }\mathbf{=}\left( X_{i}^{\prime }\right) _{i\in \mathbf{I}%
}$ is a level equivalence, the homomorphisms $C_{\ast }\left( X_{i},G\right)
\rightarrow C_{\ast }\left( X_{i}^{\prime },G\right) $ are weak
equivalences. Hence, $\mathbf{H}_{-t}\left( \mathbf{X},G\right) \rightarrow 
\mathbf{H}_{-t}\left( \mathbf{X}^{\prime },G\right) $ and $E_{2}^{st}\left( 
\mathbf{X}\right) \rightarrow E_{2}^{st}\left( \mathbf{X}^{\prime }\right) $
are isomorphisms, as desired.
\end{proof}

\subsection{Balanced homology}

\begin{definition}
\label{Def-balanced-pro-homology}Given a pro-space $\mathbf{X=}\left(
X_{i}\right) _{i\in \mathbf{I}}\in \mathbf{Pro}\left( \mathbf{TOP}\right) $,
let $\mathbf{C}_{\ast }^{b}\left( \mathbf{X},G\right) $ be the tensor
product in the sense of Theorem \ref{Th-tensor-product} 
\begin{equation*}
\mathbf{C}_{\ast }^{b}\left( \mathbf{X},G\right) :=\mathbf{C}_{\ast }\left( 
\mathbf{X},\mathbb{Z}\right) \otimes _{\mathbb{Z}}G\in \mathbf{Pro}\left( 
\mathbf{CHAIN}\left( \mathbb{Z}\right) \right) .
\end{equation*}%
where $\mathbf{C}_{\ast }\left( \mathbf{X},\mathbb{Z}\right) $ is the
pro-complex from Definition \ref{Def-pro-homology}. $\mathbf{C}_{\ast
}^{b}\left( \mathbf{X},G\right) $ can be represented by an inverse system%
\begin{equation*}
\mathbf{C}_{\ast }^{b}\left( \mathbf{X},G\right) =\left( \left( Y_{\ast
}\right) _{j}\right) _{j\in \mathbf{J}}.
\end{equation*}%
Let $\mathbf{H}_{n}^{b}\left( \mathbf{X},G\right) $ be the corresponding
pro-homology group:%
\begin{equation*}
\mathbf{H}_{n}^{b}\left( \mathbf{X},G\right) :=\left( H_{n}\left(
Y_{j}\right) \right) _{j\in \mathbf{J}}.
\end{equation*}%
Given a topological space $X$, let $X\rightarrow \mathbf{X=}\left(
X_{i}\right) _{i\in \mathbf{I}}$ be a strong polyhedral expansion. Define%
\begin{equation*}
\mathbf{H}_{n}^{b}\left( X,G\right) :=\mathbf{H}_{n}^{b}\left( \mathbf{X}%
,G\right) .
\end{equation*}
\end{definition}

\begin{remark}
Strictly speaking, the items $Y_{n}$ above are abelian pro-groups that are
represented by \textbf{different} index categories $J_{n}$:%
\begin{equation*}
\mathbf{C}_{n}^{b}\left( \mathbf{X},G\right) =\left( \left( Y_{n}\right)
_{j}\right) _{j\in \mathbf{J}_{n}}.
\end{equation*}%
However, Remark \ref{Rem-tensor-product-complexes} guarantees that the index
categories $\mathbf{J}_{n}$ can be chosen to be equal, and the differentials
to be level morphisms (Definition \ref{Def-level-morphism}).
\end{remark}

\begin{definition}
\label{Def-balanced-strong-homology}Given a pro-space $\mathbf{X=}\left(
X_{i}\right) _{i\in \mathbf{I}}\in \mathbf{Pro}\left( \mathbf{TOP}\right) $,
define%
\begin{equation*}
\overline{H}_{n}^{b}\left( \mathbf{X},G\right) :=H_{n}\left( \underleftarrow{%
\mathbf{holim}}\left( \mathbf{C}_{\ast }^{b}\left( \mathbf{X},G\right)
\right) \right)
\end{equation*}%
where the complex in the brackets is the homotopy inverse limit from
Definition \ref{Def-holim-pro-CHAIN}. Given a topological space $X$, let $%
X\rightarrow \mathbf{X=}\left( X_{i}\right) _{i\in I}$ be a strong
polyhedral expansion. Define%
\begin{equation*}
\overline{H}_{n}^{b}\left( X,G\right) :=\overline{H}_{n}^{b}\left( \mathbf{X}%
,G\right) .
\end{equation*}
\end{definition}

\begin{remark}
Due to Theorem \ref{Th-map-to-weak-products}, there are natural (on $X$ and $%
G$) morphisms 
\begin{equation*}
\mathbf{C}_{\ast }^{b}\left( \mathbf{X},G\right) =\mathbf{C}_{\ast }\left( 
\mathbf{X},\mathbb{Z}\right) \otimes _{\mathbb{Z}}G\longrightarrow \mathbf{C}%
_{\ast }\left( \mathbf{X},\mathbb{Z}\right) \widetilde{\otimes }_{\mathbb{Z}%
}G=\mathbf{C}_{\ast }\left( \mathbf{X},G\right)
\end{equation*}%
of pro-complexes inducing the morphisms $\mathbf{H}_{n}^{b}\left( X,G\right)
\rightarrow \mathbf{H}_{n}\left( X,G\right) $ and $\overline{H}%
_{n}^{b}\left( X,G\right) \rightarrow \overline{H}_{n}\left( X,G\right) $.
\end{remark}

Let us remind that $\mathbf{FAB}$ is the class of finitely generated abelian
groups.

\begin{proposition}
\label{Prop-invariance}~

\begin{enumerate}
\item \label{Prop-invariance-balanced-pro}Balanced pro-homology is strong
shape invariant.

\item \label{Prop-invariance-compare-pro}The mappings $\mathbf{H}%
_{n}^{b}\left( X,G\right) \rightarrow \mathbf{H}_{n}\left( X,G\right) $ are
isomorphisms if $G\in \mathbf{FAB}$.

\item \label{Prop-invariance-balanced-strong}Balanced strong homology is
strong shape invariant.

\item \label{Prop-invariance-compare-strong}The mappings $\overline{H}%
_{n}^{b}\left( X,G\right) \rightarrow \overline{H}_{n}\left( X,G\right) $
are isomorphisms if $G\in \mathbf{FAB}$.
\end{enumerate}
\end{proposition}

\begin{proof}
~

\begin{enumerate}
\item The functor 
\begin{equation*}
\mathbf{X}\longmapsto \mathbf{C}_{\ast }\left( \mathbf{X},\mathbb{Z}\right)
:?\longrightarrow \mathbf{Pro}\left( \mathbf{CHAIN}\left( \mathbb{Z}\right)
\right)
\end{equation*}%
is evidently defined on the category $\mathbf{Inv}\left( \mathbf{TOP}\right) 
$ (see Definition \ref{Def-Inv(C)}), therefore also on the category $\mathbf{%
Inv}\left( \mathbf{POL}\right) $. Let $G\in \mathbf{Mod}\left( \mathbb{Z}%
\right) $. Applying the tensor product from Theorem \ref{Th-tensor-product},
one gets a functor%
\begin{equation*}
\mathbf{X}\longmapsto \mathbf{C}_{\ast }^{b}\left( \mathbf{X},G\right) :=%
\mathbf{C}_{\ast }\left( \mathbf{X},\mathbb{Z}\right) \otimes _{\mathbb{Z}}G:%
\mathbf{Inv}\left( \mathbf{POL}\right) \longrightarrow \mathbf{Pro}\left( 
\mathbf{CHAIN}\left( \mathbb{Z}\right) \right)
\end{equation*}%
and functors%
\begin{equation*}
\mathbf{X}\longmapsto \mathbf{H}_{n}^{b}\left( \mathbf{X},G\right) :=\mathbf{%
H}_{n}\left( \mathbf{C}_{\ast }^{b}\left( \mathbf{X},G\right) \right) :%
\mathbf{Inv}\left( \mathbf{POL}\right) \longrightarrow \mathbf{Pro}\left( 
\mathbb{Z}\right) .
\end{equation*}%
In order to use Theorem \ref{Th-description-SSh}, let us check that special
morphisms are mapped into isomorphisms. Due to Proposition \ref%
{Prop-special-morphisms}, it is enough to check this for cofinal morphisms
and for level equivalences. If $f:\mathbf{X}\rightarrow \mathbf{Y}$ is a
cofinal morphism in $\mathbf{Inv}\left( \mathbf{TOP}\right) $ then 
\begin{equation*}
\mathbf{C}_{\ast }\left( f,\mathbb{Z}\right) :\mathbf{C}_{\ast }\left( 
\mathbf{X},\mathbb{Z}\right) \longrightarrow \mathbf{C}_{\ast }\left( 
\mathbf{Y},\mathbb{Z}\right)
\end{equation*}%
is a cofinal morphism in $\mathbf{Inv}\left( \mathbf{CHAIN}\left( \mathbb{Z}%
\right) \right) $, inducing therefore an isomorphism in $\mathbf{Pro}\left( 
\mathbf{CHAIN}\left( \mathbb{Z}\right) \right) $. It follows that%
\begin{equation*}
\mathbf{C}_{\ast }^{b}\left( f,G\right) :\mathbf{C}_{\ast }^{b}\left( 
\mathbf{X},G\right) \longrightarrow \mathbf{C}_{\ast }^{b}\left( \mathbf{Y}%
,G\right)
\end{equation*}%
and 
\begin{equation*}
\mathbf{H}_{\ast }^{b}\left( f,G\right) :\mathbf{H}_{\ast }^{b}\left( 
\mathbf{X},G\right) \longrightarrow \mathbf{H}_{\ast }^{b}\left( \mathbf{Y}%
,G\right)
\end{equation*}%
are isomorphisms as well (in the categories $\mathbf{Pro}\left( \mathbf{CHAIN%
}\left( \mathbb{Z}\right) \right) $ and $\mathbf{Pro}\left( \mathbb{Z}%
\right) $ respectively). Let now 
\begin{equation*}
f=\left( f_{i}\right) _{i\in \mathbf{I}}:\mathbf{X=}\left( X_{i}\right)
_{i\in \mathbf{I}}\longrightarrow \mathbf{Y=}\left( Y_{i}\right) _{i\in 
\mathbf{I}}
\end{equation*}%
be a level equivalence. Let $n\in \mathbb{Z}$. It follows that both $\mathbf{%
H}_{n}\left( \mathbf{X},\mathbb{Z}\right) \rightarrow \mathbf{H}_{n}\left( 
\mathbf{Y},\mathbb{Z}\right) $ and $\mathbf{H}_{n-1}\left( \mathbf{X},%
\mathbb{Z}\right) \rightarrow \mathbf{H}_{n-1}\left( \mathbf{Y},\mathbb{Z}%
\right) $ are isomorphisms in $\mathbf{Pro}\left( \mathbb{Z}\right) $.
Hence, both 
\begin{equation*}
\mathbf{H}_{n}\left( \mathbf{X},\mathbb{Z}\right) \otimes _{\mathbb{Z}%
}G\longrightarrow \mathbf{H}_{n}\left( \mathbf{Y},\mathbb{Z}\right) \otimes
_{\mathbb{Z}}G
\end{equation*}%
and%
\begin{equation*}
\mathbf{Tor}_{1}^{\mathbb{Z}}\left( \mathbf{H}_{n-1}\left( \mathbf{X},%
\mathbb{Z}\right) ,G\right) \longrightarrow \mathbf{Tor}_{1}^{\mathbb{Z}%
}\left( \mathbf{H}_{n-1}\left( \mathbf{Y},\mathbb{Z}\right) ,G\right)
\end{equation*}%
are isomorphisms. The singular chain complexes $C_{\ast }\left( X_{i},%
\mathbb{Z}\right) $ and $C_{\ast }\left( Y_{i},\mathbb{Z}\right) $ consist
of free abelian groups, therefore the pro-complexes $\mathbf{C}_{\ast
}\left( \mathbf{X},\mathbb{Z}\right) $ and $\mathbf{C}_{\ast }\left( \mathbf{%
Y},\mathbb{Z}\right) $ consist of quasi-projective $\mathbb{Z}$-modules
(Definition \ref{Def-quasi-projective} and Proposition \ref%
{Prop-quasi-projective}), and one can apply Theorem \ref%
{Th-UCF-pro-complexes}: there exists a morphism of short exact sequences%
\begin{equation*}
\begin{diagram}[size=2.0em,textflow]
0 & \rTo & \QTR{bf}{H}_{n}\left( \QTR{bf}{X},\QTR{Bbb}{Z}\right) \otimes _{\QTR{Bbb}{Z}}G & \rTo & \QTR{bf}{H}_{n}^{b}\left( \QTR{bf}{X},G\right) & \rTo & Tor_{1}^{\QTR{Bbb}{Z}}\left( \QTR{bf}{H}_{n-1}\left( \QTR{bf}{X},\QTR{Bbb}{Z}\right) ,G\right) & \rTo & 0 \\
  &      & \dTo                   &      & \dTo                   &      &  \dTo          \\
0 & \rTo & \QTR{bf}{H}_{n}\left( \QTR{bf}{Y},\QTR{Bbb}{Z}\right) \otimes _{\QTR{Bbb}{Z}}G & \rTo & \QTR{bf}{H}_{n}^{b}\left( \QTR{bf}{Y},G\right) & \rTo & Tor_{1}^{\QTR{Bbb}{Z}}\left( \QTR{bf}{H}_{n-1}\left( \QTR{bf}{Y},\QTR{Bbb}{Z}\right) ,G\right) & \rTo & 0 \\
\end{diagram}%
\end{equation*}%
The $5$-lemma gives the desired isomorphism $\mathbf{H}_{\ast }^{b}\left(
f,G\right) :\mathbf{H}_{\ast }^{b}\left( \mathbf{X},G\right) \rightarrow 
\mathbf{H}_{\ast }^{b}\left( \mathbf{Y},G\right) $. Hence, the functors $%
\mathbf{H}_{n}^{b}\left( ?,G\right) $ maps special morphisms into
isomorphisms. Theorem \ref{Th-description-SSh} implies that one has
well-defined functors $\mathbf{H}_{n}^{b}\left( ?,G\right) :\mathbf{SSh}%
\rightarrow \mathbf{Mod}\left( \mathbb{Z}\right) $.

\item Follows from Theorem \ref{Th-map-to-weak-products}, because $\mathbb{Z}
$ is noetherian, therefore any finitely generated $\mathbb{Z}$-module is
finitely presented.

\item The functors%
\begin{equation*}
\mathbf{X}\longmapsto \overline{C}_{\ast }^{b}\left( \mathbf{X},G\right) :=%
\underleftarrow{\mathbf{holim}}~\mathbf{C}_{\ast }^{b}\left( \mathbf{X}%
,G\right) :\mathbf{Inv}\left( \mathbf{POL}\right) \longrightarrow \mathbf{%
CHAIN}\left( \mathbb{Z}\right) ,
\end{equation*}%
\begin{equation*}
\mathbf{X}\longmapsto \overline{H}_{n}^{b}\left( \mathbf{X},G\right)
:=H_{n}\left( \overline{C}_{\ast }^{b}\left( \mathbf{X},G\right) \right) :%
\mathbf{Inv}\left( \mathbf{POL}\right) \longrightarrow \mathbf{Mod}\left( 
\mathbb{Z}\right) ,
\end{equation*}%
are well-defined. In order to extend the definitions to the strong shape
category $\mathbf{SSh}$, it is enough, due to Theorem \ref%
{Th-description-SSh} and Proposition \ref{Prop-special-morphisms}, to check
whether cofinal morphisms and level equivalences are mapped into
isomorphisms. If $f:\mathbf{X}\rightarrow \mathbf{Y}$ is cofinal, it follows
from Theorem \ref{Spectral-holimit} (\ref{Spectral-holimit-cofinality}), that%
\begin{equation*}
\overline{C}_{\ast }^{b}\left( \mathbf{X},G\right) =\underleftarrow{\mathbf{%
holim}}~\mathbf{C}_{\ast }^{b}\left( \mathbf{X},G\right) \longrightarrow 
\underleftarrow{\mathbf{holim}}~\mathbf{C}_{\ast }^{b}\left( \mathbf{Y}%
,G\right) =\overline{C}_{\ast }^{b}\left( \mathbf{X},G\right)
\end{equation*}%
is a weak equivalence of complexes, hence%
\begin{equation*}
\overline{H}_{\ast }^{b}\left( \mathbf{X},G\right) =H_{n}\left( 
\underleftarrow{\mathbf{holim}}~\mathbf{C}_{\ast }^{b}\left( \mathbf{X}%
,G\right) \right) \longrightarrow H_{n}\left( \underleftarrow{\mathbf{holim}}%
~\mathbf{C}_{\ast }^{b}\left( \mathbf{Y},G\right) \right) =\overline{H}%
_{\ast }^{b}\left( \mathbf{X},G\right)
\end{equation*}%
is an isomorphism. Let now 
\begin{equation*}
f=\left( f_{i}\right) _{i\in \mathbf{I}}:\mathbf{X=}\left( X_{i}\right)
_{i\in \mathbf{I}}\longrightarrow \mathbf{Y=}\left( Y_{i}\right) _{i\in 
\mathbf{I}}
\end{equation*}%
be a level equivalence. It follows from (\ref{Prop-invariance-balanced-pro})
that 
\begin{equation*}
E_{2}^{st}\left( \mathbf{C}_{\ast }^{b}\left( \mathbf{X},G\right) \right) =%
\underleftarrow{\lim }^{s}~H_{-t}\left( C_{\ast }^{{}}\left( \mathbf{X}%
,G\right) \right) \longrightarrow \underleftarrow{\lim }^{s}~H_{-t}\left(
C_{\ast }^{{}}\left( \mathbf{X},G\right) \right) =E_{2}^{st}\left( \mathbf{C}%
_{\ast }^{b}\left( \mathbf{Y},G\right) \right)
\end{equation*}%
is an isomorphism for all $s$, $t\in \mathbb{Z}$. Using Theorem \ref%
{Spectral-holimit} (\ref{Spectral-holimit-isomorphism}), one concludes that%
\begin{equation*}
\overline{H}_{\ast }^{b}\left( \mathbf{X},G\right) =H_{n}\left( 
\underleftarrow{\mathbf{holim}}~\mathbf{C}_{\ast }^{b}\left( \mathbf{X}%
,G\right) \right) \longrightarrow H_{n}\left( \underleftarrow{\mathbf{holim}}%
~\mathbf{C}_{\ast }^{b}\left( \mathbf{Y},G\right) \right) =\overline{H}%
_{\ast }^{b}\left( \mathbf{X},G\right)
\end{equation*}%
is an isomorphism. Finally, Theorem \ref{Th-description-SSh} guarantees that 
$\overline{H}_{n}^{b}\left( ?,G\right) $ are well-defined functors from $%
\mathbf{SSh}$ to $\mathbf{Mod}\left( \mathbb{Z}\right) $.

\item Follows from Theorem \ref{Th-map-to-weak-products}.
\end{enumerate}
\end{proof}

\section{Proof of the main results}

\subsection{\label{Sec-proof-pairing}Proof of Proposition \protect\ref%
{Prop-pairing}}

\begin{proof}
For each $a\in G$ and for each of the four theories $h_{\ast }$ we will
define morphisms%
\begin{equation*}
\overline{a}_{\ast }:h_{\ast }\left( X,\mathbb{Z}\right) \longrightarrow
h_{\ast }\left( X,G\right)
\end{equation*}%
which satisfy the condition $\left( \overline{a+b}\right) _{\ast }=\overline{%
a}_{\ast }+\overline{b}_{\ast }$, $a$, $b\in G$. Let $X\rightarrow \mathbf{X=%
}\left( X_{i}\right) _{i\in \mathbf{I}}$ be a polyhedral expansion, and let $%
\left( C_{\ast }\left( X_{i},\mathbb{Z}\right) \right) _{i\in \mathbf{I}}$
and $\left( C_{\ast }\left( X_{i},G\right) \right) _{i\in \mathbf{I}}$ be
the corresponding pro-complexes

\begin{enumerate}
\item \textbf{Pro-homology}. The mappings $c\longmapsto c\otimes a$ define
morphisms of pro-complexes%
\begin{equation*}
\overline{a}_{\ast }:\left( C_{\ast }\left( X_{i},\mathbb{Z}\right) \right)
_{i\in \mathbf{I}}\longrightarrow \left( C_{\ast }\left( X_{i},G\right)
\right) _{i\in \mathbf{I}}
\end{equation*}%
and morphisms of their pro-homology groups%
\begin{equation*}
\overline{a}_{\ast }:\left( H_{\ast }\left( X_{i},\mathbb{Z}\right) \right)
_{i\in \mathbf{I}}\longrightarrow \left( H_{\ast }\left( X_{i},G\right)
\right) _{i\in \mathbf{I}}.
\end{equation*}%
Clearly $\left( \overline{a+b}\right) _{\ast }=\overline{a}_{\ast }+%
\overline{b}_{\ast }$, and we get a homomorphism of abelian groups%
\begin{equation*}
a\longmapsto \overline{a}_{\ast }:G\longrightarrow Hom_{\mathbf{Pro}\left( 
\mathbb{Z}\right) }\left( \mathbf{H}_{\ast }\left( \mathbf{X},\mathbb{Z}%
\right) ,\mathbf{H}_{\ast }\left( \mathbf{X},G\right) \right) .
\end{equation*}%
Theorem \ref{Th-tensor-product} gives the desired homomorphism%
\begin{equation*}
\mathbf{H}_{\ast }\left( \mathbf{X},\mathbb{Z}\right) \otimes _{\mathbb{Z}%
}G\longrightarrow \mathbf{H}_{\ast }\left( \mathbf{X},G\right) .
\end{equation*}

\item \textbf{Strong homology}. The mappings 
\begin{equation*}
\overline{a}_{\ast }:\left( C_{\ast }\left( X_{i},\mathbb{Z}\right) \right)
_{i\in \mathbf{I}}\longrightarrow \left( C_{\ast }\left( X_{i},G\right)
\right) _{i\in \mathbf{I}}
\end{equation*}%
from (1) define morphisms of the homotopy inverse limits%
\begin{equation*}
\overline{a}_{\ast }:\underleftarrow{\mathbf{holim}}_{i}\left( C_{\ast
}\left( X_{i},\mathbb{Z}\right) \right) _{i\in \mathbf{I}}\longrightarrow 
\underleftarrow{\mathbf{holim}}_{i}\left( C_{\ast }\left( X_{i},G\right)
\right) _{i\in \mathbf{I}}
\end{equation*}%
and their homologies%
\begin{equation*}
\overline{a}_{\ast }:\overline{H}_{\ast }\left( X,\mathbb{Z}\right)
\longrightarrow \overline{H}_{\ast }\left( X,G\right) .
\end{equation*}%
Clearly $\left( \overline{a+b}\right) _{\ast }=\overline{a}_{\ast }+%
\overline{b}_{\ast }$, and we get a homomorphism of abelian groups%
\begin{equation*}
a\longmapsto \overline{a}_{\ast }:G\longrightarrow Hom_{\mathbf{Mod}\left( 
\mathbb{Z}\right) }\left( \overline{H}_{\ast }\left( X,\mathbb{Z}\right) ,%
\overline{H}_{\ast }\left( X,G\right) \right) .
\end{equation*}%
The usual properties of the tensor product give the desired homomorphism%
\begin{equation*}
\overline{H}_{\ast }\left( X,\mathbb{Z}\right) \otimes _{\mathbb{Z}%
}G\longrightarrow \overline{H}_{\ast }\left( X,G\right) .
\end{equation*}

\item \textbf{Balanced pro-homology}. Due to Theorem \ref{Th-tensor-product}%
, there is a natural isomorphism of abelian groups%
\begin{eqnarray*}
&&Hom_{\mathbf{Mod}\left( \mathbb{Z}\right) }\left( G,Hom_{\mathbf{Pro}%
\left( \mathbb{Z}\right) }\left( \mathbf{C}_{\ast }\left( \mathbf{X},\mathbb{%
Z}\right) ,\mathbf{C}_{\ast }^{b}\left( \mathbf{X},G\right) \right) \right) 
\simeq
\\
&&%
\simeq%
Hom_{\mathbf{Pro}\left( \mathbb{Z}\right) }\left( \mathbf{C}_{\ast }\left( 
\mathbf{X},\mathbb{Z}\right) \otimes _{\mathbb{Z}}G,\mathbf{C}_{\ast
}^{b}\left( \mathbf{X},G\right) \right) .
\end{eqnarray*}%
Since $\mathbf{C}_{\ast }\left( \mathbf{X},\mathbb{Z}\right) \otimes _{%
\mathbb{Z}}G=\mathbf{C}_{\ast }^{b}\left( \mathbf{X},G\right) $ by
definition, let 
\begin{equation*}
\varphi \in \left( G,Hom_{\mathbf{Pro}\left( \mathbb{Z}\right) }\left( 
\mathbf{C}_{\ast }\left( \mathbf{X},\mathbb{Z}\right) ,\mathbf{C}_{\ast
}^{b}\left( \mathbf{X},G\right) \right) \right)
\end{equation*}%
be the morphism corresponding to the identity morphism $\mathbf{C}_{\ast
}\left( \mathbf{X},\mathbb{Z}\right) \otimes _{\mathbb{Z}}G\rightarrow 
\mathbf{C}_{\ast }^{b}\left( \mathbf{X},G\right) $ under the isomorphism
above. For $a\in G$, let 
\begin{equation*}
\overline{\varphi }\left( a\right) _{\ast }:\mathbf{H}_{\ast }\left( \mathbf{%
X},\mathbb{Z}\right) \longrightarrow \mathbf{H}_{\ast }^{b}\left( \mathbf{X}%
,G\right) .
\end{equation*}%
be the induced mapping of pro-homologies. The correspondence $a\mapsto 
\overline{\varphi }\left( a\right) _{\ast }$ defines a homomorphism%
\begin{equation*}
G\longrightarrow Hom_{\mathbf{Pro}\left( \mathbb{Z}\right) }\left( \mathbf{H}%
_{\ast }\left( \mathbf{X},\mathbb{Z}\right) ,\mathbf{H}_{\ast }^{b}\left( 
\mathbf{X},G\right) \right) .
\end{equation*}%
Applying again Theorem \ref{Th-tensor-product}, one gets the desired
morphism of abelian pro-groups%
\begin{equation*}
\mathbf{H}_{\ast }\left( \mathbf{X},\mathbb{Z}\right) \otimes _{\mathbb{Z}%
}G\longrightarrow \mathbf{H}_{\ast }^{b}\left( \mathbf{X},G\right) .
\end{equation*}

\item \textbf{Balanced strong homology}. The morphisms%
\begin{equation*}
\varphi \left( a\right) \in Hom_{\mathbf{Pro}\left( \mathbb{Z}\right)
}\left( \mathbf{C}_{\ast }\left( \mathbf{X},\mathbb{Z}\right) ,\mathbf{C}%
_{\ast }^{b}\left( \mathbf{X},G\right) \right)
\end{equation*}%
from (3) define morphisms%
\begin{equation*}
\underleftarrow{\mathbf{holim}}~\varphi \left( a\right) :\underleftarrow{%
\mathbf{holim}}~\mathbf{C}_{\ast }\left( \mathbf{X},\mathbb{Z}\right)
\longrightarrow \underleftarrow{\mathbf{holim}}~\mathbf{C}_{\ast }^{b}\left( 
\mathbf{X},G\right)
\end{equation*}%
and the corresponding morphisms of homologies%
\begin{equation*}
\overline{a}_{\ast }:\overline{H}_{\ast }\left( X,\mathbb{Z}\right) =%
\overline{H}_{\ast }^{b}\left( X,\mathbb{Z}\right) \longrightarrow \overline{%
H}_{\ast }^{b}\left( X,G\right) .
\end{equation*}%
Clearly $\left( \overline{a+b}\right) _{\ast }=\overline{a}_{\ast }+%
\overline{b}_{\ast }$, and one gets a homomorphism of abelian groups%
\begin{equation*}
a\longmapsto \overline{a}_{\ast }:G\longrightarrow Hom_{\mathbf{Mod}\left( 
\mathbb{Z}\right) }\left( \overline{H}_{\ast }^{b}\left( X,\mathbb{Z}\right)
,\overline{H}_{\ast }^{b}\left( X,G\right) \right) .
\end{equation*}%
The usual properties of the tensor product give the desired homomorphism%
\begin{equation*}
\overline{H}_{\ast }^{b}\left( X,\mathbb{Z}\right) \otimes _{\mathbb{Z}%
}G\longrightarrow \overline{H}_{\ast }^{b}\left( X,G\right) .
\end{equation*}
\end{enumerate}
\end{proof}

\subsection{\label{Sec-proof-UCF-FAB}Proof of Theorem \protect\ref%
{Th-UCF-FAB}}

\begin{proof}
Since for $G\in \mathbf{FAB}$, the natural morphisms $\mathbf{H}%
_{n}^{b}\left( X,G\right) \rightarrow \mathbf{H}_{n}\left( X,G\right) $ and $%
\overline{H}_{n}^{b}\left( X,G\right) \rightarrow \overline{H}_{n}\left(
X,G\right) $ are isomorphisms (Theorem \ref{Prop-invariance} (\ref%
{Prop-invariance-compare-pro}, \ref{Prop-invariance-compare-strong})), it is
enough to prove our statement for the two non-balanced homology theories.
Let $X\rightarrow \mathbf{X}$ be a strong polyhedral expansion. Consider
first the case $G=\mathbb{Z}^{n}$ is a free finitely generated abelian
group. Clearly, $\mathbf{Tor}_{1}^{\mathbb{Z}}\left( \mathbf{H}_{n-1}\left( 
\mathbf{X}\right) ,G\right) =0$ and $Tor_{1}^{\mathbb{Z}}\left( \overline{H}%
_{n-1}\left( \mathbf{X}\right) ,G\right) =0$, and it is enough to notice
that both%
\begin{equation*}
\mathbf{H}_{n}\left( \mathbf{X,}\mathbb{Z}\right) \otimes _{\mathbb{Z}%
}G=\dbigoplus\limits^{n}\mathbf{H}_{n}\left( \mathbf{X,}\mathbb{Z}\right)
\longrightarrow \mathbf{H}_{n}\left( \mathbf{X},G\right)
\end{equation*}%
and%
\begin{equation*}
\overline{H}_{n}\left( \mathbf{X,}\mathbb{Z}\right) \otimes _{\mathbb{Z}%
}G=\dbigoplus\limits^{n}\overline{H}_{n}\left( \mathbf{X,}\mathbb{Z}\right)
\longrightarrow \overline{H}_{n}\left( \mathbf{X},G\right)
\end{equation*}%
are isomorphisms.

Let now%
\begin{equation*}
0\longleftarrow G\longleftarrow F_{0}\longleftarrow F_{1}\longleftarrow 0
\end{equation*}%
be a resolution of $G$ where $F_{i}$ are free finitely generated abelian
groups. There is a short exact sequence of pro-complexes and level morphisms%
\begin{equation*}
0\longrightarrow \mathbf{C}_{\ast }\left( \mathbf{X},F_{1}\right)
\longrightarrow \mathbf{C}_{\ast }\left( \mathbf{X},F_{0}\right)
\longrightarrow \mathbf{C}_{\ast }\left( \mathbf{X},G\right) \longrightarrow
0
\end{equation*}%
inducing a short exact sequence of homotopy inverse limits%
\begin{equation*}
0\longrightarrow \underleftarrow{\mathbf{holim}}~\mathbf{C}_{\ast }\left( 
\mathbf{X},F_{1}\right) \longrightarrow \underleftarrow{\mathbf{holim}}~%
\mathbf{C}_{\ast }\left( \mathbf{X},F_{0}\right) \longrightarrow 
\underleftarrow{\mathbf{holim}}~\mathbf{C}_{\ast }\left( \mathbf{X},G\right)
\longrightarrow 0.
\end{equation*}%
The first sequence gives rise to long exact sequence of pro-homologies%
\begin{equation*}
...\longrightarrow \mathbf{H}_{n}\left( \mathbf{X},F_{1}\right)
\longrightarrow \mathbf{H}_{n}\left( \mathbf{X},F_{0}\right) \longrightarrow 
\mathbf{H}_{\ast }\left( \mathbf{X},G\right) \longrightarrow \mathbf{H}%
_{n-1}\left( \mathbf{X},F_{1}\right) \longrightarrow \mathbf{H}_{n-1}\left( 
\mathbf{X},F_{0}\right) \longrightarrow ...
\end{equation*}%
while the second one induces a long exact sequence of strong homologies%
\begin{equation*}
...\longrightarrow \overline{H}_{n}\left( \mathbf{X},F_{1}\right)
\longrightarrow \overline{H}_{n}\left( \mathbf{X},F_{0}\right)
\longrightarrow \overline{H}_{n}\left( \mathbf{X},G\right) \longrightarrow 
\overline{H}_{n-1}\left( \mathbf{X},F_{1}\right) \longrightarrow \overline{H}%
_{n-1}\left( \mathbf{X},F_{0}\right) \longrightarrow ...
\end{equation*}%
Consider a morphism of exact sequences%
\begin{equation*}
\begin{diagram}
\QTR{bf}{H}_{n}\left( \QTR{bf}{X},F_{1}\right) & \rTo & \QTR{bf}{H}_{n}\left( \QTR{bf}{X},F_{0}\right) & \rTo & \QTR{bf}{H}_{n}\left( \QTR{bf}{X},G\right) & \rTo & \QTR{bf}{H}_{n-1}\left( \QTR{bf}{X},F_{1}\right) \\
   \uTo_{\alpha} &  & \uTo_{\beta}         &   & \uTo_{\gamma}  & & \uTo \\
\QTR{bf}{H}_{n}\left( \QTR{bf}{X},\QTR{Bbb}{Z}\right) \otimes _{\QTR{Bbb}{Z}}F_{1} & \rTo & \QTR{bf}{H}_{n}\left( \QTR{bf}{X},\QTR{Bbb}{Z}\right) \otimes _{\QTR{Bbb}{Z}}F_{0} & \rTo & \QTR{bf}{H}_{n}\left( \QTR{bf}{X},\QTR{Bbb}{Z}\right) \otimes _{\QTR{Bbb}{Z}}G & \rTo & 0 \\
\end{diagram}%
\end{equation*}%
Since $\alpha $ and $\beta $ are isomorphisms, $\gamma $ is a monomorphism
with the cokernel equal to 
\begin{eqnarray*}
&&\ker \left( \mathbf{H}_{n-1}\left( \mathbf{X},F_{1}\right) \longrightarrow 
\mathbf{H}_{n-1}\left( \mathbf{X},F_{0}\right) \right) 
\simeq%
\ker \left( \mathbf{H}_{n-1}\left( \mathbf{X},\mathbb{Z}\right) \otimes _{%
\mathbb{Z}}F_{1}\longrightarrow \mathbf{H}_{n-1}\left( \mathbf{X},\mathbb{Z}%
\right) \otimes _{\mathbb{Z}}F_{0}\right) 
\simeq
\\
&&%
\simeq%
\mathbf{Tor}_{1}^{\mathbb{Z}}\left( \mathbf{H}_{n-1}\left( \mathbf{X},%
\mathbb{Z}\right) ,G\right) .
\end{eqnarray*}%
Similarly, $\overline{H}_{n}\left( \mathbf{X},\mathbb{Z}\right) \otimes _{%
\mathbb{Z}}G\rightarrow \overline{H}_{n}\left( \mathbf{X},G\right) $ is a
monomorphism with the cokernel equal to%
\begin{eqnarray*}
&&\ker \left( \overline{H}_{n-1}\left( \mathbf{X},F_{1}\right)
\longrightarrow \overline{H}_{n-1}\left( \mathbf{X},F_{0}\right) \right) 
\simeq%
\ker \left( \overline{H}_{n-1}\left( \mathbf{X},\mathbb{Z}\right) \otimes _{%
\mathbb{Z}}F_{1}\longrightarrow \overline{H}_{n-1}\left( \mathbf{X},\mathbb{Z%
}\right) \otimes _{\mathbb{Z}}F_{0}\right) 
\simeq
\\
&&%
\simeq%
Tor_{1}^{\mathbb{Z}}\left( \overline{H}_{n-1}\left( \mathbf{X},\mathbb{Z}%
\right) ,G\right) .
\end{eqnarray*}
\end{proof}

\subsection{\label{Sec-proof-balanced-pro}Proof of Theorem \protect\ref%
{Th-balanced-pro}}

\begin{proof}
Let $X\rightarrow \mathbf{X}$ be a strong polyhedral expansion. Apply
Theorem \ref{Th-UCF-pro-complexes} to the pro-complex $\mathbf{C}_{\ast
}\left( \mathbf{X},\mathbb{Z}\right) $.
\end{proof}

\subsection{\label{Sec-proof-pro}Proof of Theorem \protect\ref{Th-pro}}

\begin{definition}
\label{Def-P(H)}For an arbitrary abelian group $H$, let $\mathbf{P}\left(
H\right) ,\mathbf{P}^{\omega }\left( H\right) \in \mathbf{Pro}\left( \mathbb{%
Z}\right) $ be the following pro-groups:%
\begin{eqnarray*}
\mathbf{P}\left( H\right) &=&\left( H\longleftarrow H\oplus H\longleftarrow
H\oplus H\oplus H\longleftarrow ...\right) , \\
\mathbf{P}^{\omega }\left( H\right) &=&\dbigoplus\limits_{0}^{\infty }%
\mathbf{P}\left( H\right) .
\end{eqnarray*}
\end{definition}

\begin{proof}
Let $G$ be a free countably generated abelian group. Since $\mathbf{Tor}%
_{1}^{\mathbb{Z}}\left( \mathbf{H}_{n-1}\left( X,\mathbb{Z}\right) ,G\right)
=0$, it is enough to construct an example $X$ such that $\mathbf{H}%
_{n}\left( X,\mathbb{Z}\right) \otimes _{\mathbb{Z}}G\rightarrow \mathbf{H}%
_{n}\left( X,G\right) $ is \textbf{not} an isomorphism. Let 
\begin{equation*}
X_{k}=\dbigvee\limits_{0}^{\infty }S^{k}
\end{equation*}%
be the $k$-dimensional Hawaiian ear-ring, i.e. a \textbf{compact wedge} (or
the \textbf{cluster}) of countably many $k$-spheres. Assume for simplicity
that $k\neq 0$. $X_{k}$ can be described as a subspace of $\mathbb{R}^{k+1}$%
: 
\begin{equation*}
X_{k}=\dbigcup\limits_{i=0}^{\infty }S_{k}^{\left( i\right) }
\end{equation*}%
where $S_{k}^{\left( i\right) }$ is the sphere of radius $\frac{1}{i+1}$
with the center in $\left( \frac{1}{i+1},0,...,0\right) \in \mathbb{R}^{k+1}$%
. It follows from (\cite{Prasolov-Non-additivity-MR2175366}, Section 4.2, p.
505) that (see Definition \ref{Def-P(H)}) 
\begin{equation*}
\mathbf{H}_{n}\left( X_{k},\mathbb{Z}\right) =\left\{ 
\begin{array}{ccc}
\mathbf{P}\left( \mathbb{Z}\right) & \text{if} & n=k, \\ 
0 & \text{if} & n\neq k,0, \\ 
\mathbb{Z} & \text{if} & n=0.%
\end{array}%
\right.
\end{equation*}%
\begin{equation*}
\mathbf{H}_{k}\left( X_{k},\mathbb{Z}\right) \otimes _{\mathbb{Z}}G=\left\{ 
\begin{array}{ccc}
\mathbf{P}^{\omega }\left( \mathbb{Z}\right) & \text{if} & n=k, \\ 
0 & \text{if} & n\neq k,0, \\ 
G & \text{if} & n=0.%
\end{array}%
\right.
\end{equation*}%
\begin{equation*}
\mathbf{H}_{k}\left( X_{k},G\right) =\left\{ 
\begin{array}{ccc}
\mathbf{P}\left( G\right) & \text{if} & n=k, \\ 
0 & \text{if} & n\neq k,0, \\ 
G & \text{if} & n=0.%
\end{array}%
\right.
\end{equation*}%
Assume on the contrary that $\mathbf{H}_{k}\left( X_{k},\mathbb{Z}\right)
\otimes _{\mathbb{Z}}G\rightarrow \mathbf{H}_{k}\left( X_{k},G\right) $ is
an isomorphism. It would follow that%
\begin{equation*}
Hom_{\mathbf{Pro}\left( \mathbb{Z}\right) }\left( \mathbf{H}_{k}\left(
X_{k},G\right) ,B\right) \longrightarrow Hom_{\mathbf{Pro}\left( \mathbb{Z}%
\right) }\left( \mathbf{H}_{k}\left( X_{k},\mathbb{Z}\right) \otimes _{%
\mathbb{Z}}G,B\right)
\end{equation*}%
is an isomorphism for any abelian group $B$. Direct computation shows that%
\begin{eqnarray*}
Hom_{\mathbf{Pro}\left( \mathbb{Z}\right) }\left( \mathbf{H}_{k}\left(
X_{k},G\right) ,B\right) &=&\underrightarrow{\lim }~\left(
\dprod\limits_{0}^{\infty }B\longrightarrow \dprod\limits_{0}^{\infty
}B\times B\longrightarrow \dprod\limits_{0}^{\infty }B\times B\times
B\longrightarrow ...\right) , \\
Hom_{\mathbf{Pro}\left( \mathbb{Z}\right) }\left( \mathbf{H}_{k}\left( X_{k},%
\mathbb{Z}\right) \otimes _{\mathbb{Z}}G,B\right)
&=&\dprod\limits_{0}^{\infty }\underrightarrow{\lim }~\left(
B\longrightarrow B\times B\longrightarrow B\times B\times B\longrightarrow
...\right) .
\end{eqnarray*}%
Notice \textbf{direct} (instead of inverse) limits. Elements of $%
\dprod\limits_{0}^{\infty }B^{s+1}$ can be represented by infinite $s\times
\infty $ matrices $\left( b_{ij}\right) _{0\leq i\leq s,j\geq 0}$, or,
equivalently, by $\infty \times \infty $ matrices $\left( b_{ij}\right)
_{i\geq 0,j\geq 0}$ with $b_{ij}=0$ whenever $i>s$. It follows that 
\begin{equation*}
\underrightarrow{\lim }~\left( \dprod\limits_{0}^{\infty }B\longrightarrow
\dprod\limits_{0}^{\infty }B\times B\longrightarrow
\dprod\limits_{0}^{\infty }B\times B\times B\longrightarrow ...\right)
\end{equation*}%
can be represented by $\infty \times \infty $ matrices $\left( b_{ij}\right)
_{i\geq 0,j\geq 0}$ matrices with $b_{ij}=0$ whenever $i>s$ for \textbf{some}
$s$. Similarly, elements of%
\begin{equation*}
\underrightarrow{\lim }\left( B\longrightarrow B\times B\longrightarrow
B\times B\times B\longrightarrow ...\right)
\end{equation*}%
can be represented by $1\times \infty $ matrices $\left( b_{j}\right)
_{j\geq 0}$ where $b_{j}=0$ if $j>s$ for \textbf{some} $s$. Finally,
elements of 
\begin{equation*}
\dprod\limits_{0}^{\infty }\underrightarrow{\lim }~\left( B\longrightarrow
B\times B\longrightarrow B\times B\times B\longrightarrow ...\right)
\end{equation*}%
can be represented by $\infty \times \infty $ matrices $\left( b_{ij}\right)
_{i\geq 0,j\geq 0}$ matrices with $b_{ij}=0$ whenever $i>s_{j}$ for \textbf{%
some} $s_{j}$ (depending on $j$). Let $B$ be any non-trivial abelian group,
let $c\in B$, $c\neq 0$. Define%
\begin{equation*}
b_{ij}=\left\{ 
\begin{array}{ccc}
0 & \text{if} & i>j, \\ 
c & \text{if} & i\leq j.%
\end{array}%
\right.
\end{equation*}%
Clearly 
\begin{equation*}
\left( b_{ij}\right) \in \dprod\limits_{0}^{\infty }\underrightarrow{\lim }%
~\left( B\longrightarrow B\times B\longrightarrow B\times B\times
B\longrightarrow ...\right) ,
\end{equation*}%
but $\left( b_{ij}\right) $ does not lie in the image of%
\begin{equation*}
\underrightarrow{\lim }~\left( \dprod\limits_{0}^{\infty }B\longrightarrow
\dprod\limits_{0}^{\infty }B\times B\longrightarrow
\dprod\limits_{0}^{\infty }B\times B\times B\longrightarrow ...\right) .
\end{equation*}

The pairing $\mathbf{H}_{k}\left( X_{k},\mathbb{Z}\right) \otimes _{\mathbb{Z%
}}G\rightarrow \mathbf{H}_{k}\left( X_{k},G\right) $ is \textbf{not} an
isomorphism. Contradiction.
\end{proof}

\subsection{\label{Sec-proof-strong}Proof of Theorem \protect\ref{Th-strong}}

\begin{proof}
Let $X_{k}$, $k\geq 1$, and $G$, be the same as in Section \ref%
{Sec-proof-pro}. It follows from (\cite{Prasolov-Non-additivity-MR2175366},
Section 4.2, p. 505) that%
\begin{eqnarray*}
&&\overline{H}_{k}\left( X_{k},\mathbb{Z}\right) 
\simeq%
\dprod\limits_{0}^{\infty }\mathbb{Z}, \\
&&\overline{H}_{k}\left( X_{k},\mathbb{Z}\right) \otimes _{\mathbb{Z}}G%
\simeq%
\dbigoplus\limits_{0}^{\infty }\dprod\limits_{0}^{\infty }\mathbb{Z}, \\
&&\overline{H}_{k}\left( X_{k},G\right) 
\simeq%
\dprod\limits_{0}^{\infty }G%
\simeq%
\dprod\limits_{0}^{\infty }\dbigoplus\limits_{0}^{\infty }\mathbb{Z}.
\end{eqnarray*}%
Let us repeat the proof here. Let $F$ be any abelian group. We will
calculate $\overline{H}_{k}\left( X_{k},F\right) $ using the spectral
sequence from Theorem \ref{Th-Spectral-strong-homology} (\ref%
{Th-Spectral-strong-homology-convergence}), and the calculations from
Section \ref{Sec-proof-pro}. Clearly, since $~\mathbf{P}\left( F\right) $
(see Definition \ref{Def-P(H)}) is a tower of epimorphisms,%
\begin{equation*}
E_{2}^{st}\left( \mathbf{C}_{\ast }\left( \mathbf{X},F\right) \right)
=\left\{ 
\begin{array}{ccc}
\underleftarrow{\lim }~\mathbf{P}\left( F\right) 
\simeq%
\dprod\limits_{0}^{\infty }F & \text{if} & s=0,t=-k, \\ 
F & \text{if} & s=t=0, \\ 
\underleftarrow{\lim }^{1}~\mathbf{P}\left( F\right) =0 & \text{if} & 
s=1,t=-k, \\ 
0 & \text{otherwise.} & 
\end{array}%
\right.
\end{equation*}%
The spectral sequence degenerates, implying $\underleftarrow{\lim }%
_{r}^{1}~E_{r}^{st}=0$. It follows that%
\begin{equation*}
\overline{H}_{k}\left( X_{k},F\right) 
\simeq%
E_{\infty }^{0,-k}%
\simeq%
E_{2}^{0,-k}%
\simeq%
\dprod\limits_{0}^{\infty }F.
\end{equation*}%
Setting $F=\mathbb{Z}$ and $F=G$, one gets the desired formulas.

Moreover, $Tor_{1}^{\mathbb{Z}}\left( \overline{H}_{k-1}\left( X_{k},\mathbb{%
Z}\right) ,G\right) =0$ because $G$ is a free abelian group. The elements of 
$\overline{H}_{k}\left( X_{k},\mathbb{Z}\right) \otimes _{\mathbb{Z}}G$ are
represented by $\infty \times \infty $ matrices $\left( b_{ij}\in \mathbb{Z}%
\right) $ where $b_{ij}=0$ if $i>s$ for some $s$, while elements of $%
\overline{H}_{k}\left( X_{k},G\right) $ are represented by matrices $\left(
b_{ij}\right) $ where $b_{ij}=0$ if $i>s_{j}$ for some $s_{j}$ (depending on
\thinspace $j$). Let%
\begin{equation*}
b_{ij}=\left\{ 
\begin{array}{ccc}
0 & \text{if} & i>j, \\ 
1 & \text{if} & i\leq j.%
\end{array}%
\right.
\end{equation*}%
Clearly $\left( b_{ij}\right) \in \overline{H}_{k}\left( X_{k},G\right) $,
but $\left( b_{ij}\right) $ does not lie in the image of $\overline{H}%
_{k}\left( X_{k},\mathbb{Z}\right) \otimes _{\mathbb{Z}}G$. The pairing $%
\overline{H}_{k}\left( X_{k},\mathbb{Z}\right) \otimes _{\mathbb{Z}%
}G\rightarrow \overline{H}_{k}\left( X_{k},G\right) $ is \textbf{not} an
isomorphism. Contradiction.
\end{proof}

\subsection{\label{Sec-proof-balanced-strong}Proof of Theorem \protect\ref%
{Th-balanced-strong}}

This is the most complicated argument.

\subsubsection{\label{Sec-example-CH}The counter-example depending on the
Continuum Hypothesis}

\begin{proof}
Take again $G$ a countably generated abelian group. The first
counter-example will be the same $X_{k}$ (the cluster of $k$-spheres) as
above. However, the statement that 
\begin{equation*}
\overline{H}_{k-1}\left( X_{k},\mathbb{Z}\right) \otimes _{\mathbb{Z}}G=%
\overline{H}_{k-1}^{b}\left( X_{k},\mathbb{Z}\right) \otimes _{\mathbb{Z}%
}G\longrightarrow \overline{H}_{k-1}^{b}\left( X_{k},G\right)
\end{equation*}%
is \textbf{not} an isomorphism, would depend on the Continuum Hypothesis (in
fact, on a weaker assumption $d=\aleph _{1}$). A counter-example which is 
\textbf{independent} on the Continuum Hypothesis, will be much more
complicated (see Section \ref{Sec-absolute-example}). Assume for simplicity
that $k\geq 2$. Let $X_{k}\rightarrow \mathbf{X}$ be a strong expansion.
Consider the spectral sequence from Theorem \ref%
{Th-Spectral-holimit-pro-complex} (\ref%
{Th-Spectral-holimit-pro-complex-convergence}) for the pro-complex 
\begin{equation*}
\mathbf{C}_{\ast }^{b}\left( \mathbf{X},G\right) =\mathbf{C}_{\ast }\left( 
\mathbf{X},\mathbb{Z}\right) \otimes _{\mathbb{Z}}G=\dbigoplus\limits_{0}^{%
\infty }\mathbf{C}_{\ast }\left( \mathbf{X},\mathbb{Z}\right) .
\end{equation*}%
Reasoning exactly as in (\cite{Prasolov-Non-additivity-MR2175366}, Theorem
4.5 and Proposition 4.10), one calculates the $E_{2}$ terms (see Definition %
\ref{Def-P(H)}):%
\begin{equation*}
E_{2}^{st}\left( \mathbf{C}_{\ast }^{b}\left( \mathbf{X},G\right) \right)
=\left\{ 
\begin{array}{ccc}
\underleftarrow{\lim }~\mathbf{P}^{\omega }\left( \mathbb{Z}\right) 
\simeq%
\dbigoplus\limits_{0}^{\infty }\dprod\limits_{0}^{\infty }\mathbb{Z} & \text{%
if} & s=0,t=-k, \\ 
\underleftarrow{\lim }^{s}~\mathbf{P}^{\omega }\left( \mathbb{Z}\right) & 
\text{if} & s>0,t=-k, \\ 
\dbigoplus\limits_{0}^{\infty }\mathbb{Z} & \text{if} & s=t=0, \\ 
0 & \text{otherwise.} & 
\end{array}%
\right.
\end{equation*}%
The spectral sequence degenerates, implying $\underleftarrow{\lim }%
_{r}^{1}~E_{r}^{st}=0$, and it follows that%
\begin{equation*}
\overline{H}_{k-1}^{b}\left( \mathbf{X},G\right) 
\simeq%
\underleftarrow{\lim }^{1}~\mathbf{P}^{\omega }\left( \mathbb{Z}\right) .
\end{equation*}%
In \cite{van-Douwen-The-integers-and-topology-1984}, \S 3, the two cardinals
were defined: $b$ (bounding number) and $d$ (dominating number). It is known
that $\aleph _{1}\leq b\leq d\leq c$ where $c$ is the cardinality of
continuum. Moreover (see \cite{van-Douwen-The-integers-and-topology-1984}, 
\S 5), for any integers $1\leq p\leq l\leq m$ the statement%
\begin{equation*}
\left( b=\aleph _{p}\right) \&\left( d=\aleph _{l}\right) \&\left( c=\aleph
_{m}\right)
\end{equation*}%
is consistent with ZFC (Zermelo--Fraenkel axioms plus the Axiom of Choice).
Assume now $d=\aleph _{1}$ (this assumption is weaker than the Continuum
Hypothesis $c=\aleph _{1}$). It is proved in (\cite%
{Prasolov-Non-additivity-MR2175366}), Theorem 4, that under this assumption
the cardinality of $\underleftarrow{\lim }^{1}~\mathbf{P}^{\omega }\left( 
\mathbb{Z}\right) $ is very large (equal to $\left( \aleph _{0}\right)
^{\aleph _{1}}$). Since%
\begin{equation*}
\overline{H}_{k-1}^{b}\left( \mathbf{X},\mathbb{Z}\right) =\overline{H}%
_{k-1}\left( \mathbf{X},\mathbb{Z}\right) =0,
\end{equation*}%
it follows that%
\begin{equation*}
0=\overline{H}_{k-1}^{b}\left( X_{k},\mathbb{Z}\right) \otimes _{\mathbb{Z}%
}G\longrightarrow \overline{H}_{k-1}^{b}\left( X_{k},G\right) =%
\underleftarrow{\lim }^{1}~\mathbf{P}^{\omega }\left( \mathbb{Z}\right) \neq
0
\end{equation*}%
is \textbf{not} an isomorphism. Contradiction.
\end{proof}

\begin{remark}
\label{Rem-PFA}In \cite{Dow-Simon-Vaughan-Strong-homology-and-the-pro-1989}
it is proved that under PFA (the Proper Forcing Axiom) $\underleftarrow{\lim 
}^{1}~\mathbf{P}^{\omega }\left( \mathbb{Z}\right) =0$. Therefore, the
statement \textquotedblleft $\overline{H}_{k-1}^{b}\left( X_{k},\mathbb{Z}%
\right) \otimes _{\mathbb{Z}}G\longrightarrow \overline{H}_{k-1}^{b}\left(
X_{k},G\right) $ is an isomorphism\textquotedblright\ \textbf{does not depend%
} on ZFC.
\end{remark}

\subsubsection{\label{Sec-absolute-example}The \textquotedblleft
absolute\textquotedblright\ counter-example}

Let $\omega _{1}$ be the first uncountable ordinal.

\begin{definition}
\label{Def-A(w1)}(compare with \cite{Prasolov-Non-additivity-MR2175366},
Section 6) Let $C$ be an abelian group. Define $I\left( \omega _{1}\right) $
to be the category with the elements $\alpha <\omega _{1}$ as objects, and
inequalities $\beta <\alpha $ as morphisms $\alpha \rightarrow \beta $.
Notice the inverse order in which $\alpha $ and $\beta $ appear in the
morphisms. Consider the following two abelian pro-groups:%
\begin{eqnarray*}
\mathbf{A}\left( C\right) &=&\left( A_{\alpha }\left( C\right) \right)
_{\alpha \in I\left( \omega _{1}\right) }, \\
\mathbf{A}^{\omega }\left( C\right) &=&\dbigoplus\limits_{0}^{\infty }%
\mathbf{A}\left( C\right) ,
\end{eqnarray*}%
where%
\begin{equation*}
A_{\alpha }\left( C\right) =\dbigoplus\limits_{\gamma \in \left[ \alpha
,\omega _{1}\right\rangle }C,
\end{equation*}%
and for $\beta <\alpha $ the morphisms%
\begin{equation*}
\left( \alpha \longrightarrow \beta \right) _{\ast }:A_{\alpha }\left(
C\right) =\dbigoplus\limits_{\gamma \in \left[ \alpha ,\omega
_{1}\right\rangle }C\longrightarrow A_{\beta }\left( C\right)
=\dbigoplus\limits_{\gamma \in \left[ \beta ,\omega _{1}\right\rangle }C
\end{equation*}%
are induced by the inclusions $\left[ \alpha ,\omega _{1}\right\rangle
\subseteq \left[ \beta ,\omega _{1}\right\rangle $.
\end{definition}

\begin{proof}
Let $m\geq 1$ be an integer. Denote by $X_{m}$ the space $X\left( m,0,\omega
_{1}\right) $ from \cite{Mardesic-Nonvanishing-derived-limits-in-shape-1996}%
. As a set, $X_{m}$ is a wedge 
\begin{equation*}
\dbigvee\limits_{\alpha <\omega _{1}}B^{m}
\end{equation*}%
of $m$-dimensional balls equipped with a special paracompact topology. In (%
\cite{Mardesic-Nonvanishing-derived-limits-in-shape-1996}, Theorem 3 and 6),
a polyhedral resolution $X_{m}\rightarrow \mathbf{X}$ was constructed, and
the pro-homology of $X_{m}$ was calculated. Namely, 
\begin{equation*}
\mathbf{H}_{n}\left( X_{m},\mathbb{Z}\right) =\mathbf{H}_{n}\left( \mathbf{C}%
_{\ast }\left( X_{m},\mathbb{Z}\right) \right) =\left\{ 
\begin{array}{ccc}
\mathbb{Z} & \text{if} & n=0, \\ 
\mathbf{A}\left( \mathbb{Z}\right)  & \text{if} & n=m \\ 
0 & \text{otherwise.} & 
\end{array}%
\right. 
\end{equation*}%
Let us now proceed similarly to \cite{Prasolov-Non-additivity-MR2175366},
Proposition 7.3. Let $G$ be, as in Section \ref{Sec-example-CH}, a countably
generated free abelian group. Assume for simplicity that $m>2$. In the two
spectral sequences from Theorem \ref{Th-Spectral-holimit-pro-complex} (\ref%
{Th-Spectral-holimit-pro-complex-convergence}) for the pro-complexes $%
\mathbf{C}_{\ast }\left( X_{m},\mathbb{Z}\right) $ and $\mathbf{C}_{\ast
}^{b}\left( X_{m},\mathbb{Z}\right) =\mathbf{C}_{\ast }\left( X_{m},\mathbb{Z%
}\right) \otimes _{\mathbb{Z}}G$ one obtains, using \cite%
{Prasolov-Non-additivity-MR2175366}, Proposition 6.1: 
\begin{equation*}
E_{2}^{st}\left( \mathbf{C}_{\ast }\left( X_{m},\mathbb{Z}\right) \right)
=\left\{ 
\begin{array}{ccc}
\mathbb{Z} & \text{if} & s=t=0, \\ 
\underleftarrow{\lim }^{2}~\mathbf{A}\left( \mathbb{Z}\right)  & \text{if} & 
s=2,t=-m, \\ 
0 & \text{otherwise.} & 
\end{array}%
\right. 
\end{equation*}%
\begin{equation*}
E_{2}^{st}\left( \mathbf{C}_{\ast }^{b}\left( X_{m},G\right) \right)
=\left\{ 
\begin{array}{ccc}
\dbigoplus\limits_{0}^{\infty }\mathbb{Z} & \text{if} & s=t=0, \\ 
\underleftarrow{\lim }^{2}~\mathbf{A}^{\omega }\left( \mathbb{Z}\right)  & 
\text{if} & s=2,t=-m, \\ 
0 & \text{otherwise.} & 
\end{array}%
\right. 
\end{equation*}%
Both spectral sequences degenerate, implying $\underleftarrow{\lim }%
^{1}~E_{r}^{st}=0$. It follows from Theorem \ref%
{Th-Spectral-holimit-pro-complex} (\ref%
{Th-Spectral-holimit-pro-complex-convergence}) that%
\begin{equation*}
\overline{H}_{n}^{b}\left( X_{m},\mathbb{Z}\right) =\overline{H}_{n}\left(
X_{m},\mathbb{Z}\right) =\left\{ 
\begin{array}{ccc}
\mathbb{Z} & \text{if} & n=0, \\ 
\underleftarrow{\lim }^{2}~\mathbf{A}\left( \mathbb{Z}\right)  & \text{if} & 
n=m-2, \\ 
0 & \text{otherwise.} & 
\end{array}%
\right. 
\end{equation*}%
\begin{equation*}
\overline{H}_{n}^{b}\left( X_{m},G\right) =\left\{ 
\begin{array}{ccc}
\dbigoplus\limits_{0}^{\infty }\mathbb{Z} & \text{if} & n=0, \\ 
\underleftarrow{\lim }^{2}~\mathbf{A}^{\omega }\left( \mathbb{Z}\right)  & 
\text{if} & n=m-2, \\ 
0 & \text{otherwise.} & 
\end{array}%
\right. 
\end{equation*}%
The pairing%
\begin{equation*}
\overline{H}_{m-2}^{b}\left( X_{m},\mathbb{Z}\right) \otimes _{\mathbb{Z}%
}G=\dbigoplus\limits_{0}^{\infty }\underleftarrow{\lim }^{2}~\mathbf{A}%
\left( \mathbb{Z}\right) \longrightarrow \overline{H}_{m-2}^{b}\left(
X_{m},G\right) =\underleftarrow{\lim }^{2}~\mathbf{A}^{\omega }\left( 
\mathbb{Z}\right) 
\end{equation*}%
is \textbf{not an isomorphism}, due to \cite%
{Prasolov-Non-additivity-MR2175366}, Proposition 6.1 and Corollary 6.5.
Contradiction.
\end{proof}

\bibliographystyle{alpha}
\bibliography{Cosheaves}

\end{document}